\documentclass[11pt,table]{amsart}

\usepackage{epigamath}

\usepackage[english]{babel}

\numberwithin{equation}{section}

\usepackage[shortlabels]{enumitem}
\setlist[enumerate,1]{label={\rm(\arabic*)}, ref={\rm\arabic*}}

\usepackage{amsmath}
\usepackage{amsthm}
\usepackage{amsfonts}
\usepackage{amssymb}
\usepackage{mathtools}
\usepackage{bm} 
\usepackage{tikz}
\usepackage{tikz-cd}
\usepackage{graphicx}
\usepackage{scalerel}
\usepackage{todonotes}
\usepackage{comment}
\usepackage[mathscr]{eucal}
\usepackage{hyperref}
\usepackage[capitalise]{cleveref}
\usepackage{makecell}

\usepackage{array} 
\definecolor{rowgrey}{HTML}{E4E4E4}

\usepackage{longtable}

\usepackage{url}

\theoremstyle{plain}
\newtheorem{theorem}{Theorem}[section]
\newtheorem{proposition}[theorem]{Proposition}
\newtheorem{corollary}[theorem]{Corollary}
\newtheorem{lemma}[theorem]{Lemma}
\newtheorem{conjecture}[theorem]{Conjecture}

\theoremstyle{definition}
\newtheorem{definition}[theorem]{Definition}
\newtheorem{notation}[theorem]{Notation}

\theoremstyle{remark}
\newtheorem{remark}[theorem]{Remark}
\newtheorem{example}[theorem]{Example}
\newtheorem{question}[theorem]{Question}

\makeatletter
\newtheorem*{rep@theorem}{\rep@title}
\newcommand{\newreptheorem}[2]{%
	\newenvironment{rep#1}[1]{%
		\def\rep@title{#2 \ref{##1}}%
		\begin{rep@theorem}}%
		{\end{rep@theorem}}}
\makeatother

\theoremstyle{plain}
\newreptheorem{theorem}{Theorem}
\newreptheorem{corollary}{Corollary}
\newreptheorem{proposition}{Proposition}
\newreptheorem{lemma}{Lemma}

\theoremstyle{definition}
\newreptheorem{definition}{Definition}

\theoremstyle{remark}
\newreptheorem{question}{Question}

\newcommand\restr[2]{{
		\left.\kern-\nulldelimiterspace 
		#1 
		\vphantom{\big|} 
		\right|_{#2} 
}}

\DeclareMathOperator{\im}{im}
\DeclareMathOperator{\rank}{rank}

\DeclareMathOperator{\Hom}{Hom}

\DeclareMathOperator{\Ob}{Ob}

\DeclareMathOperator{\Tot}{Tot}

\DeclareMathOperator{\supp}{supp}

\DeclareMathOperator{\Aut}{Aut}

\DeclareMathOperator{\Char}{char}

\DeclareMathOperator{\Ker}{Ker}
\DeclareMathOperator{\GL}{GL}

\DeclareMathOperator*{\Alb}{Alb}
\DeclareMathOperator*{\alb}{alb}
\DeclareMathOperator*{\Supp}{Supp}

\DeclareMathOperator*{\length}{length}

\let\Re\relax
\DeclareMathOperator{\Re}{\mathrm{Re}} 
\let\Im\relax
\DeclareMathOperator{\Im}{\mathrm{Im}} 

\usepackage{xparse}

\ExplSyntaxOn

\NewDocumentCommand{\definealphabet}{mmmm}
 {
  \int_step_inline:nnn { `#3 } { `#4 }
   {
    \cs_new_protected:cpx { #1 \char_generate:nn { ##1 }{ 11 } }
     {
      \exp_not:N #2 { \char_generate:nn { ##1 } { 11 } }
     }
   }
 }

\ExplSyntaxOff

\definealphabet{b}{\mathbf}{A}{Z}
\definealphabet{c}{\mathcal}{A}{Z}
\definealphabet{mf}{\mathfrak}{A}{Z}
\definealphabet{mf}{\mathfrak}{a}{z}
\definealphabet{mr}{\mathrm}{A}{Z}
\definealphabet{mr}{\mathrm}{a}{z}

\newcommand{\Z}{\ensuremath{\mathbf{Z}}}
\newcommand{\N}{\ensuremath{\mathbf{N}}}
\newcommand{\R}{\ensuremath{\mathbf{R}}}
\newcommand{\C}{\ensuremath{\mathbf{C}}}
\renewcommand{\H}{\ensuremath{\mathbf{H}}}
\newcommand{\Q}{\ensuremath{\mathbf{Q}}}
\renewcommand{\P}{\mathbf{P}}

\newcommand{\TT}{\mathcal{T}}
\newcommand{\FF}{\mathcal{F}}
\newcommand{\PP}{\mathcal{P}}
\newcommand{\QQ}{\mathcal{Q}}
\newcommand{\DD}{\mathcal{D}}
\renewcommand{\AA}{\mathcal{A}}

\newcommand{\CC}{\mathcal{C}}
\newcommand{\EE}{\mathcal{E}}
\newcommand{\OO}{\mathcal{O}}
\newcommand{\LL}{\mathcal{L}}
\newcommand{\HH}{\mathcal{H}}

\newcommand{\UU}{\mathcal{U}}
\newcommand{\ZZ}{\mathcal{Z}}

\newcommand{\ch}{\mathrm{ch}} 
\newcommand{\Stab}{\mathrm{Stab}}
\newcommand*{\Gstab}[1]{\mathrm{Stab}_{#1}^{\mathrm{Geo}}} 
\newcommand{\Stablf}{\mathrm{Stab}_{\mathrm{lf}}}
\newcommand{\Db}{\mathrm{D^b}} 
\newcommand{\DbX}{\mathrm{D^b}(X)} 
\newcommand{\DbY}{\mathrm{D^b}(Y)} 
\newcommand{\DbGX}{\mathrm{D^b_G}(X)} 
\newcommand{\K}{\mathrm{K}} 
\newcommand{\Knum}{\mathrm{K_{num}}} 
\newcommand{\Chow}{\mathrm{Chow}} 
\newcommand{\Chownum}{\mathrm{Chow_{num}}} 
\newcommand{\GLcov}{\widetilde{\mathrm{GL}}^+_2(\R)} 
\newcommand{\Sh}{\mathrm{Sh}} 
\newcommand{\Forg}{\mathrm{Forg}} 
\newcommand{\Inf}{\mathrm{Inf}} 
\newcommand{\pipull}{\pi^\ast} 
\newcommand{\gstar}{g^\ast} 
\newcommand{\lambdanat}{\lambda_{\mathrm{nat}}} 
\newcommand{\phisky}{\phi_{\mathrm{sky}}} 

\newcommand{\MSQuadratic}[2]{\Delta^{C_{#1}}_{#1,#2}}
\newcommand{\chitop}{\chi_{\mathrm{top}}} 

\newcommand{\Pic}{\mathrm{Pic}}
\newcommand{\NS}{\mathrm{NS}}
\newcommand{\Amp}{\mathrm{Amp}}
\newcommand{\Eff}{\mathrm{Eff}}
\newcommand{\Coh}{\mathrm{Coh}}
\newcommand{\Rep}{\mathrm{Rep}}

\newcommand{\ldot}{\mathbin{.}}

\newcommand{\st}{\mid}

\newcommand*\cocolon{%
        \nobreak
        \mskip6mu plus1mu
        \mathpunct{}%
        \nonscript
        \mkern-\thinmuskip
        {:}%
        \mskip2mu
        \relax
}

\makeatletter
\newcommand{\xleftrightarrows}[2][]{\mathrel{
 \raise.40ex\hbox{$
       \ext@arrow 3095\leftarrowfill@{\phantom{#1}}{#2}$}
 \setbox0=\hbox{$\ext@arrow 0359\rightarrowfill@{#1}{\phantom{#2}}$}
 \kern-\wd0 \lower.40ex\box0}}

\newcommand{\xrightleftarrows}[2][]{\mathrel{
 \raise.40ex\hbox{$\ext@arrow 3095\rightarrowfill@{\phantom{#1}}{#2}$}
 \setbox0=\hbox{$\ext@arrow 0359\leftarrowfill@{#1}{\phantom{#2}}$}
 \kern-\wd0 \lower.40ex\box0}}

\makeatother

\newcommand{\lra}{\longrightarrow}
\def\lowsim{\vbox to 0pt{\vss\hbox{$\scriptstyle\sim$}\vskip-1.6pt}}

\newcommand{\longhookrightarrow}{\lhook\joinrel\longrightarrow}
\newcommand{\longtwoheadrightarrow}{\relbar\joinrel\twoheadrightarrow}
\newcommand{\longmapsfrom}{\mathrel{\reflectbox{\ensuremath{\longmapsto}}}}

\EpigaVolumeYear{9}{2025} \EpigaArticleNr{16} \ReceivedOn{August 10, 2023}
\InFinalFormOn{October 19, 2024}
\AcceptedOn{February 7, 2025}

\title{Stability conditions on free abelian quotients}
\titlemark{Stability conditions on free abelian quotients}

\author{Hannah Dell}
\address{Mathematical Institute and Hausdorff Center for Mathematics, University of Bonn, Endenicher Allee 60, 53115 Bonn, Germany}
\email{hdell@math.uni-bonn.de}

\authormark{H.~Dell}

\AbstractInEnglish{We study slope-stable vector bundles and Bridgeland stability conditions on varieties which are a quotient of a smooth projective variety by a finite abelian group $G$ acting freely. We show there is an analytic isomorphism between $G$-invariant geometric stability conditions on the cover and geometric stability conditions on the quotient that are invariant under the residual action of the group $\widehat{G}$ of irreducible representations of $G$. We apply our results to describe a connected component inside the stability manifolds of free abelian quotients when the cover has finite Albanese morphism. This applies to varieties with non-finite Albanese morphism which are free abelian quotients of varieties with finite Albanese morphism, such as Beauville-type and bielliptic surfaces. This gives a partial answer to a question raised by Lie Fu, Chunyi Li, and Xiaolei Zhao: if a variety $X$ has non-finite Albanese morphism, does there always exist a non-geometric stability condition on $X$? We also give counterexamples to a conjecture of Fu--Li--Zhao concerning the Le Potier function, which characterises Chern classes of slope-semistable sheaves. As a result of independent interest, we give a description of the set of geometric stability conditions on an arbitrary surface in terms of a refinement of the Le Potier function. This generalises a result of Fu--Li--Zhao from Picard rank $1$ to arbitrary Picard rank.}



\acknowledgement{The author was supported by the ERC Consolidator Grant WallCrossAG (ID 819864) and the ERC Synergy Grant HyperK (ID 854361).}

\begin{document}



\maketitle

\begin{prelims}

\DisplayAbstractInEnglish

\bigskip

\DisplayKeyWords

\medskip

\DisplayMSCclass

\end{prelims}


\newpage

\setcounter{tocdepth}{1}

\tableofcontents


\section{Introduction}\label{section: introduction}
Stability conditions on triangulated categories were introduced in \cite{bridgelandStabilityConditionsTriangulated2007} by Bridgeland, who was motivated by the study of Dirichlet branes in string theory. In the same paper, Bridgeland showed that the space $\Stab(\cD)$ of stability conditions on a given triangulated category $\cD$ has the structure of a complex manifold. When $X$ is a smooth projective variety over $\bC$, this leads to the fundamental question: how does the geometry of $X$ relate to the geometry of $\Stab(X)\coloneqq\Stab(\Db(X))$?

In this article, we investigate this question in the case of varieties that are free quotients by finite abelian groups, especially quotients of varieties with finite Albanese morphism such as bielliptic and Beauville-type surfaces.

One approach is via group actions on triangulated categories. We sharpen the correspondence between $G$-invariant stability conditions on $\DD$ and stability conditions on the $G$-equivariant category $\DD_G$ introduced by Macr\`i, Mehrotra, and Stellari in \cite[Theorem 1.1]{macriInducingStabilityConditions2009}. This is used to control the set of geometric stability conditions on any free quotient by a finite abelian group.

We also study the Le Potier function introduced by Fu, Li, and Zhao in \cite[Section~3.1]{fuStabilityManifoldsVarieties2022}. We give counterexamples to the conjecture stated in \cite[Section~4]{fuStabilityManifoldsVarieties2022} and explain how a refinement of the Le Potier function controls the set of geometric Bridgeland stability conditions on any surface.

\subsection{Geometric stability conditions and group actions}\label{intro subsection: geometric stability and group actions}

Let $k$ be an algebraically closed field, and let $G$ be a finite abelian group such that $(\Char(k),|G|)=1$. Let~$\DD$ be a $k$-linear idempotent complete triangulated category with an action of $G$ in the sense of \cite{deligneActionGroupeTresses1997}. This induces an action on $\Stab(\DD)$, the space of all numerical Bridgeland stability conditions on $\DD$. Let $\DD_G$ denote the corresponding category of $G$-equivariant objects. There is a residual action by $\widehat{G}=\Hom(G,k^\ast)$ on $\DD_G$ (see Proposition~\ref{defn of G hat action}), and $(\DD_G)_{\widehat{G}}\cong\DD$ by \cite[Theorem 4.2]{elaginEquivariantTriangulatedCategories2015}. \cref{G hat invariant correspondence} describes an analytic isomorphism between $G$-invariant stability conditions on $\DD$ and $\widehat{G}$-invariant stability conditions on $\DD_G$. This builds on \cite[Proposition 2.2.3]{polishchukConstantFamiliesTStructures2007} and \cite[Theorem 1.1]{macriInducingStabilityConditions2009}.

In this paper, we focus on the case where $\DD=\Db(X)$ for $X$ a smooth projective connected variety over~$\C$ with a $G$-action. This induces an action on $\Coh(X)$ (and hence $\Db(X)$) by pullback. In this setting, $(\DbX)_G\cong\DbGX\coloneqq \Db(\Coh_G(X))$, the bounded derived category of $G$-equivariant coherent sheaves on~$X$. If $G$ acts freely on $X$, we call $Y$ a \textit{free abelian quotient} and have $\DbY\cong\DbGX$. Furthermore, let $\pi\colon X\rightarrow Y$ denote the quotient map. Then $\pi_\ast \OO_X$ decomposes into a direct sum of numerically trivial line bundles $\LL_\chi$ according to the irreducible representations $\chi\in\widehat{G}$. The residual action of $\widehat{G}$ on $\Db(Y)$ is given by $-\otimes \LL_\chi$. The functors $\pi^*$ and $\pi_*$ give rise to the isomorphisms in \cref{G hat invariant correspondence}. More precisely, given a $G$-invariant stability condition $\sigma\in\Stab(X)$, there is a stability condition called $(\pi^*)^{-1}\sigma\in\Stab(Y)$ with the property that $E\in\Db(Y)$ is semistable if and only if $\pi^*E$ is $\sigma$-semistable. The construction with $(\pi_*)^{-1}$ is analogous.

A stability condition $\sigma\in \Stab(X)\coloneqq \Stab(\DbX)$ is called \textit{geometric} if all skyscraper sheaves of points~$\OO_x$ are $\sigma$-stable and of the same phase. In all known examples, the stability manifold contains an open set of geometric stability conditions. We prove that geometric stability conditions are preserved under \cref{G hat invariant correspondence}. 

\begin{reptheorem}{geometric G inv corresponds to geometric G hat inv}
	Suppose $G$ is a finite abelian group acting freely on $X$. Let $\pi\colon X\rightarrow Y\coloneqq X/G$ denote the quotient map. Consider the residual action of\, $\widehat{G}$ on $\Db(Y)$. Then the functors $\pi^\ast$ and $\pi_\ast$ induce an analytic isomorphism between $G$-invariant stability conditions on $\DbX$ and $\widehat{G}$-invariant stability conditions on $\DbY$ which preserve geometric stability conditions,
\begin{equation*}
		(\pi^\ast)^{-1}\colon (\Stab(X))^G\xleftrightarrows{\;\cong\;}(\Stab(Y))^{\widehat{G}} \cocolon(\pi_\ast)^{-1}.
	\end{equation*} 
        
\end{reptheorem}

Very little is known about how the geometry of a variety $X$ relates to the geometry of $\Stab(X)$. Recall that every variety $X$ has an algebraic map $\alb_X$, the Albanese morphism, to the Albanese variety $\Alb(X)\coloneqq \Pic^0(\Pic^0(X))$. Every morphism $f\colon X\rightarrow A$ to another abelian variety $A$ factors via $\alb_X$. In \cite[Theorem 1.1]{fuStabilityManifoldsVarieties2022}, the authors showed that if $X$ has finite Albanese morphism, then all stability conditions on $\DbX$ are geometric. In this set-up, we obtain a union of connected components of geometric stability conditions on any free abelian quotient of $X$.

\begin{reptheorem}{Albanese connected component}
	Let $X$ be a variety with finite Albanese morphism. Let $G$ be a finite abelian group acting freely on~$X$, and let $Y=X/G$. Then $\Stab^\ddagger(Y)\coloneqq (\Stab(Y))^{\widehat{G}}\cong \Stab(X)^G$ is a union of connected components in $\Stab(Y)$ consisting only of geometric stability conditions.
\end{reptheorem}

Let $\Gstab{}(X)$ denote the set of all geometric stability conditions. When $X$ is a surface, we have the following stronger result.

\begin{reptheorem}{finite albanese surface quotient has connected component of geos}
	Let $X$ be a surface with finite Albanese morphism. Let $G$ be a finite abelian group acting freely on~$X$. Let $S=X/G$. Then $\Stab^\ddagger(S)=\Gstab{}(S)\cong(\Stab(X))^G$. In particular, $\Stab^\ddagger(S)$ is a connected component of\, $\Stab(S)$.
\end{reptheorem}

We explain in \cref{intro subsection: Le Potier and Geometric Stability} how to describe $\Gstab{}(S)$ explicitly for any surface $S$. Moreover, \cref{finite albanese surface quotient has connected component of geos} applies to the following two classes of minimal surfaces.

\begin{example}[Beauville-type surfaces, $q=0$]\label{eg: Beauville-type surfaces}
	Let $X=C_1\times C_2$, where the $C_i$ are smooth projective curves of genus $g(C_i)\geq 2$. Each curve has finite Albanese morphism, and hence so does $X$. Suppose there is a free action of a finite (not necessarily abelian) group $G$ on $X$ such that $S=X/G$ has $q(S)\coloneqq h^1(S,\OO_S)=0$ and $p_g(S)\coloneqq h^2(S,\OO_S)=0$. Then $\alb_S$ is trivial. This generalises a construction due to Beauville in \cite[Exercise X.4]{beauvilleComplexAlgebraicSurfaces1996}, and we call $S$ a \textit{Beauville-type surface}. These are classified in \cite[Theorem 0.1]{bauerClassificationSurfacesIsogenous2008}. There are 17 families, 5 of which involve an abelian group. In the abelian cases, $G$ is one of the following groups: $(\Z/2\Z)^3$, $(\Z/2\Z)^4$, $(\Z/3\Z)^2$, $(\Z/5\Z)^2$.
\end{example}

\begin{example}[Bielliptic surfaces, $q=1$]\label{eg: bielliptic surfaces}
	Let $S\cong (E\times F)/G$, where $E,F$ are elliptic curves and $G$ is a finite group of translations of $E$ acting on $F$ such that $F/G\cong\P^1$. Then $q(S)=1$ and $\Alb(S)\cong E/G$, so $\alb_S$ is an elliptic fibration. Such surfaces are called \textit{bielliptic} and were first classified in \cite{bagneraSopraSuperficieAlgebriche1907}. There are 7 families; see \cite[List VI.20]{beauvilleComplexAlgebraicSurfaces1996}.
\end{example}

Let $S$ be a Beauville-type or bielliptic surface. As discussed above, $S$ has non-finite Albanese morphism. By \cref{finite albanese surface quotient has connected component of geos}, $\Gstab{}(S)\subset\Stab(S)$ is a connected component. In particular, if $\Stab(S)$ were connected, then the following question would have a negative answer.

\begin{question}[\textit{cf.} {\cite[Question 4.11]{fuStabilityManifoldsVarieties2022}}]\label{Question FLZ}
	Let $X$ be a variety whose Albanese morphism is not finite. Are there always non-geometric stability conditions on $\DbX$?
\end{question}

This is the converse of \cite[Theorem 1.1]{fuStabilityManifoldsVarieties2022}. In all other known examples, the answer to Question~\ref{Question FLZ} is positive (see \cref{survey of stability manifolds}).

\subsection{The Le Potier function}

A fundamental problem in the study of stable sheaves on a smooth projective connected variety $X$ over~$\bC$ is to understand the set of Chern characters of stable sheaves. This can be used to describe $\Gstab{}(X)$ for surfaces (see Theorem~\ref{thm: LP gives precise control over set of geometric stability conditions}) and to control wall-crossing and hence indirectly control Brill--Noether phenomena as in \cite[Theorem 1.1]{bayerWallcrossingImpliesBrillNoether2018} and \cite{feyzbakhshMukaiProgramReconstructing2020a}.

For $X=\P^2$, Dr\'ezet and Le Potier completely characterised the Chern characters of slope-stable sheaves in terms of a function of the slope, $\delta\colon\R\rightarrow \R$; see \cite[Theorem B]{drezetFibresStablesFibres1985}. In \cite[Section~3.1]{fuStabilityManifoldsVarieties2022}, the authors define a Le Potier function $\Phi_{X,H}$ which gives a generalisation of Dr\'ezet and Le Potier's function to any polarised surface $(X,H)$. They use this to control geometric Bridgeland stability conditions with respect to a sublattice of the numerical $\mrK$-group of $X$, $\Knum(X)$, coming from the polarisation.

Let $\NS_\R(X)\coloneqq \NS(X)\otimes\R$, where $\NS(X)$ is the N\'eron--Severi group of $X$, and let $\Amp_\R(X)$ denote the ample cone inside $\NS_\R(X)$. In \cref{section: Le Potier}, we introduce a generalisation of the Le Potier function. We state the version for surfaces below to ease notation.

\begin{repdefinition}{defn: twisted Le Potier}
	Let $X$ be a surface. Let $(H,B)\in \Amp_\R(X)\times \NS_\R(X)$. We define the \textit{Le Potier function twisted by }$B$, $\Phi_{X,H,B}\colon\R\rightarrow\R\cup\{-\infty\}$, by
	\begin{equation*}
		\Phi_{X,H,B}(x)\coloneqq \limsup_{\mu\rightarrow x}\left\{ \frac{\ch_2(F)-B\ldot \ch_1(F)}{H^2\ch_0(F)}  \; : \parbox{15em}{$F\in\Coh(X)$ is $H$-semistable with $\mu_H(F)=\mu$} \right\}.
	\end{equation*}
\end{repdefinition}

The Bogomolov--Gieseker inequality gives an upper bound for $\Phi_{X,H,B}$ (see Lemma~\ref{lem: LP function well defined and bounded}). If $B=0$, this is the same as \cite[Definition 3.1]{fuStabilityManifoldsVarieties2022}, \textit{i.e.}
$\Phi_{X,H,0}=\Phi_{X,H}$, and the upper bound is $\frac{1}{2}x^2$. The function $\Phi_{X,H,B}$ naturally generalises to higher dimensions; see Definition~\ref{defn: Le Potier}. 

The Le Potier function partially determines the non-emptiness of moduli spaces of $H$-semistable sheaves of a fixed Chern character, which in turn controls wall-crossing, along with the birational geometry of these moduli spaces, for example for $\P^2$ (see  \cite[Theorems 0.2 and~0.4]{liBirationalModelsModuli2019}), K3 surfaces (see \cite[Theorem 5.7]{bayerMMPModuliSheaves2014}), and abelian surfaces (see \cite[Theorem 4.4.1]{minamideFourierMukaiTransformsWallcrossing2012}).

The Le Potier function is known for abelian surfaces (see \cite[Corollary 0.2]{mukaiSymplecticStructureModuli1984} and~\cite{yoshiokaModuliSpacesStable2001}), K3 surfaces (see \cite[Chapter 10, Theorem 2.7]{huybrechtsLecturesK3Surfaces2016}),
del Pezzo surfaces of degrees $9 - m$ for $m\leq 6$ (see \cite[Theorem~7.15]{levineBrillNoetherExistenceSemistable2019}), Hirzebruch surfaces (see \cite[Theorem 9.7]{coskunExistenceSemistableSheaves2021}), and for surfaces with finite Albanese morphism (see \cite[Example 2.12(2)]{lahozChernDegreeFunctions2022}).

In this paper, we relate the Le Potier function of $X$ to the Le Potier function of any free (not necessarily abelian) quotient of $X$ by a finite group. We state the version of these results for arbitrary surfaces in the case $B=0$ below.

\begin{repproposition}{Le Potier Functions Agree}
	Let $X$ be a surface, and let $G$ be a finite group acting freely on $X$. Set $S:=X/G$, and let $\pi\colon X\rightarrow S$ denote the quotient map, and let $H_S\in\Amp_\R(S)$. Then $\Phi_{S,H_S}=\Phi_{X,\pipull H_S}$.
\end{repproposition}

Proposition~\ref{Le Potier Functions Agree} gives us a way to compute the Le Potier function for varieties that are finite free quotients of varieties with finite Albanese morphism.

\begin{reptheorem}{thm: LP for quotients of varieties with finite Albanese}
	Let $X$ be a surface with $\alb_X$ finite. Let $G$ be a finite group acting freely on $X$. Set $S:=X/G$, and let $\pi\colon X\rightarrow S$ denote the quotient map. Let $H_X=\alb_X^\ast H = \pi^\ast H_S\in\Amp_\R(X)$ be an ample class pulled back from $\Alb(X)$ and $S$. Then $\Phi_{S,H_S}(x)=\frac{1}{2}x^2$.
\end{reptheorem}

In Example~\ref{eg: ample classes for finite Albanese FAQs}, we explain how to choose appropriate ample classes such that \cref{thm: LP for quotients of varieties with finite Albanese} applies to bielliptic and Beauville-type surfaces. In particular, Beauville-type surfaces provide counterexamples to the following conjecture. 

\begin{conjecture}[\textit{cf.}~{\cite[Section~4.4]{fuStabilityManifoldsVarieties2022}}]\label{conj: FLZ21}
	Let $(S,H)$ be a polarised surface with $q=0$; then the Le Potier function $\Phi_{S,H}$ is not continuous at $0$.
\end{conjecture}

This conjecture was motivated by Question~\ref{Question FLZ} and the expectation that discontinuities of $\Phi_{S,H}$ could be used to show the existence of a wall of the geometric chamber for regular surfaces, as in the cases of rational and K3 surfaces.

\subsection{The Le Potier function and geometric stability conditions}\label{intro subsection: Le Potier and Geometric Stability}

Let $X$ be a surface, and fix $H\in\Amp_\R(X)$. In \cite[Theorem 3.4, Proposition 3.6]{fuStabilityManifoldsVarieties2022}, the authors showed that $\Phi_{X,H}$ gives precise control over $\Gstab{H}(X)$, the set of geometric numerical Bridgeland stability conditions with respect to a sublattice $\Lambda_H\subset\Knum(X)$ (see \cref{remark: Lambda H stability conditions on surfaces}). When $X$ has Picard rank $1$, $\Gstab{H}(X)=\Gstab{}(X)$.

We generalise this to the set of all geometric numerical Bridgeland stability conditions.

\begin{reptheorem}{thm: LP gives precise control over set of geometric stability conditions}
	Let $X$ be a surface. There is a homeomorphism of topological spaces
	\begin{equation*}
		\Gstab{}(X)\cong\C\times \left\{(H,B,\alpha,\beta)\in\Amp_\R(X)\times\NS_\R(X)\times\R^2 : \alpha>\Phi_{X,H,B}(\beta)\right\}.
	\end{equation*}
\end{reptheorem}

In particular, $\Gstab{}(X)$ is connected. We discuss in Remark~\ref{rem: Le Potier controls boundary of the geometric chamber} how Theorem~\ref{thm: LP gives precise control over set of geometric stability conditions} could be used to describe the boundary of $\Gstab{}(X)$. This emphasises how $\Phi_{X,H,B}$ is a crucial tool for understanding the existence of non-geometric stability conditions on surfaces. In particular, if one can compute the Le Potier function, one should be able to tell whether the boundary of the set of geometric stability conditions has a~wall.\looseness=-1

\subsection{Survey: Geometric stability conditions}\label{survey of stability manifolds}

It is still an open question whether geometric stability conditions exist on any smooth projective variety~$X$ over $\bC$. If the answer to \cref{Question FLZ} is negative, then the question also remains as to which geometric properties characterise the existence of non-geometric stability conditions. To give context for these questions and for the results in this paper, we now survey all of the examples where geometric and non-geometric stability conditions are known to exist.

There are the following general results:
\begin{itemize}
	\item Varieties with $\alb_X$ finite: $\Stab(X)=\Gstab{}(X)$; see \cite[Theorem 1.1]{fuStabilityManifoldsVarieties2022}. 
	\item Quotients of varieties with $\alb_X$ finite: Let $Y=X/G$ be a free abelian quotient of $X$, and assume $\alb_X$ is finite. If $G$-invariant stability conditions exist on $X$, then $\Stab^\ddagger(Y)$ is a union of connected components of $\Stab(Y)$ consisting only of geometric stability conditions; see Theorem~\ref{Albanese connected component}.
\end{itemize}

\subsubsection*{Curves} 
The universal cover $\GLcov$ of $\GL_2^+(\bR)$ acts on $\Stab(X)$ (see \cite[Remark 5.14]{macriLecturesBridgelandStability2017}). Up to this action, a geometric stability condition on a curve $C$ corresponds to slope-stability for $\Coh(C)$. Hence $\Gstab{}(C)\cong\GLcov\cong\bC \times \bH$; see \cite[Theorem 2.7]{macriStabilityConditionsCurves2007}.
\begin{itemize}
	\item We have $\Stab(\P^1)\cong \C^2$; see \cite[Theorem 1.1]{okadaStabilityManifold2006}. Okada's construction uses $\Db(\P^1)\cong\Db(\Rep(K_2))$, where $K_2$ is the Kroneker quiver. In particular, $\Stab(\bP^1)\supsetneq \Gstab{}(\bP^1)$.
	\item Let $C$ be a curve of genus $g(C)\geq 1$; then $\Stab(C)=\Gstab{}(C)\cong\GLcov$; see \cite[Theorem 9.1]{bridgelandStabilityConditionsTriangulated2007} and \cite[Theorem 2.7]{macriStabilityConditionsCurves2007}.
\end{itemize}

\subsubsection*{Surfaces}
There is a construction called \textit{tilting} which gives an open set of geometric stability conditions on any surface; see for example \cite{arcaraBridgelandStableModuliSpaces2013,macriLecturesBridgelandStability2017}.

A connected component of $\Stab(X)$ is known in the following cases. This component always contains $\Gstab{}(X)$, but in some cases non-geometric stability conditions exist.
\begin{itemize}
	\item Surfaces with finite Albanese morphism: This connected component is precisely the set of geometric stability conditions which come from tilting. This follows from \cite[Theorem 1.1]{fuStabilityManifoldsVarieties2022} together with \cref{thm: LP gives precise control over set of geometric stability conditions}. 
	\item K3 surfaces: There is a distinguished connected component $\Stab^\dagger(X)$ described by taking the closure and translates under autoequivalences of the open set of geometric stability conditions; see \cite[Theorem 1.1]{bridgelandStabilityConditionsK32008}. Moreover, $\Stab^\dagger$ contains non-geometric stability conditions. By \cite[Theorem 12.1]{bridgelandStabilityConditionsK32008}, at general points of the boundary of $\Gstab{}(X)$, either
	\begin{itemize}
		\item all skyscraper sheaves have a spherical vector bundle as a stable factor, or
		\item $\OO_x$ is strictly semistable if and only if $x\in C$ for $C$ a smooth rational curve in $X$.
	\end{itemize}
	\item $\P^2$: $\Stab(\P^2)$ has a simply connected component, $\Stab^\dagger(\P^2)$, which is a union of geometric and \textit{algebraic} stability conditions. The construction of the latter uses $\Db(\P^2)\cong \Db(\Rep (Q,R))$ for the associated Beilinson quiver $Q$ with relations $R$; see \cite[Theorem 0.1]{liSpaceStabilityConditions2017}.
	\item Enriques surfaces: Suppose $Y$ is an Enriques surface with K3 cover $X$, and let $\Stab^\dagger(X)$ be the connected component of $\Stab(X)$ described above. Then there exists a connected component $\Stab^\dagger(Y)=\Stab^\ddagger(Y)$ which embeds into $\Stab^\dagger(X)$ as a closed submanifold. Moreover, when $Y$ is very general, $\Stab^\dagger(Y)\cong \Stab^\dagger(X)$; see \cite[Theorem 1.2]{macriInducingStabilityConditions2009}. The component $\Stab^\dagger(X)$ has non-geometric stability conditions; hence by Theorem~\ref{geometric G inv corresponds to geometric G hat inv}, so does $\Stab^\dagger(Y)$.
	\item Beauville-type and bielliptic surfaces: Let $S=X/G$. By \cref{finite albanese surface quotient has connected component of geos}, there is a connected component $\Stab^\ddagger(S)=\Gstab{}(X)\cong(\Stab(X))^G$.
\end{itemize}

Non-geometric stability conditions are also known to exist in the following cases:
\begin{itemize}
	\item Rational surfaces: The boundary of $\Gstab{}(X)$ contains points where skyscrapers sheaves are destabilised by exceptional bundles. This follows from the same arguments as in \cite[Section~5]{bayerSpaceStabilityConditions2011}, where the authors use pushforwards of exceptional bundles on $\bP^2$ to destabilise skyscraper sheaves on $\Tot(\OO_{\P^2}(-3))$.
	\item Surfaces which contain a smooth rational curve $C$ with negative self intersection: These have a wall of the geometric chamber such that $\OO_x$ is stable if $x\not\in C$ and strictly semistable if $x\in C$; see \cite[Lemma 7.2]{tramelBridgelandStabilityConditions2022} and \cite[Proposition 5.3]{limCharacteristicClassesStability2022}.
\end{itemize}

\subsubsection*{Threefolds}
Fix $H\in\Amp_\R(X)$. Denote by $\Stab_H(X)$ the space of stability conditions such that the central charge factors via the sublattice $\Lambda_H\subset\Knum(X)$ (see \cref{remark: Lambda H stability conditions on surfaces}). If $\rho(X)=1$, then $\Lambda_H=\Knum(X)$ so $\Stab_H(X)=\Stab(X)$. A strategy for constructing stability conditions in $\Stab_H(X)$ for threefolds was first introduced in \cite[Sections~3 and 4]{bayerBridgelandStabilityConditions2013}. This uses so-called tilt stability conditions to construct geometric stability conditions whenever a conjectural Bogomolov--Gieseker type inequality is satisfied, \textit{i.e.}~a bound on the Chern characters of stable objects.

Geometric stability conditions in $\Stab_H(X)$ exist for some threefolds; see \cite[Theorem 1.4]{bayerSpaceStabilityConditions2016}, \cite[Theorem 1.1]{bernardaraBridgelandStabilityConditions2017}, \cite[Theorem 1.3]{piyaratneStabilityConditionsBogomolovGieseker2017}, \cite[Theorem 1.2]{kosekiStabilityConditionsProduct2018}, \cite[Theorem 1.3]{liStabilityConditionsQuintic2019}, \cite[Theorem 1.2]{kosekiStabilityConditionsThreefolds2020}, \cite[Theorem 1.3]{kosekiStabilityConditionsCalabiYau2022}, \cite[Theorem 1.2]{liuStabilityConditionCalabi2022}.

Below we describe the only threefolds where $\Stab(X)$ is known to be non-empty. These are also the only cases where a connected component of $\Stab_H(X)$ was previously known.

\begin{itemize}
	\item Abelian threefolds: There is a distinguished connected component $\Stab_H^\dagger(X)$ of $\Stab_H(X)$ which is completely described in \cite[Theorem 1.4]{bayerSpaceStabilityConditions2016}. Stability conditions in $\Stab_H^\dagger(X)$ have been shown to satisfy the support property with respect to $\Knum(X)$; in particular, they lie in a connected component $\Stab^\dagger(X)\subset\Stab(X)$; see \cite[Theorem 3.21]{oberdieckDonaldsonThomasInvariants2022}. Abelian threefolds are also a case of \cite[Theorem~1.1]{fuStabilityManifoldsVarieties2022}.
	\item Calabi--Yau threefolds of abelian type: Let $Y$ be a Calabi--Yau threefold admitting an abelian threefold~$X$ as a finite \'etale cover. Then $Y=X/G$, where $G$ is $(\Z/2)^{\oplus 2}$ or $D_4$ (the dihedral group
	of order~8); see \cite[Theorem 0.1]{oguisoCalabiYauThreefolds2001}. There is a distinguished connected component $\mathfrak{P}$ of $\Stab_H(Y)$ induced from $\Stab^\dagger_H(X)$ which contains only geometric stability conditions; see \cite[Corollary~10.3]{bayerSpaceStabilityConditions2016}. By the previous paragraph together with Theorem~\ref{geometric G inv corresponds to geometric G hat inv}, when $G=(\Z/2)^{\oplus 2}$, $\mathfrak{P}$ lies in a connected component of $\Stab^\ddagger(Y)$.
\end{itemize}

The only examples where non-geometric stability conditions are known to exist on threefolds are those with complete exceptional collections. We explain this in greater generality below. 

\subsection*{Exceptional collections} 
There are stability conditions on any triangulated category, with a complete exceptional collection called \textit{algebraic stability conditions}; see \cite[Section~3]{macriStabilityConditionsCurves2007}. On $\P^n$, this has been used to show the existence of geometric stability conditions; see \cite[Proposition 3.5]{muNewModuliSpaces2021} and~\cite[Section~3.3]{petkovicPositivityDeterminantLine2022}. If $X$ is a variety with a complete exceptional collection, non-geometric stability conditions can be constructed from abelian categories that do not contain skyscraper sheaves; see \cite[Section~4.2]{macriExamplesSpacesStability2007}.

We summarise this survey in the table below. Note that the examples in the rightmost column have non-finite Albanese morphism. This gives a positive answer to Question~\ref{Question FLZ} in those cases.
\begin{table}[h]\renewcommand*{\arraystretch}{1.2}
    \centering
    \begin{tabular}{ p{1cm} | p{3.5cm}  | p{8.5cm}}
    \centering $\dim X$ & $\Gstab{}(X)$  &  Known examples with $\Stab(X)\neq\Gstab{}(X)$\\
	\hline
    \rowcolor{rowgrey}
    \centering 1 & $\cong \GLcov$ & $\bP^1$  \\
    \centering 2 & controlled by $\Phi_{X,H,B}$ & $\P^2$, K3 surfaces, rational surfaces, $X\supset C$ rational curve s.t.~$C^2<0$ \\
    \rowcolor{rowgrey}
    \centering 3 & $\neq\emptyset$ for some 3folds & $\P^3$ \\
	\centering $\geq 4$ & $\neq\emptyset$ for $\P^n$ & $\P^n$
    \end{tabular}
\end{table}

\subsection{Related works}\label{intro subsection: related works}
\cref{G hat invariant correspondence} was independently obtained in \cite[Lemma 4.11]{perryModuliSpacesStable2023}. We generalise \cref{G hat invariant correspondence} and the results of \cref{section: geometric stability on FAQs} to non-abelian groups in \cite{dellFusionequivariantStabilityConditions2024}. Theorem~\ref{thm: LP gives precise control over set of geometric stability conditions} was used to prove that $\Gstab{}(X)$ is contractible in \cite[Theorem A]{rekuskiContractibilityGeometricStability2023}.

\subsection{Notation}\label{intro subsection: notation}
\begin{longtable}{ r l }
		$k$ & an algebraically closed field \\
		$\DD$ & a $k$-linear essentially small Ext-finite triangulated category with a Serre functor\\
		$G$ & a finite group such that $(\Char(k),|G|)=1$ \\
		$\DD_G$ & the category of $G$-equivariant objects \\
		$X$ & a smooth connected projective variety over $\bC$\\
		$\Db(X)$ & the bounded derived category of coherent sheaves on $X$ \\
		$\DbGX$ & the bounded derived category of $G$-equivariant coherent sheaves on $X$\\
		$\K(\DD), \K(X)$ & the Grothendieck group of $\DD$, resp.\ $\DbX$\\
		$\Knum(\DD), \Knum(X)$ & the numerical Grothendieck group of $\DD$, resp.\ $\DbX$ \\
		$\Stab(\DD), \Stab(X)$ & the space of numerical Bridgeland stability conditions on $\DD$, resp.\ $\DbX$\\
		$\Gstab{}(X)$ & the space of geometric numerical stability conditions on $\DbX$ \\
		$\ch(E)$ & the Chern character of an object $E \in \Db(X)$ \\
		$\NS(X)$ & $\Pic(X)/\Pic^0(X)$, the N\'eron--Severi group of $X$ \\
		$\NS_\R(X)$ & $\NS(X)\otimes \R$ \\
		$\Amp_\R(X)$ & the ample cone inside $\NS_\R(X)$ \\
		$\Eff_\R(X)$ & the effective cone inside $\NS_\R(X)$ \\
		$\Chow(X)$ & the Chow group of $X$\\
		$\Chownum(X)$ & the numerical Chow group of $X$
\end{longtable}

\subsection*{Acknowledgments}
I would like to thank my advisor Arend Bayer for suggesting the project and for many helpful discussions. I would also like to thank Augustinas Jacovskis and Sebastian Schlegel Mejia for useful conversations. Thanks to Nick Rekuski for pointing out that the Le Potier function for surfaces in \cref{defn: twisted Le Potier} still needs codomain $\R\cup\{-\infty\}$ to be well defined (but this does not affect any proofs), as well as for many other helpful comments and discussions. Thanks also to my thesis examiners and anonymous referees, whose comments and suggestions have helped to improve the exposition and results.

\section{\texorpdfstring{$\boldsymbol{G}$- and $\boldsymbol{\widehat{G}}$-}{\emph{G}-}invariant stability conditions}\label{section: G invariant stability conditions}

We review the notions of equivariant triangulated categories in \cref{abstract inducing subsection: equivariant triangulated categories} and Bridgeland stability conditions in \cref{abstract inducing subsection: Bridgeland stability}. In \cref{abstract inducing subsection: inducing stability conditions}, we use \cite{macriInducingStabilityConditions2009} to describe a correspondence between stability conditions on a triangulated category with an action of a finite abelian group and stability conditions on the corresponding equivariant category.

\subsection{Review: \texorpdfstring{$\boldsymbol{G}$}{G}-equivariant triangulated categories}\label{abstract inducing subsection: equivariant triangulated categories}

Let $\CC$ be a pre-additive category, linear over a ring $k$. Let $G$ be a finite group with $(\Char(k),|G|)=1$. The definition of a group action on a category and the corresponding equivariant category are due to Deligne  \cite{deligneActionGroupeTresses1997}. We will follow the treatment by Elagin from \cite{elaginEquivariantTriangulatedCategories2015} in our presentation below.

\begin{definition}[\textit{cf.} {\cite[Definition 3.1]{elaginEquivariantTriangulatedCategories2015}}]\label{defn: action of a group on a category}
	A \textit{$($right\,$)$ action} of $G$ on $\CC$ is defined by the following data:
	\begin{itemize}
		\item a functor $\phi_g\colon \CC \rightarrow \CC$ for every $g\in G$;
		\item a natural isomorphism $\varepsilon_{g,h}\colon \phi_g\phi_h\rightarrow \phi_{hg}$ for every $g,h\in G$, for which all diagrams
		\begin{center}
			\begin{tikzcd}
				\phi_f\phi_g\phi_h \arrow[r, "{\varepsilon_{g,h}}"] \arrow[d, "{\varepsilon_{f,g}}"] & \phi_f\phi_{hg} \arrow[d, "{\varepsilon_{f,gh}}"] \\
				\phi_{gf}\phi_h \arrow[r, "{\varepsilon_{gf,h}}"]                                    & \phi_{hgf}                                       
			\end{tikzcd}
		\end{center}
		are commutative.
	\end{itemize}
\end{definition}

\begin{remark}
	Note that this definition of a $G$-action is more than a group homomorphism $G\rightarrow \Aut(\cC)$ as there is a fixed isomorphism $\phi_g\phi_h\xrightarrow{\lowsim} \phi_{hg}$ for all $g,h\in G$. This finer notion is required to define the category of $G$-equivariant objects in \cref{defn: G equivariant category}. See \cite[Section 2.2]{beckmannEquivariantDerivedCategories2023} for details on obstructions to lifting a group homomorphism $G\rightarrow \Aut(\cC)$ to a $G$-action.
\end{remark}

\begin{example}[\textit{cf.} {\cite[Example 3.4]{elaginEquivariantTriangulatedCategories2015}}]\label{group action on scheme example}
	Let $G$ be a group acting on a scheme $X$. For each $g\in G$, let $\phi_g\coloneqq g^\ast\colon \Coh(X)\rightarrow \Coh(X)$. Then for all $g,h\in G$, there are canonical isomorphisms 
	\begin{equation*}
	\phi_g\phi_h = g^\ast h^\ast \overset{\lowsim}\lra (hg)^\ast = \phi_{hg}.
	\end{equation*}
	Together, these define an action of $G$ on the category $\Coh(X)$.
\end{example}

\begin{definition}[\textit{cf.} {\cite[Definition 3.5]{elaginEquivariantTriangulatedCategories2015}}]\label{defn: G equivariant category}
	Suppose $G$ acts on a category $\CC$. A $G$\textit{-equivariant object} in $\CC$ is a pair $(F,(\theta_g)_{g\in G})$, where $F\in\Ob\CC$ and $(\theta_g)_{g\in G}$ is a family of isomorphisms
	\begin{equation*}
		\theta_g\colon F\lra \phi_g(F) 
	\end{equation*}
	such that all diagrams
	\begin{center}
		\begin{tikzcd}
			F \arrow[r, "\theta_g"] \arrow[d, "\theta_{hg}"] & \phi_g(F) \arrow[d, "\phi_g(\theta_h)"]             \\
			\phi_{hg}(F)                                     & \phi_g(\phi_h(F)) \arrow[l, "{\varepsilon_{g,h}}"']
		\end{tikzcd}
	\end{center}
	are commutative. We call the family of isomorphisms a \textit{$G$-linearisation}. A \textit{morphism of $G$-equivariant objects} from $(F_1,(\theta_g^1))$ to $(F_2,(\theta_g^2))$ is a morphism $f\colon F_1\rightarrow F_2$ compatible with $\theta_g$, \textit{i.e.}~such that the below diagram commutes for all $g\in G$: 
	\begin{center}
		\begin{tikzcd}
			F_1 \arrow[r, "\theta^1_g"] \arrow[d, "f"] & \phi_g(F_1) \arrow[d, "\phi_g(f)"] \\
			F_2 \arrow[r, "\theta^2_g"]                & \phi_g(F_2)\rlap{.}                      
		\end{tikzcd}
	\end{center}
	The category of $G$-equivariant objects of $\CC$ is denoted by $\CC_G$. 
\end{definition}

\begin{example}\label{G equivariant sheaves example}
	Let $G$ be a group acting on a scheme $X$ with $\phi_g$ and $\varepsilon_{g,h}$ defined as in Example~\ref{group action on scheme example}. The $G$-equivariant objects in $\Coh(X)$ are $G$-equivariant coherent sheaves. Let $\Coh_G(X)\coloneqq (\Coh(X))_G$ and  $\DbGX\coloneqq \Db(\Coh_G(X))$. 
	Suppose $k=\overline{k}$ and $G$ acts freely on a smooth projective variety $X$ over $k$. Let $\pi\colon X\rightarrow X/G$ be the quotient map. Then $\Coh(X/G)\cong\Coh_G(X)$ via $\EE\mapsto (\pi^\ast \EE, (\theta_g))$, where the linearisation is given by $\theta_g\colon \pi^\ast\EE\xrightarrow{\lowsim} (\pi\circ g)^\ast \EE = g^\ast \pi^\ast\EE$. Thus $\DbGX\cong\Db(X/G)$.
\end{example}

There are few examples of group actions on categories which do not arise from a group action on a variety. The following result gives one such example. It will also be key in proving that $\cC$ can be recovered from $\cC_G$ when $G$ is abelian; see \cref{Elagin equivariant equivalence}.

\begin{proposition}[\textit{cf.} {\cite[Section~4, p.~12]{elaginEquivariantTriangulatedCategories2015}}\label{defn of G hat action}] 
  Suppose $G$ is an abelian group acting on $\cC$ and $k$ is algebraically closed. Let $\widehat{G}=\Hom(G,k^\ast)$ be the group of irreducible representations of\, $G$. Then there is an action of\, $\widehat{G}$ on $\CC_G$. For every $\chi\in\widehat{G}$, on objects, $\phi_\chi$ is given by
	\begin{equation*}
		\phi_\chi((F,(\theta_h)))\coloneqq  (F,(\theta_h))\otimes \chi \coloneqq  (F,(\theta_h\cdot \chi(h))),
	\end{equation*}
	and on morphisms, $\phi_\chi$ is the identity. For $\chi$,$\psi\in \widehat{G}$, the equivariant objects $\phi_\chi(\phi_\psi((F),(\theta_h)))$ and $\phi_{\psi\chi}((F,(\theta_h)))$ are the same; hence we set the isomorphisms $\varepsilon_{\chi,\psi}$ to be the identities.
\end{proposition}

There are two natural functors going between $\cC$ and $\cC_G$.

\begin{definition}\label{defn: forg and inf functors}
	Suppose $G$ acts on a category $\CC$. Then we denote by $\Forg_{G}\colon \CC_G \rightarrow\CC$ the \textit{forgetful functor} $\Forg_{G}(F,(\theta_g))= F$. Also let $\Inf_{G}\colon \CC \rightarrow \CC_G$ be the \textit{inflation functor} which is defined by
	\begin{align*}
		\Inf_{G}(F)\coloneqq \left(\bigoplus_{g\in G}\phi_g(F), \left(\xi_g\right)\right),
	\end{align*}
	where
	\begin{equation*}
		\xi_g\colon \bigoplus_{h\in G}\phi_h(F) \overset{\lowsim}\lra \bigoplus_{h\in G} \phi_g\phi_h(F)
	\end{equation*}
	is the collection of isomorphisms
	\begin{equation*}
		\varepsilon_{g,h}^{-1}\colon \phi_{hg}(F)\longrightarrow \phi_g\phi_h(F).
	\end{equation*}
\end{definition}

\begin{lemma}\label{Forg is faithful and adjoint}
	The forgetful functor $\Forg_{G}$ is faithful, and it is left and right adjoint to $\Inf_{G}$.
\end{lemma}

\begin{proof}
	The faithfulness follows immediately from the definition of morphisms between $G$-equivariant objects. For the fact that $\Forg_{G}$ is left and right adjoint to $\Inf_{G}$, see \cite[Lemma 3.8]{elaginEquivariantTriangulatedCategories2015}
\end{proof}

The following proposition builds on a result of Balmer in \cite[Theorem 5.17]{balmerSeparabilityTriangulatedCategories2011a}. We will need it later to construct Bridgeland stability conditions on equivariant categories.

\begin{proposition}[\textit{cf.} {\cite[Corollary 6.10]{elaginEquivariantTriangulatedCategories2015}}]\label{Forg is exact}
	 Suppose $G$ acts on a triangulated category $\cD$ which has a DG-enhancement; then $\cD_G$ is triangulated in such a way that $\Forg_{G}$ is exact.
\end{proposition}

The proof of the following theorem will use comonads. The full definitions can be found in \cite[Section~2]{elaginEquivariantTriangulatedCategories2015}, but for the proof we will only need to know the following: given a comonad $T$ on a category~$\CC$, a \textit{comodule} over $T$ is a pair $(F,h)$, where $F\in\Ob \CC$ and $h\colon F\rightarrow TF$ is a morphism, called the \textit{comonad structure}, satisfying certain conditions (see \cite[Definition 2.5]{elaginEquivariantTriangulatedCategories2015}). All comodules over a given comonad $T$ on $\CC$ form a category, which is denoted by $\CC_T$. There is a forgetful functor $\Forg_{T}\colon \CC_T\rightarrow \CC$ which forgets the comonad structure; \textit{i.e.}~$(F,h)\mapsto F$.

\begin{theorem}[\textit{cf.} {\cite[Theorem 4.2]{elaginEquivariantTriangulatedCategories2015}}]\label{Elagin equivariant equivalence}
	Suppose $k$ is an algebraically closed field, and let $\cC$ be a $k$-linear idempotent complete category. Let $G$ be a finite abelian group with $(\Char(k),|G|)=1$. Suppose $G$ acts on $\cC$. Then
	\begin{equation*}
		(\cC_G)_{\widehat{G}}\cong \cC.
	\end{equation*}
	In particular, under this equivalence $\Forg_{\widehat{G}}\colon (\cC_G)_{\widehat{G}}\rightarrow \cC_G$ is identified with $\Inf_{G}\colon \cC\rightarrow \cC_G$, and their adjoints $\Inf_{\widehat{G}}\colon \cC_G \rightarrow (\cC_G)_{\widehat{G}}$ and $\Forg_{G}\colon \cC_G \rightarrow \cC$ are also identified.
\end{theorem}

\begin{proof}
	Elagin's proof that $(\cC_G)_{\widehat{G}}\cong \cC$ uses the following chain of equivalences:
	\begin{equation*}
		(\cC_G)_{\widehat{G}} \stackrel{(1)}{\cong} (\cC_G)_{T(\Forg_{\widehat{G}},\Inf_{\widehat{G}})} \stackrel{(2)}{\cong} (\cC_G)_\mathcal{R} \stackrel{(3)}{\cong}
		(\cC_G)_{T(\Inf_{G},\Forg_{G})} \stackrel{(4)}{\cong} \cC,
	\end{equation*}
	where $T(\Forg_{\widehat{G}},$ $\Inf_{\widehat{G}}),$ $\mathcal{R},
	T(\Inf_{G},\Forg_{G})$ are comonads on the corresponding categories.
	The equivalences in (1) and (4) are the comparison functors from \cite[Proposition 2.6]{elaginEquivariantTriangulatedCategories2015}. In particular, under (1), $\Forg_{\widehat{G}}\cong \Forg_{T(\Forg_{\widehat{G}},\Inf_{\widehat{G}})}$, and under (4), $\Forg_{T(\Inf_{G},\Forg_{G})}\cong \Inf_{G}$. Moreover, the equivalences (2) and (3) only change the comonad structure; hence the images of the forgetful functors for each category of comodules are the same. Therefore, under the equivalence $(\cC_G)_{\widehat{G}}$, we have $\Forg_{\widehat{G}}\cong \Inf_{G}$. Finally, recall that $\Forg_{\widehat{G}}$ and $\Inf_{\widehat{G}}$ are left and right adjoint, as are $\Forg_{G}$ and $\Inf_{G}$. Hence $\Inf_{\widehat{G}}\cong\Forg_{G}$ follows immediately.
\end{proof}

\subsection{Review: Bridgeland stability conditions}\label{abstract inducing subsection: Bridgeland stability}

For the rest of this section, assume that $\DD$ is a $k$-linear essentially small Ext-finite triangulated category with a Serre functor.

\begin{definition}[\textit{cf.} {\cite[Definition 3.3]{bridgelandStabilityConditionsTriangulated2007}}]\label{slicing}
	A \textit{slicing} $\PP$ on $\DD$ is a collection of full additive subcategories $\PP(\phi)\subset \DD$ for each $\phi\in\R$ such that
	\begin{enumerate}
		\item\label{slicing-1} $\PP(\phi)[1]=\PP(\phi+1)$;
		\item\label{slicing-2} if $F_1\in \PP(\phi_1), F_2\in \PP(\phi_2)$, then $\phi_1>\phi_2$ implies $\Hom_\cD(F_1,F_2)=0$;
		\item\label{slicing-3} every $E\in \DD$ has a Harder--Narasimhan (HN) filtration; \textit{i.e.}~there exist objects $E_1,\ldots, E_m\in \DD$, real numbers $\phi_1>\phi_2>\cdots > \phi_m$, and a collection of distinguished triangles
		\begin{center}
			\begin{tikzcd}[column sep=tiny]
				0=E_0 \arrow[rr] &                        & E_1 \arrow[rr] \arrow[ld] &                        & E_2 \arrow[rr] \arrow[ld] &  & \cdots \arrow[rr] &  & E_{m-1} \arrow[rr] &                        & E_m=E, \arrow[ld] \\
				& A_1 \arrow[lu, dotted] &                           & A_2 \arrow[lu, dotted] &                           &  &                   &  &                    & A_m \arrow[lu, dotted] &                 
			\end{tikzcd}
		\end{center}
		where $A_i\in \cP(\phi_i)$ for $1\leq i\leq m$. We call the $A_i$ the \textit{HN factors} of $E$.
	\end{enumerate}
\end{definition}

\begin{notation}\leavevmode
	\begin{enumerate}
		\item If $0\neq E \in \cP(\phi)$, we call $\phi(E)=\phi$ the \textit{phase} of $E$.
		\item Given an interval $I\subset \bR$, we denote by $\cP(I)$ the smallest additive subcategory of $\cD$ containing all objects $E$ whose HN factors all have phases lying in $I$, \textit{i.e.}~$\phi_i\in I$.
	\end{enumerate}
\end{notation}

\begin{definition}[\textit{cf.} {\cite[Definition 5.1]{bridgelandStabilityConditionsTriangulated2007}}]\label{Bridgeland pre stab}
	A \textit{Bridgeland pre-stability condition} on $\DD$ is a pair $\sigma=(\PP,Z)$ such that
	\begin{enumerate}[label=(\arabic*)]
		\item $\PP$ is a slicing;
		\item $Z\colon \K(\DD) \rightarrow \C$ is a group homomorphism such that if $0\neq E\in \PP(\phi)$ for some $\phi\in \R$, then $Z([E])=m(E)e^{i\pi \phi}$, where $m(E)\in\R_{>0}$. 
	\end{enumerate}
	We call $Z$ the \textit{central charge}.
\end{definition}

\begin{remark}\leavevmode
	\begin{enumerate}
		\item To ease notation, we write $Z(E)\coloneqq  Z([E])$.
		\item The HN filtration in \cref{slicing}\eqref{slicing-3} is unique up to isomorphism. We set $\phi_\sigma^+(E)\coloneqq \phi_1$,  $\phi_\sigma^-(E)\coloneqq \phi_m$, and $m_\sigma(E)\coloneqq \sum_i|Z(A_i)|$.
		\item Each $\cP(\phi)$ is an abelian category; see \cite[Lemma 5.2]{bridgelandStabilityConditionsTriangulated2007}. Non-zero objects of $\PP(\phi)$ are called $\sigma$-\textit{semistable} of phase $\phi$, and non-zero simple objects of $\PP(\phi)$ are called $\sigma$-\textit{stable} of phase $\phi$.
	\end{enumerate}
\end{remark}

\begin{definition}\label{defn: numerical Bridgeland stability condition}\label{defn: support property}
	Let $\Lambda$ be a finite-rank lattice with a surjective group homomorphism $\K(\cD)\stackrel{\lambda}{\twoheadrightarrow} \Lambda$.
	\begin{enumerate}
		\item\label{d:nBsc-1} A Bridgeland pre-stability condition $\sigma=(\PP,Z)$ on $\DD$ satisfies the \textit{support property with respect to} $(\Lambda,\lambda)$~if
		\begin{enumerate}
			\item $Z$ factors via $\Lambda$, \textit{i.e.}~$Z\colon\K(\DD)\stackrel{\lambda}{\twoheadrightarrow} \Lambda \rightarrow \C$, and
			\item there exists a quadratic form $Q$ on $\Lambda\otimes\R$ such that
			\begin{enumerate}[label=(\roman*)]
				\item $\Ker Z\otimes \bR$ is negative definite with respect to $Q$, and
				\item every $\sigma$-semistable object $E\in \DD$ satisfies $Q(\lambda(E))\geq 0$.
			\end{enumerate}
		\end{enumerate}
		\item A Bridgeland pre-stability condition $\sigma$ on $\cD$ that satisfies the support property with respect to $(\Lambda,\lambda)$ is called a \textit{Bridgeland stability condition} (with respect to $(\Lambda,\lambda)$). If $\lambda$ also factors via $\Knum(\DD)$, we call $\sigma$ a \textit{numerical Bridgeland stability condition}.
	\end{enumerate}
\end{definition}

The set of stability conditions with respect to $(\Lambda,\lambda)$ will be denoted by $\Stab_\Lambda(\DD)$. Unless stated otherwise, we will assume that all Bridgeland stability conditions are numerical. The set of numerical stability conditions on $\DD$ will be denoted by $\Stab(\DD)$.

As described in \cite[Proposition 8.1]{bridgelandStabilityConditionsTriangulated2007}, $\Stab_\Lambda(\DD)$ has a natural topology induced by the generalised metric
\begin{equation*}
	d(\sigma_1,\sigma_2) = \sup_{0\neq E\in\DD} \left\{ \left|\phi^-_{\sigma_2}(E) - \phi^-_{\sigma_1}(E)\right|, \left|\phi^+_{\sigma_2}(E) - \phi^+_{\sigma_1}(E)\right|, \left| \log \frac{m_{\sigma_2}(E)}{m_{\sigma_1}(E)}\right| \right\}.
\end{equation*}

\begin{theorem}[\textit{cf.} {\cite[Theorem 1.2]{bridgelandStabilityConditionsTriangulated2007}}]\label{Bridgelandmanifold}
	The space of stability conditions $\Stab_\Lambda(\DD)$ has the natural structure of a complex manifold of dimension $\rank(\Lambda)$. The forgetful map $\cZ$ defines the local homeomorphism
	\begin{align*}
		\ZZ\colon \Stab_\Lambda(\DD) &\longrightarrow \Hom_\Z(\Lambda,\C)\\
		\sigma = (\PP,Z) &\longmapsto Z.
	\end{align*}
\end{theorem}

In other words, the central charge gives a local system of coordinates for the stability manifold.

\begin{remark}\label{remark: locally-finite defn}
	Theorem~\ref{Bridgelandmanifold} was originally stated for locally finite stability conditions: suppose $\sigma=(\PP,Z)$ is a pre-stability condition and there exists an $\varepsilon>0$ such that $\PP(\phi-\varepsilon,\phi+\varepsilon)$ is a quasi-abelian category of finite length for all $\phi\in\R$; then $\sigma$ is called \textit{locally finite}; see \cite[Definition 5.7]{bridgelandStabilityConditionsTriangulated2007}. The support property implies local finiteness; see \cite[Appendix A]{bayerSpaceStabilityConditions2016} for details. Denote by $\Stablf(\DD)$ the space of all locally finite stability conditions on $\DD$.
\end{remark}

\begin{remark}[\textit{cf.} {\cite[Remark 5.14]{macriLecturesBridgelandStability2017}}]\label{rem: GLcov action on Stab}
	There is a right action on $\Stab(\DD)$ by the universal cover $\GLcov$ of $\GL^+_2(\R)$; see \cite[Remark 5.14]{macriLecturesBridgelandStability2017} for details. If we consider $\C^\ast$ as a subgroup of $\GL_2^+(\R)$, then this induces an action of $\widetilde{\C^\ast}=\C$ on $\Stab(\DD)$.
\end{remark}

There is an equivalent characterisation of Bridgeland stability conditions, which uses the notion of a $t$\nobreakdash-structure on a triangulated category. The theory of $t$-structures was first introduced in \cite[Section~1.3]{beilinsonFaisceauxPervers1982}. We first need the following definitions.

\begin{definition}\label{defn: heart}
	A \textit{heart of a bounded} $t$\textit{-structure} in $\DD$ is a full additive subcategory $\AA$ such that
	\begin{enumerate}[label=(\arabic*)]
		\item if $k_1>k_2$, then $\Hom_\cD(\AA[k_1],\AA[k_2])=0$;
		\item for any object $E$ in $\DD$, there are integers $k_1>k_2>\cdots > k_n$ and a sequence of exact triangles
		\begin{center}
			\begin{tikzcd}[column sep=tiny]
				0=E_0 \arrow[rr] &                        & E_1 \arrow[rr] \arrow[ld] &                        & E_2 \arrow[rr] \arrow[ld] &  & \cdots \arrow[rr] &  & E_{n-1} \arrow[rr] &                        & E_n=E \arrow[ld] \\
				& A_1 \arrow[lu, dotted] &                           & A_2 \arrow[lu, dotted] &                           &  &                   &  &                    & A_n \arrow[lu, dotted] &                 
			\end{tikzcd}
		\end{center}
		such that $A_i\in \AA[k_i]$ for $1\leq i \leq n$.
	\end{enumerate} 
\end{definition}

\begin{definition}[\textit{cf.} {\cite[Definitions 2.1 and~2.2]{bridgelandStabilityConditionsTriangulated2007}}] \label{stabilityfunction}
	Let $\AA$ be an abelian category. A \textit{stability function} for $\AA$ is a group homomorphism $Z\colon K(\AA)\rightarrow \C$ such that for every non-zero object $E$ of $\AA$, \begin{equation*}
		Z([E])\in\H\coloneqq \{m\cdot e^{i\pi\phi} \mid m\in\R_{>0}, \phi\in(0,1]\}\subset \C.
	\end{equation*}
	For every non-zero object $E$, we define the \textit{phase} by $\phi(E)=\frac{1}{\pi}\arg (Z([E]))\in(0,1]$. We say an object $E$ is \textit{Z-stable} (resp.\ \textit{Z-semistable}) if $E\neq0$ and for every proper non-zero subobject $A$, we have $\phi(A)<\phi(E)$ (resp.\ $\phi(A)\leq \phi(E)$).
\end{definition}

\begin{definition}[\textit{cf.} {\cite[Definition 2.3]{bridgelandStabilityConditionsTriangulated2007}}]\label{HN property}
	Let $\AA$ be an abelian category, and let $Z\colon K(\AA)\rightarrow \C$ be a stability function on $\AA$. A \textit{Harder--Narasimhan} (\textit{HN}$\mkern1mu$) \textit{filtration} of a non-zero object $E$ of $\AA$ is a finite chain of subobjects in $\cA$,
	\begin{equation*}
		0=E_0\subset E_1\subset \cdots E_{n-1} \subset E_n=E,
	\end{equation*}
	such that each factor $F_i=E_i/E_{i-1}$ (called a \textit{Harder--Narasimhan factor}) is a $Z$-semistable object of $\AA$ and $\phi(F_1)>\phi(F_2)>\cdots > \phi(F_n)$. Moreover, we say that $Z$ has the \textit{Harder--Narasimhan property} if every non-zero object of $\AA$ has a Harder--Narasimhan filtration. 
\end{definition}

\begin{proposition}[\textit{cf.} {\cite[Proposition 5.3]{bridgelandStabilityConditionsTriangulated2007}}]\label{Bridgelandheart}
	To give a Bridgeland pre-stability condition $(\PP,Z)$ on $\DD$ is equivalent to giving a pair $(Z_\AA,\AA)$, where $\AA$ is the heart of a bounded $t$-structure on $\cD$ and $Z_\AA$ is a stability function for $\AA$ which has the Harder--Narasimhan property.
	
	Moreover, $(\PP,Z)$ is a numerical Bridgeland stability condition if and only if $Z_\AA$ factors via $\Knum(\DD)$ and satisfies the support property $($Definition~{\rm\ref{defn: numerical Bridgeland stability condition}}\eqref{d:nBsc-1}$)$ for $Z_\AA$-semistable objects.
\end{proposition}

\subsection{Inducing stability conditions}\label{abstract inducing subsection: inducing stability conditions}

Suppose a finite group $G$ with $(\Char(k),|G|)=1$ acts on $\cD$ by exact autoequivalences $\Phi_g$. This induces an action on the stability manifold via $\Phi_g\cdot(\PP,Z)=(\Phi_g(\PP),Z\circ (\Phi_g)^{-1}_\ast)$, where $(\Phi_g)_\ast\colon \K(\DD) \rightarrow \K(\DD)$ is the natural morphism induced by $\Phi_g$. We say that a stability condition $\sigma$ is $G$\textit{-invariant} if $\Phi_g\cdot \sigma = \sigma$. Write $(\Stablf(\DD))^G$ for the space of all $G$-invariant locally finite stability conditions.

Let $\sigma\in(\Stablf(\DD))^G$. By Lemma~\ref{Forg is faithful and adjoint} and Proposition~\ref{Forg is exact}, $\Forg_{G}\colon\DD_G\rightarrow \DD$ is exact and faithful. This means we can apply the construction from \cite[Section~2.2]{macriInducingStabilityConditions2009}, which induces a (locally finite) stability condition on $\cD_G$ as follows.

Define $\Forg_{G}^{-1}(\sigma)\coloneqq \sigma_G=(\PP_{\sigma_G},Z_{\sigma_G})$, where
\begin{align*}
	\PP_{\sigma_G}(\phi) &\coloneqq \left\{\EE\in\DD_G : \Forg_{G}(\EE)\in\PP_{\sigma}(\phi)\right\},\\
	Z_{\sigma_G} &\coloneqq  Z_{\sigma} \circ \left(\Forg_{G}\right)_\ast.
\end{align*}
Here $(\Forg_{G})_\ast\colon \K(\DD_G) \rightarrow \K(\DD)$ is the natural morphism induced by $\Forg_{G}$. 

\begin{proposition}[\textit{cf.} {\cite[Theorem 2.14]{macriInducingStabilityConditions2009}}]\label{Forg_G induces a stability condition}
	Suppose $G$ acts on $\cD$ and $\sigma=(\PP,Z)\in(\Stablf(\DD))^G$. Then $\Forg_{G}^{-1}(\sigma)\in\Stablf(\DD_G)$.
\end{proposition}

\begin{proof}
	By \cref{Forg is exact} and our assumptions on $\DD$, it follows that $\DD_G$ is a triangulated category and that the assumptions stated before \cite[Theorem 2.14]{macriInducingStabilityConditions2009} are satisfied.

	Suppose $\EE\in\PP(\phi)$. Then $\Forg_{G}(\Inf_{G}(\EE))=\bigoplus_{g\in G}\Phi_g(\EE)$. Since $\sigma$ is $G$-invariant, $\Phi_g(\EE)\in\PP_\sigma(\phi)$ for all $g\in G$. Moreover, $\PP_{\sigma}(\phi)$ is extension closed; hence $\bigoplus_{g\in G}\Phi_g(\EE) \in\PP_{\sigma}(\phi)$. The result then follows from \cite[Theorem 2.14]{macriInducingStabilityConditions2009}.
\end{proof}

\begin{lemma}\label{Induced stability condition is G hat invariant}
  	Suppose $G$ is abelian and acts on $\cD$. Consider the action of\, $\widehat{G}$ on $\DD_G$ by tensoring as in Proposition~{\rm\ref{defn of G hat action}}. Then $\Forg_{G}^{-1}(\sigma)$ is $\widehat{G}$-invariant.
\end{lemma}

\begin{proof}
	First note that, for every class $[\EE]=[(E,(\theta_g))]\in\Knum(\DD_G)$, we have $(\Forg_{G})_\ast([(E,(\theta_g))])=[E]$. Hence $Z_{\sigma_G}([\EE])=Z_\sigma \circ (\Forg_{G})_\ast([(E,(\theta_g))])=Z_\sigma([E])$, where $[E]\in\Knum(\DD)$. Moreover, from the definition of $\PP_{\sigma_{G}}$, we have
	\begin{align*}
		\PP_{\sigma_{G}}(\phi)&= \left\{\EE\in\DD_G : \Forg_{G}(\EE)\in\PP_{\sigma}(\phi)\right\}\\
		&= \left\{\left(E,\left(\theta_g\right)\right)\in\DD_G : E\in\PP_{\sigma}(\phi)\right\}.
	\end{align*}
	In particular, since the action of $\widehat{G}$ on $(E,(\theta_g))\in\DD_G$ does not change $E$, it follows that the central charge $Z_{\sigma_{G}}$ and slicing $\PP_{\sigma_{G}}$ are $\widehat{G}$-invariant, and hence $\sigma_G\in(\Stablf(\DD_G))^{\widehat{G}}$.
\end{proof}

\begin{proposition}[\textit{cf.} {\cite[Proposition 2.17]{macriInducingStabilityConditions2009}}]\label{Forg is continuous and image is closed}
		Under the hypotheses of Proposition~\ref{Forg_G induces a stability condition}, the morphism $\Forg_{G}^{-1}\colon (\Stablf(\DD))^G\rightarrow (\Stablf(\DD_G))^{\widehat{G}}$ is continuous, and the image of $\Forg_{G}^{-1}$ is a closed embedded sub\-ma\-ni\-fold.
\end{proposition}

\begin{proof}
	The proof of \cite[Proposition 2.17]{macriInducingStabilityConditions2009} is for the action of a finite group $G$ on $\DbX$, induced by the action of $G$ on $X$, a variety over $\C$ (\textit{i.e.}~$\Phi_g=g^\ast)$. The result follows in our setting by replacing this with the action of exact autoequivalences $\Phi_g$ on $\DD$ in the proof.
\end{proof}

In the case where $G$ is abelian, we have the following description of the image of $\Forg_{G}^{-1}$.

\begin{theorem}\label{G hat invariant correspondence}
	Suppose $k$ is an algebraically closed field. Let $\DD$ be a $k$-linear essentially small idempotent complete Ext-finite triangulated category with a Serre functor and a DG-enhancement. Let $G$ be a finite abelian group such that $(\Char(k),|G|)=1$. Suppose $G$ acts on $\DD$ by exact autoequivalences $\Phi_g$ for every $g\in G$, and consider the action of\, $\widehat{G}$ on $\DD_G$ as in Proposition~{\rm\ref{defn of G hat action}}. Then the functors $\Forg_G$ and $\Inf_G$ induce an analytic isomorphism between $G$-invariant stability conditions on $\DD$ and $\widehat{G}$-invariant stability conditions on $\DD_G$,
	\begin{equation*}
		\Forg_{G}^{-1}\colon (\Stablf(\DD))^G\xrightleftarrows{\;\cong\;} (\Stablf(\DD_G))^{\widehat{G}} \cocolon\Forg_{\widehat{G}}^{-1}.
	\end{equation*}
	More precisely, the compositions $\Forg_{\widehat{G}}^{-1}\circ \Forg_{G}^{-1}$ and $\Forg_{G}^{-1}\circ \Forg_{\widehat{G}}^{-1}$ fix slicings and rescale central charges by~$|G|$.
\end{theorem}

\begin{proof}
	Let $\sigma\in(\Stablf(\DD))^G$. Therefore, by \cref{Forg_G induces a stability condition} and \cref{Induced stability condition is G hat invariant}, $\sigma_G\coloneqq \Forg_{G}^{-1}(\sigma)$ is in $(\Stablf(\DD_G))^{\widehat{G}}$. We now apply \cref{Forg_G induces a stability condition} again but with $\Forg_{\widehat{G}}$. In particular, let $\sigma_{\widehat{G}}\coloneqq \Forg_{\widehat{G}}^{-1}(\sigma_G)$, where
	\begin{align*}
		\PP_{\sigma_{\widehat{G}}}(\phi) &=\left\{\EE\in(\DD_G)_{\widehat{G}} : \Forg_{\widehat{G}}(\EE)\in \PP_{\sigma_G}(\phi)\right\}\\
		&=\left\{\EE\in(\DD_G)_{\widehat{G}} : \Forg_{G}(\Forg_{\widehat{G}}(\EE))\in \PP_{\sigma}(\phi)\right\}.
	\end{align*}
	By Proposition~\ref{Forg_G induces a stability condition}, $\Forg_{\widehat{G}}^{-1}(\sigma_G)\in\Stablf((\DD_G)_{\widehat{G}})$. To complete the proof, we need to show that, under the equivalence $(\DD_{G})_{\widehat{G}}\cong \DD$, we have $\sigma_{\widehat{G}}=\sigma$ up to rescaling the central charge by $|G|$. From Theorem~\ref{Elagin equivariant equivalence} we know that $\Forg_{\widehat{G}}\cong\Inf_{G}$ under this equivalence. Hence we can apply the same argument as in the proof of \cite[Proposition 2.17]{macriInducingStabilityConditions2009}. In particular,
	\begin{align*}
		\PP_{\sigma_{\widehat{G}}}(\phi) &=\left\{\EE\in\DD : \Forg_{G}(\Inf_{G}(\EE))\in \PP_{\sigma}(\phi)\right\}\\
		&=\left\{\EE\in\DD : \bigoplus_{g\in G} \Phi_g (\EE) \in\PP_{\sigma}(\phi)\right\}.
	\end{align*}
	Suppose $\EE\in\PP_{\sigma_{\widehat{G}}}(\phi)$. Since $\cP(\phi)$ is closed under direct summands, $\Phi_g(\EE)\in\PP_{\sigma}(\phi)$ for all $g\in G$. Thus $\EE\in \PP_{\sigma}(\phi)$. Now suppose $\EE\in\PP_{\sigma}(\phi)$; then by the proof of Proposition~\ref{Forg_G induces a stability condition}, it follows that we have $\Forg_{G}(\Inf_{G}(\EE))=\oplus_{g\in G} \phi_g(\cE)\in\PP_{\sigma}(\phi)$. Therefore, $\EE\in\PP_{\sigma_{\widehat{G}}}(\phi)$. In particular, $\PP_{\sigma_{\widehat{G}}}=\PP_{\sigma}$. Now let $[\EE]\in\Knum(\DD)\otimes \C$, and consider the central charge
	\begin{equation*}
		Z_{\sigma_{\widehat{G}}}([\EE]) = Z_\sigma \circ (\Forg_{G})_\ast\circ(\Inf_{G})_\ast([\EE])=Z_\sigma\left(\sum_{g\in G} ([\Phi_g(\EE)])\right).
	\end{equation*}
	The central charge
        $Z_\sigma$ is $G$-invariant; hence $Z_\sigma([\EE])=(\Phi_g)_\ast Z_{\sigma}([\EE])=Z_\sigma([\Phi_g(\EE)])$ for all $g\in G$. Finally, since $Z_\sigma$ is a homomorphism, it follows that $Z_{\sigma_{\widehat{G}}}([\EE])=|G|\cdot Z_\sigma ([\EE])$.
	
	Note that if we start instead with a stability condition $\sigma_G\in(\Stablf(\DD_G))^{\widehat{G}}$, then by a symmetric argument it follows that $\sigma_G = \Forg_{G}^{-1}\circ \Forg_{\widehat{G}}^{-1}(\sigma_G)$, up to rescaling the central charge by $|\widehat{G}|=|G|$. Therefore, $\Forg_{G}^{-1}$ and $\Forg_{\widehat{G}}^{-1}$ are homeomorphisms since they are continuous by \cref{Forg is continuous and image is closed} and rescaling the central charge is itself a homeomorphism. In fact, rescaling the central charge by $|G|$ is a linear isomorphism on $\Hom_\bZ(\Knum(\cD),\bC)$ and $\Hom_\bZ(\Knum(\cD_G),\bC)$. Hence $\Forg_G^{-1}$ and $\Forg_{\widehat{G}}^{-1}$ are analytic isomorphisms since they are isomorphisms on the level of tangent spaces; \textit{i.e.}
	\begin{align*}\pushQED{\qed}
		(\Hom_\bZ(\Knum(\cD),\bC))^G &\xrightleftarrows{\;\cong\;} (\Hom_\bZ(\Knum(\cD_G),\bC))^{\widehat{G}}\\
		Z &\longmapsto Z \circ \Forg_{\widehat{G}} \\
		Z' \circ \Forg_G &\longmapsfrom Z'.\qedhere \popQED
	\end{align*}
\renewcommand{\qed}{}    
\end{proof}

\begin{remark}
	 If $\DD=\DbX$, where $X$ is a scheme, and if the action of $G$ on $\DD$ is induced by an action of~$G$ on $X$, \textit{i.e.}~$\Phi_g=g^\ast$, then the analytic isomorphism above gives the bijection in the abelian case of \cite[Proposition 2.2.3]{polishchukConstantFamiliesTStructures2007}.
\end{remark}

\begin{remark}\label{remark: lf to support property}
	As in \cite[Theorem 10.1]{bayerSpaceStabilityConditions2016}, \cref{G hat invariant correspondence} also goes through with the support property. In particular, a stability condition $\sigma\in(\Stablf(\DD))^G$ satisfies the support property with respect to $(\Lambda,\lambda)$ if and only if the induced stability condition $\sigma_G\in(\Stablf(\DD_G))^{\widehat{G}}$ satisfies the support property with respect to $(\Lambda, \lambda\circ(\Forg_{G})_*)$.
\end{remark}

\section{Geometric stability conditions on abelian quotients}\label{section: geometric stability on FAQs}

We apply the methods of \cref{abstract inducing subsection: inducing stability conditions} to describe geometric stability conditions on free abelian quotients. In particular, we show that geometric stability conditions are preserved under the analytic isomorphism in \cref{G hat invariant correspondence}, and we use this to describe a union of connected components of geometric stability conditions on free abelian quotients of varieties with finite Albanese morphism. In the case of surfaces, we obtain a stronger result using a description of the set of geometric stability conditions from \cref{section: geometric stability and Le Potier}.

\subsection{Inducing geometric stability conditions}\label{geometric subsection: inducing geometric stability conditions}

Let $X$ be a smooth projective connected variety over $\C$. Let $G$ be a finite group acting freely on $X$. Let $Y=X/G$, and denote by $\pi\colon X \rightarrow Y$ the quotient map. Let $\DbGX$ denote the derived category of $G$-equivariant coherent sheaves on $X$ as in Example~\ref{G equivariant sheaves example}.

Recall that $\DbY\cong \DbGX$, where the equivalence is given by  
\begin{align*}
	\Psi\colon \DbY &\longrightarrow \DbGX \\
	\EE &\longmapsto (\pi^\ast(\EE),\lambdanat)
\end{align*}
and $\lambdanat=\{\lambda_g\}_{g\in G}$ is the $G$-linearisation given by
\begin{equation*}
	\lambda_g\colon \pi^\ast \EE \overset{\lowsim}\lra g^\ast \pi^\ast \EE = (\pi\circ g)^\ast \EE \cong \pi^\ast\EE.
\end{equation*}

Now assume $G$ is abelian. By Theorem~\ref{Elagin equivariant equivalence}, there is an equivalence $\Omega\colon\Db(X)\xrightarrow{\lowsim}(\DbGX)_{\widehat{G}}$. This fits into the following diagram of functors: 
\begin{equation}\label{key functors diagram}
	\begin{tikzcd}
		\DbY \arrow{rr}{\Psi}[swap]{\sim} \arrow[dd, "\pi^\ast", shift left]                                     &  & \DbGX \arrow[dd, "\Inf_{\widehat{G}}", shift left] \arrow[lldd, "\Forg_{G}", shift left] \\
		&  &                                                                                        \\
		\DbX \arrow{rr}{\Omega}[swap]{\sim} \arrow[uu, "\pi_\ast", shift left] \arrow[rruu, "\Inf_{G}", shift left] &  & (\DbGX)_{\widehat{G}}\rlap{,} \arrow[uu, "\Forg_{\widehat{G}}", shift left]                    
	\end{tikzcd}
\end{equation}
where
\begin{equation*}
	\pi^*\stackrel{\Psi}{\cong}\Forg_{G}\stackrel{\Omega}{\cong} \Inf_{\widehat{G}},\quad \pi_*\stackrel{\Psi}{\cong} \Inf_{G}\stackrel{\Omega}{\cong} \Forg_{\widehat{G}},\quad \pi_*\circ\pi^*\stackrel{\Psi}{\cong}\Inf_{G}\circ\Forg_{G},\quad \Forg_{G}\circ\Inf_{G}\stackrel{\Omega}{\cong}\Inf_{\widehat{G}}\circ\Forg_{\widehat{G}}.
\end{equation*}

The residual action of $\widehat{G}$ on $\DbY$ is given by tensoring with numerically trivial line bundles $\LL_\chi$ for each $\chi\in\widehat{G}$.

\begin{definition}\label{Geometric Stability Def}
	A Bridgeland stability condition $\sigma$ on $\DbX$ is called \textit{geometric} if for every point $x\in X$, the skyscraper sheaf $\OO_x$ is $\sigma$-stable.
\end{definition}

\begin{proposition}[\textit{cf.} {\cite[Proposition 2.9]{fuStabilityManifoldsVarieties2022}}]\label{prop: numerical geometric implies all skyscraper sheaves have the same phase}
	Let $\sigma$ be a geometric numerical stability condition on $\DbX$. Then all skyscraper sheaves are of the same phase.
\end{proposition}

In this context, the isomorphism from \cref{G hat invariant correspondence} preserves geometric stability.

\begin{theorem}\label{geometric G inv corresponds to geometric G hat inv}
	Suppose $G$ is a finite abelian group acting freely on $X$. Let $\pi\colon X\rightarrow Y\coloneqq X/G$ denote the quotient map. Consider the action of\, $\widehat{G}$ on $\DbGX\cong\Db(Y)$ as in Proposition~{\rm\ref{defn of G hat action}}. Then the functors $\pi^\ast$ and $\pi_\ast$ induce an analytic isomorphism between $G$-invariant stability conditions on $\DbX$ and $\widehat{G}$-invariant stability conditions on $\DbY$ which preserve geometric stability conditions: 
	\begin{equation*}
		(\pi^\ast)^{-1}\colon (\Stab(X))^G\xrightleftarrows{\;\cong\;} (\Stab(Y))^{\widehat{G}} \cocolon(\pi_\ast)^{-1}.
	\end{equation*}
	The compositions $(\pi_\ast)^{-1}\circ (\pipull)^{-1}$ and $(\pipull)^{-1}\circ (\pi_\ast)^{-1}$ fix slicings and rescale central charges by $|G|$.
	
	In particular, suppose $\sigma=(\PP_\sigma, Z_\sigma)\in(\Stab(X))^G$ satisfies the support property with respect to $(\Lambda,\lambda)$. Then $(\pipull)^{-1}(\sigma)=:\sigma_Y=(\PP_{\sigma_{Y}},Z_{\sigma_{Y}})\in(\Stab(Y))^{\widehat{G}}$ is defined by
	\begin{align*}
		\PP_{\sigma_{Y}}(\phi) &=\left\{\EE\in\DbY : \pi^\ast(\EE)\in\PP_\sigma(\phi)\right\},\\
		Z_{\sigma_{Y}} &= Z_\sigma \circ \pi^\ast,
	\end{align*}
	where $\pi^\ast$ is the natural induced map on $\K(\DbY)$ and $\sigma_Y$ satisfies the support property with respect to $(\Lambda,\lambda\circ\pipull)$. 
\end{theorem}

\begin{proof}
  First note that $\pi_\ast \circ \pi^\ast\colon \Knum(Y)\rightarrow \Knum(Y)$ is just multiplication by $|G|$ because it sends $[E]$ to $\big[E\otimes \bigoplus_{\chi\in\widehat{G}} \LL_\chi\big]$. Therefore, $\pi^\ast\colon \Knum(Y)\rightarrow \Knum(X)$ is injective.

	Together with \cref{G hat invariant correspondence} and Remark~\ref{remark: lf to support property}, the above implies that $(\pi^\ast)^{-1}$ and $(\pi_\ast)^{-1}$ give an analytic isomorphism between numerical Bridgeland stability conditions as described above. It remains to show that $\sigma\in(\Stab(X))^G$ is geometric if and only if $\sigma_Y=(\pi^\ast)^{-1}(\sigma)$ is.

\begin{enumerate}[wide,label={\it Step}~\arabic*.,ref=\arabic*]
\item\label{T33-step1}	Suppose $\sigma=(\PP_{\sigma},Z_{\sigma})\in(\Stab(X))^G$ is geometric. Let $y\in Y$. This corresponds to the orbit $Gx$ for some $x\in X$ (so $x$ is unique up to the action of $G$). We need to show $\OO_y$ is $\sigma_Y$-stable. Recall 
	\begin{equation*}
		\PP_{\sigma_{Y}}(\phi) =\left\{\EE\in\DbY : \pi^\ast(\EE)\in\PP_\sigma(\phi)\right\}
	\end{equation*}
	for every $\phi\in\R$. Now consider
	\begin{equation*}
		\pi^\ast\OO_y=\bigoplus_{g\in G} \OO_{g^{-1} x} \in\DbX.
	\end{equation*}
	By our assumption on $\sigma$ and Proposition~\ref{prop: numerical geometric implies all skyscraper sheaves have the same phase}, all skyscraper sheaves of points of $X$ are $\sigma$-stable and of the same phase, which we denote by $\phisky$. In particular, $\OO_{g^{-1}x}\in\PP_{\sigma}(\phisky)$ for all $g\in G$. Moreover, $\PP_{\sigma}(\phisky)$ is extension closed; hence $\bigoplus_{g\in G} \OO_{g^{-1}x}\in\PP_{\sigma}(\phisky)$, and thus $\OO_y\in\PP_{\sigma_Y}(\phisky)$.
      
	Now suppose that $\OO_y$ is strictly semistable; then there exist $\EE,\FF\in\PP_{\sigma_Y}(\phisky)$ such that
	\begin{equation*}
		\EE \longhookrightarrow \OO_y \longtwoheadrightarrow \FF
	\end{equation*}
	is non-trivial, \textit{i.e.}~$\EE$ is not isomorphic to $0$ or $\OO_y$. By the definition of $\PP_{\sigma_Y}(\phisky)$, the pullbacks $\pi^\ast(\EE)$ and $\pi^\ast(\FF)$ are objects in $\PP_{\sigma}(\phisky)$. Hence we have the following exact sequence in $\PP_{\sigma}(\phisky)$: 
	\begin{equation*}
		\pi^\ast(\EE) \longhookrightarrow \pi^\ast(\OO_y)=\bigoplus_{g\in G} \OO_{g^{-1}x}\longtwoheadrightarrow \pi^\ast(\FF). 
	\end{equation*}
	Since $\pi^\ast(\EE)$ is a subobject of $\pi^\ast(\OO_y)$, we must have $\pi^\ast(\EE)=\bigoplus_{a\in A} \OO_{a^{-1}x}$, where $A\subset G$ is a subset. Hence,
	\begin{equation*}
		\supp (\pi^\ast(\EE)) = \left\{a^{-1}x : a\in A\right\} \subset \left\{g^{-1}x : g\in G\right\} = \supp\left(\pi^\ast\left(\OO_y\right)\right).
	\end{equation*}
	Note that $\pi^\ast(\EE)$ is a $G$-invariant sheaf. But $\supp(\pi^\ast(\EE))$ is $G$-invariant if and only if $A=\emptyset$ or $A=G$. Hence $\EE=0$ or $\EE=\OO_y$, and we have a contradiction.
	
\item\label{T33-step2}	Suppose that $\sigma_Y=(\PP_{\sigma_Y},Z_{\sigma_Y})\in(\Stab(Y))^{\widehat{G}}$ is geometric. Recall
	\begin{equation*}
	\PP_{\sigma_Y}(\phi)=\left\{\EE\in\DbY : \pi^\ast(\EE) \in\PP_{\sigma}(\phi)\right\} 
	\end{equation*}
	for all $\phi\in\R$. Fix $x\in X$, and let $y\in Y$ be the point corresponding to the orbit $Gx$. By assumption, $\OO_y$ is $\sigma_Y$-stable. Let $\phisky$ denote its phase. Then $\pi^\ast(\OO_y)=\bigoplus_{g\in G}\gstar\OO_x\in\PP_{\sigma}(\phisky)$. Moreover, since $\cP_\sigma(\phisky)$ is closed under direct summands, $\gstar \OO_x\in\PP_{\sigma}(\phisky)$ for all $g\in G$. In particular, $\OO_x\in\PP_{\sigma}(\phisky)$. Now suppose that $\OO_x$ is strictly semistable; then there exist $A,B\in\PP_{\sigma}(\phisky)$ such that
	\begin{equation*}
		A\longhookrightarrow\OO_x\longtwoheadrightarrow B
	\end{equation*}
	is a non-trivial exact sequence in $\PP_{\sigma}(\phisky)$, \textit{i.e.}~$A$ is not isomorphic to $0$ or $\OO_x$. By Step~\ref{T33-step1}, $(\pi_\ast)^{-1}$ sends $\PP_\sigma(\phisky)$ to $\PP_{\sigma_Y}(\phisky)$. Hence we have a short exact sequence in $\PP_{\sigma_{Y}}(\phisky)$,
	\begin{equation*}
		\pi_\ast(A)\longhookrightarrow\pi_\ast(\OO_x)=\OO_y\longtwoheadrightarrow\pi_\ast(B).
	\end{equation*}
	However, $\OO_y$ is stable; hence $\pi_\ast(A)=0$ or $\pi_\ast(B)=0$. But $\pi$ is finite; hence $\pi_\ast$ is conservative. Therefore, $A=0$ or $B=0$, and we have a contradiction.
        \qedhere
        \end{enumerate}
\end{proof}

\subsection{Group actions and geometric stability conditions on surfaces}\label{geometric subsection: group actions and the classification of geometric stability conditions on surfaces}

We denote by $\Gstab{}(X)$ the set of all geometric stability conditions on $X$. We will see in \cref{thm: geo stab determined by Z} that if $X$ is a surface, then $\sigma\in\Gstab{}(X)$ is determined by its central charge up to shifting by $[2n]$. This means that to test if $\sigma$ is $G$-invariant, we only have to check the central charge. 

\begin{lemma}\label{geometric stability condition G invariant if and only if central charge is}
	Let $G$ be a group acting on a surface $X$. Then $\sigma=(\PP,Z)\in\Gstab{}(X)$ is $G$-invariant if and only if\, $Z$ is $G$-invariant.
\end{lemma}

\begin{proof}
	If $\sigma=(\PP,Z)\in \Gstab{}(X)$ is $G$-invariant, then so is $Z$. Suppose $\sigma=(\PP,Z)\in\Gstab{}(X)$ and $Z$ is $G$-invariant. Fix $g\in G$. Then $g^\ast \sigma = (g^\ast(\PP),Z\circ g^\ast)$ and $\sigma$ are both geometric, and skyscraper sheaves have the same phase. By Theorem~\ref{thm: geo stab determined by Z}, we have $\sigma = g^\ast \sigma$.
\end{proof}

\begin{lemma}\label{lem: Ghat acts trivially on Knum}
	Let $G\subseteq\Pic^0(X)$ be a finite subgroup. Then the induced action of\, $G$ on $\Knum(X)$ is trivial.
\end{lemma}

\begin{proof}
	Let $\LL\in G$ and $[E]\in\Knum(X)$. The induced action of $G$ on $\Knum(X)$ is given by $\LL\cdot [E] \coloneqq  [E\otimes \LL]$. Since $\LL$ is a numerically trivial line bundle, $\ch(\LL)=e^{c_1(\LL)}$ and $c_1(\LL)=0$ in $\Chownum(X)$. Therefore,
	\begin{equation*}
		\ch\mid_\Knum([E\otimes \LL])=\ch\mid_\Knum([E])\cdot\ch\mid_\Knum(\LL)=\ch\mid_\Knum([E]).
	\end{equation*}
	By the Hirzebruch--Riemann--Roch theorem, the map $\ch\colon \K(X)\rightarrow \Chow(X)$ induces an injective map $\ch\colon \Knum(X)\rightarrow\Chownum(X)$. Therefore, $\LL\cdot [E]=[E\otimes \LL] = [E]$ in $\Knum(X)$.
\end{proof}

\begin{corollary}\label{all geometrics are Ghat invariant}
	Let $S$ be a surface, and let $G\subseteq \Pic^0(S)$ be a finite subgroup. Then every geometric stability condition on $S$ is $G$-invariant.
\end{corollary}

\begin{proof}
	Let $\sigma=(\PP,Z)\in\Gstab{}(S)$. By \cref{geometric stability condition G invariant if and only if central charge is}, it is enough to show that $Z$ is $G$-invariant. By \cref{lem: Ghat acts trivially on Knum}, the group $G$ acts trivially on $\Knum(S)$. Since $\sigma$ is numerical, $Z\colon\K(S)\rightarrow \C$ factors via $\Knum(S)$; hence $Z$ is $G$-invariant.
\end{proof}

\begin{example}\label{eg: Ghat acts by degree 0 line bundles}
	Suppose $G$ is a finite abelian group acting freely on a variety $X$, and let $Y\coloneqq X/G$. Then by Proposition~\ref{defn of G hat action}, there is also an action of $\widehat{G}=\Hom(G,\C)$ on $\DbGX\cong\DbY$. As discussed in \cref{geometric subsection: inducing geometric stability conditions}, the corresponding action on $\DbY$ is given by tensoring with a numerically trivial line bundle $\LL_\chi$ for each $\chi\in\widehat{G}$. If $X$ is a surface, then \cref{all geometrics are Ghat invariant} tells us that every geometric stability condition on $\DbY$ is $\widehat{G}$-invariant.
\end{example}

\subsection{Applications to varieties with finite Albanese morphism}

\begin{lemma}\label{G invariant union of connected components}
	Suppose that a finite group $G$ acts on a triangulated category $\DD$ by exact autoequivalences such that the induced action on $\Knum(\DD)$ is trivial. Then $(\Stab(\DD))^G$ is a union of connected components inside $\Stab(\DD)$.
\end{lemma}

\begin{proof}
	By \cref{Bridgelandmanifold}, there is a local homeomorphism
	\begin{equation*}
		\mathcal{Z}\colon \Stab(\DD)\longrightarrow \Hom_\bZ(\Knum(\DD),\C).
	\end{equation*}
	Let $g\in G$, and denote by $(\Phi_g)_\ast$ the induced action of $g$ on $\K(\DD)$ and $\Knum(\DD)$. Recall that the action of $G$ on $\Stab(\DD)$ is given by $(\Phi_g)_\ast\cdot\sigma=(\Phi_g(\PP),Z\circ (\Phi_g)_\ast^{-1})$. The induced action of $G$ on $\Knum(\DD)$ is trivial; hence $\ZZ(\sigma)$ is $G$-invariant and $\mathcal{Z}(g\cdot\sigma)= \mathcal{Z}(\sigma)$. Furthermore, $G$ acts continuously on $\Stab(\cD)$, and the local homeomorphism $\ZZ$ commutes with this action. Hence the properties of being $G$-invariant and not being $G$-invariant are open in $\Stab(\DD)$, so the result follows.
\end{proof}

We now combine this with the results of Sections~\ref{geometric subsection: inducing geometric stability conditions} and~\ref{geometric subsection: group actions and the classification of geometric stability conditions on surfaces}.

\begin{theorem}\label{Albanese connected component}
	Let $X$ be a variety with finite Albanese morphism. Let $G$ be a finite abelian group acting freely on~$X$, and let $Y=X/G$. Then $\Stab^\ddagger(Y)\coloneqq (\Stab(Y))^{\widehat{G}}\cong \Stab(X)^G$ is a union of connected components in $\Stab(Y)$ consisting only of geometric stability conditions.
\end{theorem}

\begin{proof}
The variety $X$ has finite Albanese morphism, so it follows from \cite[Theorem 1.1]{fuStabilityManifoldsVarieties2022} that all stability conditions on~$X$ are geometric. In particular, all $G$-invariant stability conditions on $X$ are geometric, so from Theorem~\ref{geometric G inv corresponds to geometric G hat inv} it follows that all $\widehat{G}$-invariant stability conditions on $Y$ are geometric. Hence $(\Stab(Y))^{\widehat{G}}\subset \Gstab{}(Y)$.
	
	Recall from Example~\ref{eg: Ghat acts by degree 0 line bundles} that $\widehat{G}$ acts on $\Db(Y)$ by tensoring with numerically trivial line bundles. Now we may apply Lemma~\ref{lem: Ghat acts trivially on Knum}, so it follows that $\widehat{G}$ acts trivially on $\Knum(Y)$. Hence, by Lemma~\ref{G invariant union of connected components}, $(\Stab(Y))^{\widehat{G}}$ is a union of connected components.
\end{proof}

When $X$ is a surface, we will see in \cref{thm: geometric open set connected for surfaces} that $\Gstab{}(S)$ is connected. Hence we have the following stronger result.

\begin{theorem}\label{finite albanese surface quotient has connected component of geos}
	Let $X$ be a surface with finite Albanese morphism. Let $G$ be a finite abelian group acting freely on~$X$. Let $S=X/G$. Then $\Stab^\ddagger(S)=\Gstab{}(S)\cong(\Stab(X))^G$. In particular, $\Stab^\ddagger(S)$ is a connected component of\, $\Stab(S)$.
\end{theorem}

\begin{proof}
	By Theorem~\ref{Albanese connected component}, $\Stab^\ddagger(S)\subset\Gstab{}(S)$ is a union of connected components. By \cref{thm: geometric open set connected for surfaces}, $\Gstab{}(S)$ is connected. In particular, $\Stab^\ddagger(S)=\Gstab{}(S)$, and this is a connected component of $\Stab(S)$.
\end{proof}

\begin{remark} \leavevmode
  \begin{enumerate}
	\item The equality $\Gstab{}(S)=(\Stab(S))^{\widehat{G}}$ also follows by combining \cref{Albanese connected component} with \cref{all geometrics are Ghat invariant}.
	\item The equality  $\Stab^\ddagger(S)=\Gstab{}(S)$ will be explicitly described in \cref{thm: LP gives precise control over set of geometric stability conditions}.
\end{enumerate}	
\end{remark}

\begin{example}\label{example: stability conditions on abelian Beauville and bielliptic}
	Let $S=(C_1\times C_2) / G$ be the quotient of a product of smooth curves such that $g(C_1),$ $g(C_2)\geq 1$ and $G$ is a finite abelian group acting freely on $S$. Then $C_1\times C_2$ has finite Albanese morphism. By \cref{finite albanese surface quotient has connected component of geos}, $\Gstab{}(S)$ is a connected component. In particular, we could take $S$ to be any bielliptic surface (see Example~\ref{eg: bielliptic surfaces}) or a Beauville-type surface with $G$ abelian (see Example~\ref{eg: Beauville-type surfaces}).
\end{example}

\begin{remark}\label{remark: Lambda H stability conditions on surfaces}
	For an ample class $H$ on a variety of dimension $n$, consider the following surjection from $K(X)$: 
	\begin{equation*}
		[E]\longmapsto \left(H^n \ch_0(E), H^{n-1}\ldot \ch_1(E), \ldots, \ch_n(E)\right) \subseteq \bR^n.
	\end{equation*}
	Let $\Lambda_H$ denote the image. The submanifold $\Stab_H(X)\coloneqq \Stab_{\Lambda_H}(X)\subseteq \Stab(X)$ is often studied. Note that these are the same when $X$ has Picard rank 1.

	Now let $X$ be a surface with finite Albanese morphism, and let $G$ be an abelian group acting freely on $X$. Let $S=X/G$, and denote by $\pi\colon X\rightarrow S$ the quotient map. Moreover, let $H_X$ be a $G$-invariant polarization of $X$, and let $H_S$ be the corresponding polarization on $S$ such that $\pi^\ast H_S = H_X$. Then if a homomorphism $Z\colon K(X)\rightarrow \bC$ factors via $\Lambda_{H_X}$, it is $G$-invariant. Hence by \cref{geometric stability condition G invariant if and only if central charge is}, all stability conditions in $\Stab_{H_X}(X)$ are $G$-invariant.
	
	From \cref{finite albanese surface quotient has connected component of geos} it follows that $\Stab^\ddagger_{H_S}(S)\cong(\Stab_{H_X}(X))^G=\Stab_{H_X}(X)$. The component $\Stab_{H_X}(X)$ is the same as the component described in \cite[Corollary 3.7]{fuStabilityManifoldsVarieties2022}. This gives another proof that $\Stab_{H_S}^\ddagger(S)=\Gstab{H_S}(S)$ is connected. 
\end{remark}

\begin{example}\label{example: CY 3fold of abelian type}
	A \textit{Calabi--Yau threefold of abelian type} is an \'etale quotient $Y=X/G$ of an abelian threefold $X$ by a finite group $G$ acting freely on $X$ such that the canonical line bundle of $Y$ is trivial and $H^1(Y,\C)=0$. As discussed in \cite[Example 10.4(i)]{bayerSpaceStabilityConditions2016}, these are classified in \cite[Theorem 0.1]{oguisoCalabiYauThreefolds2001}. In particular, $G$ can be chosen to be $(\Z/2\Z)^{\oplus 2}$ or $D_4$ (the dihedral group	of order~8), and the Picard rank of $Y$ is 3 or 2, respectively.
	
	Fix a polarization $(Y,H)$, and consider $\Stab_H(Y)$ as in \cref{remark: Lambda H stability conditions on surfaces}. This has a connected component~$\mathfrak{P}$ of geometric stability conditions induced from $\Stab_H(X)$, see \cite[Corollary~10.3]{bayerSpaceStabilityConditions2016}, which is described explicitly in \cite[Lemma 8.3]{bayerSpaceStabilityConditions2016}. When $G=(\Z/2)^{\oplus 2}$, by \cite[Theorem 3.21]{oberdieckDonaldsonThomasInvariants2022}, the stability conditions constructed by Bayer--Macr\`i--Stellari in $\Stab_H(X)$ satisfy the full support property (\textit{i.e.}~the support property with respect to $\Knum(X)$), so they actually lie in $\Stab(X)$. Together with Theorem~\ref{Albanese connected component}, this implies that $\sigma\in\mathfrak{P}$ also satisfies the full support property. In particular, $\mathfrak{P}$ lies in a connected component of $\Stab^\ddagger(Y)$.
\end{example}
\section{The Le Potier function}\label{section: Le Potier}

We compute the Le Potier function of free abelian quotients and varieties with finite Albanese morphism. We apply this to Beauville-type surfaces which provides counterexamples to Conjecture~\ref{conj: FLZ21}. Throughout, $X$ will be a smooth projective connected variety over $\bC$.

\subsection{\texorpdfstring{$\boldsymbol{H}$}{H}-stability}

\begin{notation}
	Let $A\cdot B$ denote the intersection product of elements of $\Chownum(X)\otimes\R$. If $A\cdot B$ is 0-dimensional, we define $A\ldot B \coloneqq \deg(A\cdot B)$.
\end{notation}

\begin{definition}\label{defn: H-stability}
	Let $\dim X = n$. Fix an ample class $H\in\Amp_\R(X)$. Given $0\neq F\in\Coh(X)$, we define the \textit{H-slope} of $F$ as follows:
	\begin{equation*}
		\mu_H(F) \coloneqq  \begin{cases}
			\frac{H^{n-1}\ldot\ch_1(F)}{H^n\ch_0(F)} & \text{if }\ch_0(F)>0, \\ +\infty & \text{if }\ch_0(F)=0.
		\end{cases}
	\end{equation*}
	We say that $F$ is $H$-\textit{stable} (resp.\ $H$-\textit{semistable}) if for every non-zero subobject $E \subsetneq F$, 
	\begin{equation*}
	  \mu_H(E)<\mu_H(F/E)\quad (\text{resp. }\mu_H(E)\leq\mu_H(F/E)). 
	\end{equation*}
\end{definition}

\subsection{The Le Potier function}\label{Le Potier}

When studying $H$-stability, a natural question that arises is whether there are necessary and sufficient conditions on a cohomology class $\gamma\in H^\ast(X,\Q)$ for there to exist an $H$-semistable sheaf $F$ with $\ch(F)=\gamma$.

The Bogomolov--Gieseker inequality (see \cite[Section~10]{bogomolovHolomorphicTensorsVector1979} or \cite[Theorem 3.4.1]{huybrechtsGeometryModuliSpaces2010}) gives the following necessary condition for $H$-semistable sheaves on surfaces:
\begin{equation*}
	2 \ch_0(F)\ch_2(F)\leq \ch_1(F)^2.
\end{equation*}
This generalises to the following statement for any variety $X$ of dimension $n\geq 2$ via the Mumford--Mehta--Ramanathan restriction theorem.

\begin{theorem}[\textit{cf.} {\cite[Theorem 3.2]{langerSemistableSheavesPositive2004}, \cite[Theorem 7.3.1]{huybrechtsGeometryModuliSpaces2010}}]\label{thm: BG for higher dimensions}
	Assume $\dim X=n\geq 2$. Fix $H\in\Amp_\R(X)$. If\, $F$ is a torsion-free $H$-semistable sheaf, then
	\begin{equation*}
		2 \ch_0(F)\left(H^{n-2}\ldot\ch_2(F)\right)\leq H^{n-2}\ldot\ch_1(F)^2.
	\end{equation*}
\end{theorem}

\begin{remark}
	Let $B\in\NS_\R(X)$. The \textit{twisted Chern character} is defined by $\ch^B \coloneqq \ch\cdot e^{-B}$. Then
	\begin{align*}
		2\ch_0^B(F)\left(H^{n-2}\ldot\ch_2^B(F)\right) -H^{n-2}\ldot\left(\ch_1^B(F)\right)^2 = 2 \ch_0(F)\left(H^{n-2}\ldot\ch_2(F)\right)- H^{n-2}\ldot\ch_1(F)^2; 
	\end{align*}
	hence \cref{thm: BG for higher dimensions} also holds for twisted Chern characters.
\end{remark}

Now assume $\dim X=n\geq 2$, and fix $(H,B)\in\Amp_\R(X)\times\NS_\R(X)$. Then $H^n>0$. Let $F$ be any $H$-semistable torsion-free sheaf. By the twisted version of Theorem~\ref{thm: BG for higher dimensions},
\begin{equation*}
	2H^{n} \ch_0(F)\left(H^{n-2}\ldot\ch_2^B(F)\right)\leq H^n \left(H^{n-2}\ldot\ch_1^B(F)^2\right)\leq\left(H^{n-1}\ldot \ch_1^B(F)\right)^2,
\end{equation*}
where the final inequality is by the Hodge index theorem. Since $F$ is torsion-free,
\begin{equation*}
	\frac{H^{n-2}\ldot\ch_2^B(F)}{H^n\ch_0(F)}\leq \frac{1}{2}\left(\frac{H^{n-1}\ldot \ch_1^B(F)}{H^n\ch_0(F)}\right)^2.
\end{equation*}
Now we expand the expressions for $\ch_2^B(F)$ and $\ch_1^B(F)$:
\begin{align*}
	&\frac{H^{n-2}\ldot\ch_2(F)-H^{n-2}\ldot B\ldot\ch_1(F)+\frac{1}{2}H^{n-2}\ldot B^2\ldot\ch_0(F)}{H^{n}\ch_0(F)} \\
	&\leq \frac{1}{2} \left(\frac{H^{n-1}\ldot\ch_1(F)-H^{n-1}\ldot B\ch_0(F)}{H^n\ch_0(F)}\right)^2 \\
	&= \frac{1}{2}\left(\mu_H(F)-\frac{H^{n-1}\ldot B}{H^n}\right)^2.
\end{align*}
Therefore,
\begin{equation}\label{eq: Le Potier upper bound}
	\nu_{H,B}(F)\coloneqq  \frac{H^{n-2}\ldot\ch_2(F)-H^{n-2}\ldot B\ldot\ch_1(F)}{H^n\ch_0(F)}\leq \frac{1}{2}\left(\mu_H(F)-\frac{H^{n-1} \ldot B}{H^n}\right)^2 - \frac{1}{2}\frac{H^{n-2}\ldot B^2}{H^n}.
\end{equation}
For a given $\mu\in\bR$, if $\mu_H(F)=\mu$, we can therefore ask how large $\nu_{H,B}(F)$ can be. These leads us to make the following definition.

\begin{definition}\label{defn: Le Potier}
	Assume $\dim X=n\geq 2$. Let $(H,B)\in \Amp_\R(X)\times \NS_\R(X)$. We define the \textit{Le Potier function twisted by $B$}, $\Phi_{X,H,B}\colon\R\rightarrow\R\cup\{-\infty\}$, by
	\begin{equation}\label{eq: Le Potier function}
		\Phi_{X,H,B}(x)\coloneqq \limsup_{\mu\rightarrow x}\left\{  \nu_{H,B}(F) : \text{$F\in\Coh(X)$ is $H$-semistable with  $\mu_H(F)=\mu$} \right\}.
	\end{equation}
\end{definition}

\begin{remark}
	If $B=0$, we will write $\Phi_{X,H}\coloneqq \Phi_{X,H,0}$. If $n=2$, then $\Phi_{X,H}$ is exactly \cite[Definition 3.1]{fuStabilityManifoldsVarieties2022}. 
\end{remark}

The above discussion and definition generalises \cite[Proposition 3.2]{fuStabilityManifoldsVarieties2022}. 

\begin{lemma}[\textit{cf.} {\cite[Proposition 3.2]{fuStabilityManifoldsVarieties2022}}]\label{lem: LP function well defined and bounded}
	Let $X$ be a variety of dimension $n\geq 2$. Let  $(H, B)$ be classes in $\Amp_\R(X)\times\NS_\R(X)$. Then $\Phi_{X,H,B}$ is well defined and satisfies
	\begin{equation*}
		\Phi_{X,H,B}(x)\leq 	\frac{1}{2}\left[\left(x-\frac{H^{n-1}\ldot B}{H^n}\right)^2 - \frac{H^{n-2}\ldot B^2}{H^n}\right].
	\end{equation*}
	It is the smallest upper-semi-continuous function such that
	\begin{equation*}
		\nu_{H,B}(F)
		\leq
		\Phi_{X,H,B}\left(\mu_H(F)\right)
	\end{equation*}
	for every torsion-free $H$-semistable sheaf $F$.
\end{lemma}

\subsection{The Le Potier function for free quotients}

Let $G$ be a finite group acting freely on $X$. There is an \'etale covering $\pi\colon X\rightarrow X/G=:Y$. Then $\Pic(Y)\cong \Pic_G(X)$, the group of isomorphism classes of $G$-equivariant line bundles on $X$. Fix  $H_S\in\Amp_\R(Y)$. Then $\pi^\ast H_S\in\Amp_\R(X)$ is $G$-invariant. Beauville-type and bielliptic surfaces provide examples of such quotients.

\begin{example}[Ample classes on Beauville-type surfaces]\label{eg: ample classes on Beauville}
  Let $S= X / G$ be a Beauville-type surface, as introduced in Example~\ref{eg: Beauville-type surfaces}. Then $X=C_1\times C_2$ is a product of curves of genus $g(C_i)\geq 2$, $q(S)\coloneqq h^1(S,\OO_S)=0$, and $p_g(S)\coloneqq h^2(S,\OO_S)=0$, so $\chi(\OO_S)=1$, and $K_S^2=8$, where $K_S$ is the canonical divisor of $S$.
	
	Assume that there are actions of $G$ on each curve $C_i$ such that the action of $G$ on $C_1\times C_2$ is the diagonal action. This is called the \textit{unmixed case} in \cite[Theorem 0.1]{bauerClassificationSurfacesIsogenous2008} and excludes 3 families of dimension 0. To classify ample classes on $S$, we follow similar arguments to \cite[Section~2.2]{galkinExceptionalCollectionsLine2013}. Let $p_i\colon X\rightarrow C_i$ denote the projections to each curve. For $i,j\in\Z$, define the $G$-invariant divisor class
	\begin{equation*}
		[\OO(i,j)] \coloneqq  p_1^\ast\left(\left[\OO_{C_1}(i)\right]\right)\otimes p_2^\ast \left(\left[\OO_{C_2}(j)\right]\right)\in\NS_G(X).
	\end{equation*}
	Moreover,
	\begin{equation*}
		\chitop(S) = \frac{\chitop(C_1)\cdot\chitop(C_2)}{|G|} = 4 \frac{(1-g(C_1))(1-g(C_2))}{|G|} = 4 \chi(\OO_S) =4.
	\end{equation*}
	Therefore, $\rank \NS(S) = b_2(S)=2$ and
	\begin{equation*}
		\NS_\Q(S) \cong \Q \cdot [\OO(1,0)]\oplus \Q\cdot[\OO(0,1)]. 
	\end{equation*}
	In particular, $\Amp_\R(S) \cong \R_{>0}\cdot[\OO(1,0)]\oplus\R_{>0}\cdot[\OO(0,1)]$.
\end{example}

\begin{lemma}[\textit{cf.} {\cite[Lemma 3.2.2]{huybrechtsGeometryModuliSpaces2010}}]\label{pullbacks are semistable}
	Let $f\colon X\rightarrow Y$ be a finite morphism of varieties of dimension $n\geq 2$, and let $ E\in\Coh(Y)$. Let $(H_Y, B_Y)\in\Amp_\R(Y)\times\NS_\R(Y)$. Then $ E$ is $H_Y$-semistable if and only if $f^* E$ is $f^* H_Y$-semistable. Moreover, if\, $\ch_0( E)\neq 0$, then $\mu_{H_Y}( E) = \mu_{f^* H_Y}(f^* E)$ and $\nu_{H_Y,B_Y}( E) = \nu_{f^* H_Y, f^* B_Y}(f^* E)$. In particular, $\Phi_{X,f^* H_Y, f^* B_Y}\geq \Phi_{Y,H_Y,B_Y}$. 
\end{lemma}

\begin{proof}
	The claim that $ E$ is $H_Y$-semistable if and only if $f^* E$ is $f^* H_Y$-semistable follows from the same arguments as in the proof of \cite[Lemma 3.2.2]{huybrechtsGeometryModuliSpaces2010}. If $\ch_0( E)\neq 0$, then
	\begin{align*}
		\mu_{f^\ast H_Y}(f^\ast E)
		&= \frac{\deg((f^\ast H_Y)^{n-1}\cdot f^\ast(\ch_1( E)))}{\deg\left((f^\ast H_Y)^n\cdot f^\ast(\ch_0( E))\right)}\\
		&= \frac{\deg(f^\ast (H_Y^{n-1}\cdot\ch_1( E)))}{\deg(f^\ast (H_Y^n\cdot \ch_0( E)))} \quad
		\text{($f$ is flat, so $f^\ast$ is a ring morphism)}\\
		&= \frac{\deg(f)\deg (H_Y^{n-1}\cdot\ch_1( E))}{\deg(f)\deg(H_Y^n\cdot\ch_0( E))} \quad \text{(projection formula)}\\
		& = \mu_{H_Y}( E)
	\end{align*}
	By the same arguments, $\nu_{f^* H_Y, f^* B_Y}(f^* E)=\nu_{H_Y,B_Y}( E)$.
\end{proof}

\begin{lemma}\label{pushforwards are semistable}
	Suppose a finite group $G$ acts freely on $X$. Let $\pi\colon X\rightarrow Y\coloneqq X/G$ denote the quotient map. If $ F\in\Coh(X)$ is $\pi^\ast H_Y$-semistable, then $\pi_\ast  F$ is $H_Y$-semistable. Moreover, if $\ch_0( F)\neq 0$, then $\mu_{H_Y}(\pi_*  F) = \mu_{\pi^* H_Y}( F)$ and $\nu_{H_Y,B_Y}(\pi_*  F) = \nu_{\pi^* H_Y, \pi^* B_Y}( F)$. In particular, $\Phi_{X,\pi^* H_Y, \pi^* B_Y}\leq \Phi_{Y,H_Y,B_Y}$.
\end{lemma}

\begin{proof}
	Suppose that $ F\in\Coh(X)$ is $\pi^\ast H_Y$-semistable. Recall from \cref{geometric subsection: inducing geometric stability conditions} that, under the equivalence $\Coh(Y)\cong \Coh_G(X)$, we have $\pi^*\circ\pi_*\cong \Forg_{G}\circ\Inf_{G}$, so 
	\begin{equation*}
		\pipull(\pi_\ast( F)) \cong \Forg_{G}\circ\Inf_{G}( F)= \bigoplus_{g\in G}g^\ast  F.
	\end{equation*}
	Since $\pi^\ast H_Y$ is $G$-invariant, it follows that $g^\ast  F$ is $\pi^\ast H_Y$-semistable for every $g\in G$. In particular, $\bigoplus_{g\in G}g^\ast F$ is $\pi^\ast H_Y$-semistable. By Lemma~\ref{pullbacks are semistable}, $\pi_\ast F$ is $H_Y$-semistable.

	Now suppose $\ch_0( F)\neq 0$. Since the Chern character is additive, $\mu_{\pi^* H_Y}(\pipull\pi_\ast  F) = \mu_{\pi^* H_Y}( F)$. By Lemma~\ref{pullbacks are semistable}, $\mu_{H_Y}( \pi_\ast  F)=\mu_{\pi^* H_Y}(\pipull\pi_\ast  F)=\mu_{\pi^* H_Y}( F)$, as required.
	
	By the same arguments, $\nu_{H_Y,B_Y}(\pi_*  F) = \nu_{\pi^* H_Y, \pi^* B_Y}( F)$.
\end{proof}

\begin{proposition}\label{Le Potier Functions Agree}
	Suppose a finite group $G$ acts freely on $X$. Let $\pi\colon X\rightarrow Y\coloneqq X/G$ denote the quotient map. Let $(H_Y, B_Y)\in\Amp_\R(Y)\times\NS_\R(Y)$. Then $\Phi_{Y,H_Y,B_Y}=\Phi_{X,\pipull H_Y,\pipull B_Y}$.
\end{proposition}

\begin{proof}
	 This follows from Lemmas~\ref{pullbacks are semistable} and~\ref{pushforwards are semistable}.
\end{proof}

\subsection{The Le Potier function for varieties with finite Albanese morphism}

The Le Potier function for surfaces with finite Albanese morphism was known previously; see \cite[Example 2.12(2)]{lahozChernDegreeFunctions2022}. Below, we give a different proof which works for $\Phi_{X,H,B}$ in any dimension. We first need the following definition.

\begin{definition}[\textit{cf.} {\cite[Definitions 4.4 and~5.2]{mukaiSemihomogeneousVectorBundles1978}}]
	A vector bundle $E$ on an abelian variety $A$ is \textit{homogeneous} if it is invariant under translations, \textit{i.e.}~for every $x\in A$, $T_x^\ast(E)\cong E$, where $T_x$ is translation on $A$ by $x$. The vector bundle $E$ is called \textit{semi-homogeneous} if for every $x\in A$, there exists a line bundle $L$ on $A$ such that $T_x^\ast(E)\cong E\otimes L$.
\end{definition}

See \cite[Proposition 5.1]{mukaiSemihomogeneousVectorBundles1978} for some equivalent characterisations for when a vector bundle is semi-homogeneous. We will need the following properties.

\begin{theorem}[\textit{cf.} {\cite[Theorem 4.17, Lemma 6.11]{mukaiSemihomogeneousVectorBundles1978}}]\label{thm: properties of homog and semihomog vbdles}
	Let $E$ be a vector bundle with $ch_0(E)=r$ on an abelian variety $A$.
	\begin{enumerate}
		\item\label{t413-1} The vector bundle $E$ is homogeneous if and only if $E\cong \oplus_{i=1}^k \left(P_i\otimes U_i\right)$, where each $P_i$ is a numerically trivial line bundle and each $U_i$ is a unipotent line bundle, \textit{i.e.}~an iterated self-extension of\, $\OO_{A}$.
		\item\label{t413-2} Suppose $E$ is semi-homogeneous, and consider the multiplication by $r$ map, $r_A\colon A\rightarrow A$. Then we have $r_A^* E \cong \det(E)^{\otimes r} \otimes V$, where $V$ is a homogeneous vector bundle with $\ch_0(V)= \ch_0(r_A^* E)$ and $c_1(V)=0$.
	\end{enumerate}
\end{theorem}

There are many $H$-semistable semi-homogeneous vector bundles on any abelian variety.

\begin{proposition}[\textit{cf.} {\cite[Theorem 7.11]{mukaiSemihomogeneousVectorBundles1978}}]\label{Mukai PFVB}
	Let $A$ be an abelian variety, and fix $H\in\Amp_\R(A)$. For every divisor class $C\in\NS_\Q(A)$, there exists an $H$-semistable semi-homogeneous vector bundle $E_C$ on $A$ with $C=\frac{\ch_1(E_C)}{\ch_0(E_C)}$ and $\ch(E_C)=\ch_0(E_C) \cdot e^{C}$.
\end{proposition}

\begin{proof}
	These vector bundles are constructed as follows: for any $C\in \NS_\bQ(A)$, write $C=\frac{[L]}{l}$, where $[L]$ is the equivalence class of $L\in \NS(A)$ and $l\in\bZ_{>0}$. Let $l_A\colon A \rightarrow A$ denote the multiplication by $l$ map, and define $F= (l_A)_*((L)^{\otimes l})$. By \cite[Proposition 6.22]{mukaiSemihomogeneousVectorBundles1978}, $F$ is a semi-homogeneous vector bundle with $C = \delta (F) \coloneqq \frac{\det(F)}{\ch_0(F)}$. Moreover, $F$ has a filtration by semi-homogeneous vector bundles $E_1,\ldots, E_m$. By \cite[Proposition 6.15]{mukaiSemihomogeneousVectorBundles1978}, each $E_i$ is $\mu_H$-semistable for any $H\in\Amp_\R(A)$ and satisfies $C=\delta(E_i)$.

	Let $E_C\coloneqq E_1$, and let $r=\ch_0(E_C)$. We claim that $\ch(E_C)= r e^C$. Consider the multiplication by $r$ map $r_A\colon X \rightarrow X$. By \cref{thm: properties of homog and semihomog vbdles}\eqref{t413-2}, we have $r_A^* E_C \cong \det(E_C)^{\otimes r} \otimes V$, where $V$ is a homogeneous vector bundle, and
	\begin{equation}\label{eq: chern character of semihog pulled back by r}
		\ch(r_A^* E_C) = \ch(V) \cdot \ch\left(\det(E_C)^{\otimes r}\right) = \ch_0\left(r_A^* E_C\right) e^{r\det(E_C)} = \ch_0\left(r_A^* E_C\right) e^{r^2 C}.
	\end{equation}
	Now recall that $H^{2i}(A,\bC)=\bigwedge^{2i} H^1(A,\bC)$. On $H^1(A,\bC)$, $r_A^*$ is multiplication by $r$. It follows that on $\Chow_{\mathrm{num}}^i(X)$, $r_A^*$ is multiplication by $r^{2i}$. In particular, $\delta(r_A^* E_C) = r^2 \delta(E_C)$. Hence \eqref{eq: chern character of semihog pulled back by r} becomes
	\begin{equation*}
		\ch\left(r_A^* E_C\right) = \ch_0\left(r_A^* E_C\right) e^{\delta(r_A^* E_C)} = r_A^*\left( \ch_0(E_C) e^{\delta(E_C)}\right) = r_A^*\left(r e^C\right).
	\end{equation*}
	Since $r_A$ is flat, the claim follows.
\end{proof}

We use this to compute the Le Potier function for abelian varieties.

\begin{proposition}\label{prop: special semihomog bundles for Le Potier computation}
	Let $A$ be an abelian variety of dimension $n\geq 2$. Fix $(H,B)\in\Amp_\R(A)\times \NS_R(A)$. Then
	\begin{equation*}
		\Phi_{A,H,B}(x) = \frac{1}{2}\left[ \left( x - \frac{H^{n-1}\ldot B}{H^n}\right)^2 - \frac{H^{n-2}\ldot B^2}{H^n}\right].
	\end{equation*}
\end{proposition}

\begin{proof}
	 For any $k\in \bQ$, define $C_k \coloneqq kH + B$. Then by \cref{Mukai PFVB}, there exists a $\mu_H$-semistable vector bundle $E_{C_k}$ with $C_k=\frac{\ch_1(E_{C_k})}{\ch_0(E_{C_k})}$ and $\ch(E_{C_k})=\ch_0(E_{C_k}) \cdot e^{C_k}$. Let $r=\ch_0(E_{C_k})$. Hence 
	\begin{equation*}
		\mu_H(E_{C_k})= \frac{H^{n-1}\ldot r C_k}{H^n r} = k + \frac{H^{n-1}\ldot B}{H^n},
	\end{equation*}
	and
	\begin{align*}
		\nu_{H,B}(E_{C_k}) &= \frac{H^{n-2}\ldot \frac{1}{2} r C_k^2 - H^{n-2}\ldot B\ldot rC_k}{H^n r}\\ 
			&= \frac{1}{2}\frac{H^{n-2}\ldot\left(k^2H^2 + 2kH\ldot B + B^2\right) - H^{n-2}\ldot (2kH\ldot B + 2B^2)}{H^n} \\
			&= \frac{1}{2}\left[ k^2 - \frac{H^{n-2}\ldot B^2}{H^n}\right]\\
			&= \frac{1}{2}\left[ \left(\mu_H(E_{C_k}) - \frac{H^{n-1}\ldot B}{H^n}\right)^2 - \frac{H^{n-2}\ldot B^2}{H^2}\right].
	\end{align*}
	This gives a lower bound for $\Phi_{A,H,B}\left(\mu_H(E_{C_k})\right)$, which is the same as the upper bound in Lemma~\ref{lem: LP function well defined and bounded}. Now note that for any $x\in \bQ$, we can choose $k$ so that $\mu_H(E_{C_k})=x$. Hence $\Phi_{A,H,B}(x)$ attains its upper bound for all $x\in\bQ$. Finally, by the definition of the Le Potier function, it must attain this upper bound for all $x\in\bR$.
\end{proof}

Varieties with finite Albanese morphism also have many $H$-semistable vector bundles.

\begin{proposition}[\textit{cf.} {\cite[Example 2.12(2)]{lahozChernDegreeFunctions2022}}]\label{stable pullback from albanese}
	Let $X$ be a variety with finite Albanese morphism $a\colon X\rightarrow \Alb(X)$ and $n\coloneqq \dim X \geq 2$. Let $H_X\in\Amp_\R(X)$. Then $a^\ast E_C$ is $H_X$-semistable for every $C\in\NS_\Q(\Alb(X))$.
\end{proposition}

\begin{proof}
	Fix $C\in\NS_\Q(\Alb(X))$ and $H_A\in\Amp_\R(\Alb(X))$. Let $E_C$ be the corresponding $H_A$-semistable semi-homogeneous vector bundle on $\Alb(X)$ from Proposition~\ref{Mukai PFVB}. Let $r\coloneqq \ch_0(E_C)$, and consider the multiplication by $r$ map $r_{\Alb(X)} \colon \Alb(X)\rightarrow \Alb(X)$. By \cref{thm: properties of homog and semihomog vbdles},
	\begin{equation*}
		r_{\Alb(X)}^\ast(E_C) = L^{-1} \otimes \left(\bigoplus_{i=1}^k P_i\otimes U_i\right),
	\end{equation*}
	where $L$ is a line bundle and for all $i$, $P_i$ is a numerically trivial line bundle and $U_i$ is an iterated self-extension of $\OO_{\Alb(X)}$. Therefore, $L\otimes r_{\Alb(X)}^\ast(E_C)$ is an iterated extension of numerically trivial line bundles.

	Now consider the fibre square
	\begin{center}
	\begin{tikzcd}
		\cZ\coloneqq  X\times_{\Alb(X)}\Alb(X) \arrow[d, "p_X"] \arrow[r, "p_A"] & \Alb(X) \arrow[d, "r_A"] \\
		X \arrow[r, "a"]                                              & \Alb(X)\rlap{.}                 
	\end{tikzcd}
	\end{center}
	Without loss of generality, fix a connected component $Z$ of $\cZ$. Then on $Z$, 
	\begin{equation*}
		(p_X|_Z)^\ast a^\ast (E_C)= (p_A|_Z)^\ast r_A^\ast (E_C).
	\end{equation*}
	The property of being an extension of numerically trivial line bundles is preserved by taking pullback. Hence $p_A^*(L)\otimes (p_X|_Z)^\ast a^\ast(E_C)$ is an iterated extension of numerically trivial line bundles. Recall that line bundles are stable with respect to any ample class. Thus $p_A^*(L)\otimes (p_X|_Z)^\ast a^\ast(E_C)$ is $(p_X|_Z)^\ast H_X$-semistable; hence so is $(p_X|_Z)^\ast a^\ast(E_C)$. By Lemma~\ref{pullbacks are semistable}, $a^\ast(E_C)$ is $H_X$-semistable.
\end{proof}

\begin{proposition}[\textit{cf.} {\cite[Example 2.12(2)]{lahozChernDegreeFunctions2022}}]\label{Le Potier for Albanese theorem}
	Let $X$ be a variety with finite Albanese morphism $a\colon X\rightarrow \Alb(X)$. Fix $(H,B)\in\Amp_\R(\Alb(X))\times\NS_\R(\Alb(X))$, and assume $n\coloneqq \dim X \geq 2$. Then
	\begin{equation*}
		\Phi_{\Alb(X), H, B}(x) = \Phi_{X,a^\ast H,a^\ast B}(x) = \frac{1}{2}\left[\left(x-\frac{(a^\ast H)^{n-1}\ldot a^\ast B}{(a^\ast H)^n}\right)^2 - \frac{(a^\ast H)^{n-2}\ldot (a^\ast B) ^2}{(a^\ast H)^n}\right].
	\end{equation*}
\end{proposition}

\begin{proof}
	First note that, by the projection formula, the upper bounds of $\Phi_{\Alb(X), H, B}$ and $\Phi_{X,a^\ast H,a^\ast B}$ are the same. By \cref{pullbacks are semistable}, $\Phi_{\Alb(X), H, B}\leq \Phi_{X,a^\ast H,a^\ast B}$. Hence it suffices to show that $\Phi_{\Alb(X), H, B}$ attains this upper bound. This follows from \cref{prop: special semihomog bundles for Le Potier computation}.
\end{proof}

We now combine this with Proposition~\ref{Le Potier Functions Agree}.

\begin{theorem}\label{thm: LP for quotients of varieties with finite Albanese}
	Let $X$ be a variety with finite Albanese morphism $a\colon X\rightarrow \Alb(X)$, and let $G$ be a finite group acting freely on $X$.  Let $\pi\colon X\rightarrow X/G=:Y$ denote the quotient map. Suppose we have 
	\begin{itemize}
		\item $H_X=a^\ast H = \pi^\ast H_Y$: a class in $\Amp_\R(X)$ pulled back from $\Alb(X)$ and $Y$, and 
		\item $B_X=a^\ast B = \pi^\ast B_Y$: a class in $\NS_\R(X)$ pulled back from $\Alb(X)$ and $Y$.
	\end{itemize}
	Then
	\begin{equation*}
		\Phi_{Y,H_Y,B_Y}(x)=\frac{1}{2}\left[\left(x-\frac{H_Y^{n-1}\mathbin{.} B_Y}{H_Y^n}\right)^2 - \frac{H_Y^{n-2}\mathbin{.}B_Y^2}{H_Y^n}\right].
	\end{equation*}
\end{theorem}

\begin{proof}
	By Propositions~\ref{Le Potier Functions Agree} and~\ref{Le Potier for Albanese theorem}, it follows that 
	\begin{equation*}
		\Phi_{Y,H_Y,B_Y}(x)=\Phi_{X,\pi^\ast H_Y, \pipull B_Y}(x)=\frac{1}{2}\left[\left(x-\frac{(\pi^\ast H_Y)^{n-1}\ldot \pi^\ast B_Y}{(\pi^\ast H_Y)^n}\right)^2 - \frac{(\pi^\ast H_Y)^{n-2}\ldot (\pi^\ast B_Y) ^2}{(\pi^\ast H_Y)^n}\right].
	\end{equation*}
	The result follows by the projection formula.
\end{proof}

\begin{example}\label{eg: ample classes for finite Albanese FAQs}
	Suppose $X$ has finite Albanese morphism $a\colon X\rightarrow\Alb(X)$, and let $G$ be a finite group acting freely on $X$. This induces an action of $G$ on $\NS(\Alb(X)$). Fix $L\in\Amp(\Alb(X))$ and $B=0$. Then $H\coloneqq \otimes_{g\in G}g^\ast L\in\Amp_\R(\Alb(X))$ satisfies the hypotheses of \cref{thm: LP for quotients of varieties with finite Albanese}. In particular, this applies to bielliptic surfaces ($q=1$) and Beauville-type surfaces ($q=0$). The latter provides a counterexample to Conjecture~\ref{conj: FLZ21} since $\Phi_{Y,H_Y,0}(x) = \frac{1}{2}{x^2}$ is continuous.
\end{example}

\section{Geometric stability conditions and the Le Potier function}\label{section: geometric stability and Le Potier}

We use the Le Potier function to describe the set of geometric stability conditions on any surface. This was previously known for surfaces with Picard rank 1; see \cite[Theorem 3.4 and Proposition 3.6]{fuStabilityManifoldsVarieties2022}.

\subsection{The deformation property and tilting}

To prove the existence of stability conditions later in this section, we will need the following refinement of Theorem~\ref{Bridgelandmanifold}. 

\begin{proposition}[\textit{cf.} {\cite[Proposition A.5]{bayerSpaceStabilityConditions2016}, \cite[Theorem 1.2]{bayerShortProofDeformation2019}}]\label{prop: deformation property}
	Let $\cD$ be a triangulated category. Assume $\sigma=(\PP,Z)\in\Stab_\Lambda(\DD)$ satisfies the support property with respect to $(\Lambda,\lambda)$ and a quadratic form $Q$ on $\Lambda\otimes\R$. Consider the open subset of\, $\Hom_\Z(\Lambda,\C)$ consisting of central charges whose kernel is negative definite with respect to $Q$, and let $U$ be the connected component containing $Z$. Let $\ZZ$ denote the local homeomorphism from Theorem~{\rm\ref{Bridgelandmanifold}}, and let $\mathcal{U}\subset\Stab_\Lambda(\DD)$ be the connected component of the preimage $\mathcal{Z}^{-1}(U)$ containing $\sigma$. Then
	\begin{enumerate}
		\item\label{p:dp-1} the restriction $\mathcal{Z}|_{\mathcal{U}} \colon \mathcal{U}\rightarrow U$ is a covering map, and 
		\item\label{p:dp-2} any stability condition $\sigma'\in\mathcal{U}$ satisfies the support property with respect to $Q$.
	\end{enumerate}
\end{proposition}

\begin{corollary}\label{cor: defo property for ImZ constant}
	Let $\DD$ be a triangulated category. Assume $\sigma=(\PP, Z)\in\Stab_\Lambda(\DD)$ satisfies the support property with respect to $(\Lambda, \lambda)$ and a quadratic form $Q$ on $\Lambda\otimes\R$.	Let $U\subset\Hom_\Z(\Lambda,\C)$, and let $\mathcal{U}\subset\Stab_\Lambda(\DD)$ be the connected components from Proposition~{\rm\ref{prop: deformation property}}. Suppose there is a path $Z_t$ in $U$ parametrised by $t\in[0,1]$ such that $\Im {Z_t}$ is constant and $Z_{t_0}=Z$ for some $t_0\in[0,1]$. Then this lifts to a unique path $\sigma_t = (\QQ_t,Z_t)$ in $\mathcal{U}$ passing through $\sigma$ along which $\QQ_t(0,1]=\PP(0,1]$ and $\sigma_t$ satisfies the support property with respect to $Q$.
\end{corollary}

\begin{proof}
	Let $\ZZ$ denote the local homeomorphism from Theorem~\ref{Bridgelandmanifold}. By Proposition~\ref{prop: deformation property}\eqref{p:dp-1}, $\ZZ|_\UU \colon \UU \rightarrow U$ is a covering map. By the path lifting property, there is a unique path $\sigma_t = (\QQ_t,Z_t)$ in $\mathcal{U}$ with $\sigma=\sigma_{t_0}$. By Proposition~\ref{prop: deformation property}\eqref{p:dp-2}, $\sigma_t$ satisfies the support property with respect to $Q$ for all $t$. It remains to show that $\QQ_t(0,1]=\PP(0,1]$.
	
	    Fix a non-zero object $E\in\cD$. We claim that the set of points in the path $\sigma_t$ where $E\in \QQ_t(0,1]$ is open and closed. Suppose $E\in\QQ_T(0,1]$ for some $T\in[0,1]$. Then all Jordan--H\"older factors $E_i$ of $E$ with respect to $\sigma_T$ are in $\QQ_T(0,1]$ and satisfy $\Im Z_T(E_i)\geq 0$. The property for an object to be stable is open in $\Stab_\Lambda(\cD)$ (see \cite[Proposition 3.3]{bayerSpaceStabilityConditions2011}). Moreover, $0<\phi_{\QQ_t}(E_i)$ is an open property. Since $\Im Z_t$ is constant, $\Im Z_t(E_i)\geq 0$ for all $t$. Hence, for all sufficiently close $\sigma_t$, we have $\phi_{\QQ_t}(E_i)\leq 1$ and $E\in\QQ_t(0,1]$.

	Now suppose $\sigma_T$ is in the closure and not the interior of $\{\sigma_t : E\in\QQ_t(0,1]\}$ inside $\{\sigma_t : t\in[0,1]\}$.	Recall that $\phi^+(E)$ and $\phi^-(E)$ are continuous.	Hence $\phi^-_{\QQ_T}(E)=0$, and $E$ has a morphism to a stable object in $\QQ_T(0)$ which is also stable nearby. In particular, $\{\sigma_t : E\not\in\QQ_t(0,1]\}$ is open, which proves the claim. Hence $\QQ_t(0,1]$ is constant. Since $\QQ_{t_0}(0,1]=\PP(0,1]$, the result follows.
\end{proof}

To construct stability conditions, we will also need the following definition.

\begin{definition}[\textit{cf.} {\cite[Section~I.2]{happelTiltingAbelianCategories1996}}]
	Let $\AA$ be an abelian category. A \textit{torsion pair} in $\AA$ is a pair of full additive subcategories $(\TT,\FF)$ of $\AA$ such that
	\begin{enumerate}[label={(\arabic*)}]
		\item for any  $T\in\TT$ and $F\in\FF$, $\Hom(T,F)=0$; 
		\item for any $E\in\AA$, there are $T\in\TT$, $F\in\FF$, and an exact sequence
		\begin{equation*}
			0 \longrightarrow T \longrightarrow E \longrightarrow F \longrightarrow 0.
		\end{equation*}
	\end{enumerate}
\end{definition}

\begin{proposition}[\textit{cf.} {\cite[Proposition 2.1]{happelTiltingAbelianCategories1996}}]\label{HRStilting}
	Suppose $(\TT,\FF)$ is a torsion pair in an abelian category $\AA$. Then
	\begin{equation*}
		\AA^\sharp \coloneqq  \left\{E\in\text{\normalfont D}^b(\AA) : \HH_\AA^0(E)\in\TT, \; \HH_A^{-1}(E)\in\FF, \; \HH_\AA^i(E)=0 \text{ for all } i\neq0,-1\right\}
	\end{equation*}
	is the heart of a bounded $t$-structure on $\text{\normalfont D}^b(\AA)$. We call $\AA^\sharp$  the \emph{tilt} of $\AA$ with respect to $(\TT,\FF)$.
\end{proposition}

\subsection{The central charge of a geometric stability condition}

For the rest of this section, let $X$ be a smooth projective connected surface over $\C$. We are particularly interested in geometric Bridgeland stability conditions, \textit{i.e.}~$\sigma\in\Stab(X)$ such that the skyscraper sheaf $\OO_x$ is $\sigma$-stable for every point $x\in X$. Denote by $\Gstab{}(X)$ the set of all geometric stability conditions.

\begin{theorem}[\textit{cf.} {\cite[Proposition 10.3]{bridgelandStabilityConditionsK32008}}]\label{thm: geo stab determined by Z}
	Let $X$ be a surface, and let $\sigma=(\PP,Z) \in \Gstab{}(X)$. Then $\sigma$ is determined by its central charge up to shifting the slicing by $[2n]$ for any $n\in\Z$.
	
	Moreover, if $\sigma$ is normalised using the action of\, $\C$ such that $Z(\OO_x)=-1$ and $\phi(\OO_x)=1$ for all $x\in X$, then
	\begin{enumerate}
		\item\label{t:gsdZ-1} the central charge can be uniquely written in the  form 
		\begin{equation*}
			Z([E]) = (\alpha-i\beta)H^2\ch_0([E]) + (B+iH)\ldot \ch_1([E]) -\ch_2([E]),
		\end{equation*}
		where $\alpha,\beta\in\R$, $(H,B)\in\Amp_\R(X)\times \NS_\R(X)$; 
	
		\item\label{t:gsdZ-2} the heart, $\PP(0,1]$, is the tilt of\, $\Coh(X)$ at the torsion pair $(\TT,\FF)$, where
		\begin{align*}
			\TT &\coloneqq  \left\{ E\in \Coh(X) : \parbox{20em}{any $H$-semistable Harder--Narasimhan factor $F$ of the torsion-free part of\, $E$ satisfies $\Im Z([F])> 0$}\;  \right\}, \\
			\FF &\coloneqq  \left\{ E\in \Coh(X) : \parbox{20em}{$E$ is torsion-free, and any $H$-semistable Harder--Narasimhan factor $F$ of\, $E$ satisfies $\Im Z([F]) \leq 0$}\; \right\}.
		\end{align*}
	\end{enumerate}
\end{theorem}

\begin{notation}
	We will use $Z_{H,B,\alpha,\beta}=Z$ to denote central charges of the above form. Since $\Im Z_{H,B,\alpha,\beta}$ only depends on $H$ and $\beta$, we will denote the torsion pair by $(\TT_{H,\beta},\FF_{H,\beta})$ and write $\Coh^{H,\beta}(X)$ for the corresponding tilted heart. Then $\sigma_{H,B,\alpha,\beta}\coloneqq ( Z_{H,B,\alpha,\beta}, \Coh^{H,\beta}(X))$.
\end{notation}

The proof is similar to the case of K3 surfaces proved in \cite[Section~10]{bridgelandStabilityConditionsK32008}. We first need the following result which immediately generalises to any surface. 

\begin{lemma}[\textit{cf.} {\cite[Lemma 10.1]{bridgelandStabilityConditionsK32008}}]\label{lem: Bri08}
	Suppose $\sigma=(\PP,Z)\in\Stab(X)$ is a stability condition on a surface $X$ such that for each point $x\in X$, the sheaf\, $\OO_x$ is $\sigma$-stable of phase~$1$. Let $E$ be an object of $\DbX$. Then
	\begin{enumerate}
		\item\label{l:B08-1} if\, $E\in\PP(0,1]$, then $H^i(E)=0$ unless $i\in\{-1,0\}$, and moreover $H^{-1}(E)$ is torsion-free; 
		\item\label{l:B08-2} if\, $E\in\PP(1)$ is stable, then either $E=\OO_x$ for some $x\in X$, or $E[-1]$ is a locally free sheaf; 
		\item\label{l:B08-3} if\, $E\in\Coh(X)$ is a sheaf, then $E\in\PP(-1,1]$; if\, $E$ is a torsion sheaf, then $E\in\PP(0,1]$; 
		\item\label{l:B08-4} the pair of subcategories
		\begin{equation*}
			\TT = \Coh(X)\cap \PP(0,1]\quad \text{ and }\quad \FF=\Coh(X)\cap \PP(-1,0]
		\end{equation*}
		defines a torsion pair on $\Coh(X)$, and $\PP(0,1]$ is the corresponding tilt.
	\end{enumerate}
\end{lemma}

\begin{proof}[Proof Theorem~\ref{thm: geo stab determined by Z}]\leavevmode
\begin{enumerate}[wide,label={\it Step}~\arabic*.,ref=\arabic*]
\item\label{pTgsdZ-1} Since $\sigma$ is numerical, the central charge can be written as follows:
	\begin{equation*}
		Z([E]) = a\; \ch_0([E]) + B\ldot \ch_1([E]) + c\; \ch_2([E]) + i(d\; \ch_0([E]) + H \ldot\ch_1([E]) + e\;\ch_2([E])),
	\end{equation*}
	where $a,c,d,e\in\R$ and $B,H\in \NS_\R(X)$.
	
	Since $\sigma$ is geometric, $\OO_x$ is $\sigma$-stable and of the same phase for every point $x\in X$ by Proposition~\ref{prop: numerical geometric implies all skyscraper sheaves have the same phase}. As discussed in Remark~\ref{rem: GLcov action on Stab}, $\C$ acts on $\Stab(X)$. In particular, there is a unique element $g\in\C$ such that $g^\ast \sigma=(\PP',Z')$ satisfies $Z'([\OO_x])=-1$ and $\OO_x\in\PP'(1)$ for all $x\in X$. Now we may assume that $Z([\OO_x])=-1$ and $\OO_x\in\PP(1)$ for all $x\in X$. Hence $-1=c$ and $e=0$. Let $C\subset X$ be a curve. By Lemma~\ref{lem: Bri08}\eqref{l:B08-3}, we have $\OO_C\in \PP(0,1]$. Since $\ch_0(\OO_C)=0$ and $\ch_1(\OO_C)=C$,
	\begin{equation*}
		\Im Z([\OO_C]) = H \ldot C \geq 0.
	\end{equation*}
	This holds for any curve $C\subset X$, so $H\in\NS_\R(X)$ is nef. By \cite[Proposition 9.4]{bridgelandFourierMukaiTransformsK32002}, $\Gstab{}(X)$ is open. Moreover, by Theorem~\ref{Bridgelandmanifold}, a small deformation from $\sigma$ to $\sigma'$ in $\Gstab{}(X)$ corresponds to a small deformation of the central charges $Z$ to $Z'$, and in turn a small deformation of $H$ to $H'$ inside $\NS_\R(X)$. In particular, $H'\ldot C \geq 0$ for any curve $C\subset X$. Therefore, $H$ lies in the interior of the nef cone; hence $H$ is ample.
	
	Now let $\alpha\coloneqq \frac{a}{H^2}$ and $\beta\coloneqq \frac{-d}{H^2}$. Then the central charge is of the form 
	\begin{equation*}
		Z([E]) = (\alpha-i\beta)H^2\ch_0([E]) + (B+iH)\ldot \ch_1([E]) -\ch_2([E]).
	\end{equation*}
	
\item\label{pTgsdZ-2} Consider the torsion pair $(\TT,\FF)$ of Lemma~\ref{lem: Bri08}\eqref{l:B08-4}, so $\PP(0,1]$ is the tilt of $\Coh(X)$ at $(\TT,\FF)$. By Lemma~\ref{lem: Bri08}\eqref{l:B08-3}, all torsion sheaves lie in $\TT$. To complete the proof, we need the following claim:	
	\begin{equation*}\tag{$\ast$}
		E\in\Coh(X) \text{ is $H$-stable and torsion-free} \implies
		\begin{cases}
			E\in \TT &\text{if }\Im Z([E])>0,\\
			E\in \FF &\text{if }\Im Z([E])\leq0.
		\end{cases}
	\end{equation*}

	This is Step 2 of the proof of \cite[Lemma~10.3]{bridgelandStabilityConditionsK32008}. Bridgeland first shows that $E$ must lie in $\TT$ or $\FF$. We explain why it then follows that $\Im Z([E])=0$ implies $E\in\FF$. Assume $E$ is non-zero and $E\in\TT$. Since $Z([E])\in\R$, it follows that $E\in\PP(1)$. For any $x\in\Supp(E)$, $E$ has a non-zero map $f\colon E\rightarrow \OO_x$. Let $E_1$ be its kernel in $\Coh(X)$. Since $\OO_x$ is stable, $f$ is a surjection in $\PP(1)$. Thus $E_1$ also lies in $\PP(1)$ and hence in $\TT$. Moreover, $Z([E_1])=Z([E])-Z([\OO_x])=Z([E])-1$. Repeating this by replacing $E$ with $E_1$ and so on creates a chain $E\supsetneq E_1 \supsetneq E_2 \supsetneq \cdots$ of strict subobjects in $\PP(1)$ such that $Z([E_n])=Z([E])-n$. If this process does not terminate, then $Z([E_k])\in\R_{>0}$ for some $k\in\N$, contradicting the fact that $E_n\in\PP(1)$. Otherwise, $E_i\cong \OO_x$ for some $i$, contradicting the fact that $E$ is torsion-free.
        \qedhere
        \end{enumerate}
\end{proof}

\subsection{The set of all geometric stability conditions on surfaces}

In the previous section, we saw that a geometric stability condition on a surface with $Z(\OO_x)=-1$ and $\phi(\OO_x)=1$ is determined by its central charge. In particular, it depends on parameters $(H,B,\alpha,\beta)$ in $\Amp_\R(X)\times\NS_\R(X)\times\R^2$. To characterise geometric stability conditions on surfaces, we will find necessary and sufficient conditions for when these parameters define a geometric stability condition. In Definition~\ref{defn: Le Potier}, we introduced the Le Potier function twisted by $B$. We restate the version for surfaces below.

\begin{definition}\label{defn: twisted Le Potier}
	Let $X$ be a surface. Let $(H,B)\in \Amp_\R(X)\times \NS_\R(X)$. We define the \textit{Le Potier function twisted by }$B$, $\Phi_{X,H,B}\colon\R\rightarrow\R\cup\{-\infty\}$, by
	\begin{equation*}
		\Phi_{X,H,B}(x)\coloneqq \limsup_{\mu\rightarrow x}\left\{ \frac{\ch_2(F)-B\ldot \ch_1(F)}{H^2\ch_0(F)}  \; : \parbox{15em}{$F\in\Coh(X)$ is $H$-semistable with $\mu_H(F)=\mu$} \right\}.
	\end{equation*}
\end{definition}

\begin{remark}
	By \cite[Theorem 5.2.5]{huybrechtsGeometryModuliSpaces2010}, for every rational number $\mu\in\Q$, there exists an $H$-stable sheaf $F$ with $\mu_H(F)=\mu$.
\end{remark}

The goal of this section is to prove the following result.

\begin{theorem}\label{thm: LP gives precise control over set of geometric stability conditions}
	Let $X$ be a surface. There is a homeomorphism of topological spaces
	\begin{equation*}
		\Gstab{}(X)\cong\C\times \left\{(H,B,\alpha,\beta)\in\Amp_\R(X)\times\NS_\R(X)\times\R^2 : \alpha>\Phi_{X,H,B}(\beta)\right\}.
	\end{equation*}
\end{theorem}

\begin{remark}\leavevmode
	\begin{enumerate}
		\item Theorem 6.10 of \cite{macriLecturesBridgelandStability2017} describes a subset of $\Gstab{}(X)$ parametrised by classes $(H,B)$ in the product $\Amp_\R(X)\times\NS_\R(X)$. This corresponds to where $\alpha>\frac{1}{2}\big[\big(\beta - \frac{H\ldot B}{H^2}\big)^2-\frac{B^2}{H^2}\big]$ in  Theorem~\ref{thm: LP gives precise control over set of geometric stability conditions} (see \cref{prop: special values of alpha and beta give a geometric stability condition} for details). We will call this the \textit{BG range}.
		\item The subset of $\Amp_\R(X)\times\NS_\R(X)\times\R^2$ on the right-hand side of the homeomorphism can be viewed as a complex submanifold of $\NS_\bC(X)\times \bC$ via $(H,B,\alpha,\beta)\mapsto (H+iB,\alpha-i\beta)$. With this identification, the homeomorphism above is in fact one of complex manifolds.
	\end{enumerate}
\end{remark}

\begin{notation}
	To ease notation, we make the following definitions: 
	\begin{align*}
		\mathcal{U}&\coloneqq \left\{(H,B,\alpha,\beta)\in \Amp_\R(X)\times\NS_\R(X)\times\R^2 : \alpha>\Phi_{X,H,B}(\beta)\right\},\\
		\Gstab{N}(X)&\coloneqq \left\{\sigma=(\PP,Z)\in\Gstab{}(X):Z(\OO_x)=-1, \OO_x\in\PP(1)\; \forall x\in X\right\}. 
	\end{align*}
\end{notation}

\subsubsection*{Idea of the proof of Theorem~\ref{thm: LP gives precise control over set of geometric stability conditions}}
By Theorem~\ref{thm: geo stab determined by Z}, for every $\sigma\in\Gstab{}(X)$, there exists a unique $g$ in~$\C$ such that $g^\ast \sigma\in\Gstab{N}(X)$. To prove Theorem~\ref{thm: LP gives precise control over set of geometric stability conditions}, it is enough to show that there is a homeomorphism $\Gstab{N}(X)\cong\UU$. We do this in two steps:

\begin{enumerate}[wide,label={\it Step}~\arabic*.,ref=\arabic*]
\item\label{T510-step1} \textit{Construct an injective, local homeomorphism} $\Pi\colon \Gstab{N}(X)\rightarrow \UU$. Theorem~\ref{thm: geo stab determined by Z} shows that, for every $\sigma\in\Gstab{N}(X)$, there are unique $(H,B,\alpha,\beta)\in\Amp_\R(X)\times\NS_\R(X)\times\R^2$ such that $\sigma=\sigma_{H,B,\alpha,\beta}$. This gives an injective map 
\begin{align*}
	\Pi\colon \Gstab{N}(X) &\longrightarrow \Amp_\R(X)\times\NS_\R(X)\times\R^2\\
	\sigma=\sigma_{H,B,\alpha,\beta} &\longmapsto (H,B,\alpha,\beta).
\end{align*}
We will show that  $\Pi$ is  a local homeomorphism (\cref{prop: inj local homeo}) and that the image is contained in $\UU$ (\cref{prop: geometric implies twisted LP condition}).

\item\label{T510-step2} \textit{Construct a pointwise inverse $\Sigma\colon \UU\rightarrow \Gstab{N}(X)$}. We will first show this is possible for $(H,B,\alpha,\beta)$ in the BG range (\cref{prop: special values of alpha and beta give a geometric stability condition}). In Proposition~\ref{prop: LP stability conditions satisfy support property}, we extend this to any $\alpha>\Phi_{X,H,B}(\beta)$ by applying Corollary~\ref{cor: defo property for ImZ constant} as follows: 
\begin{itemize}
	\item Fix $(H,B)\in\Amp_\R(X)\times\NS_\R(X)$ and $\alpha_0>\Phi_{X,H,B}(\beta_0)$.
	\item  Fix $\alpha_1>\frac{1}{2}\big[\big(\beta_0 - \frac{H\ldot B}{H^2}\big)^2-\frac{B^2}{H^2}\big]$.
\end{itemize}
If only $\alpha$ varies, then $\Im Z_{H,B,\alpha,\beta_0}$ is constant. We construct a quadratic form (Proposition~\ref{prop: final quadratic form}) and show that it gives the support property for $\sigma_{H,B,\alpha_1,\beta_0}$ (Lemma~\ref{lem: special values of alpha and beta satisfy new support property}) and is negative definite on $\Ker Z_{H,B,\alpha,\beta_0}$ for all $\alpha>\Phi_{X,H,B}(\beta_0)$ (Lemma~\ref{lem: Q negative definite on kernel as alpha varies}).
\end{enumerate}

\subsubsection{STEP 1: Construction of the map $\boldsymbol{\Gstab{N}(X)\rightarrow \UU}$}

\begin{proposition}\label{prop: inj local homeo}
	Let $X$ be a surface. Then there is an injective local homeomorphism
	\begin{align*}
		\Pi\colon \Gstab{N}(X) &\longrightarrow  \Amp_\R(X)\times\NS_\R(X)\times\R^2\\
		\sigma=\sigma_{H,B,\alpha,\beta} &\longmapsto (H,B,\alpha,\beta).
	\end{align*}
\end{proposition}

\begin{proof}
	Let $\ZZ\colon \Stab(X)\rightarrow\Hom_\bZ(\Knum(X),\C)$ denote the local homeomorphism from Theorem~\ref{Bridgelandmanifold}. Also define $\cN\coloneqq\{(\cP, Z)\in\Stab(X) : Z(\cO_x)=-1\}$, and consider the following diagram: 
	\begin{equation*}
		\begin{tikzcd}[column sep=small]
			& {\Stab(X)} & \cN & {\Gstab{N}(X)} \\
			{\Hom(\Knum(X), \bC)} & {\{ Z : Z(\cO_x)=-1\}} & {\cZ(\cN)} & {\cZ(\Gstab{N}(X))}\rlap{.}
			\arrow["\supset"{description}, draw=none, from=1-2, to=1-3]
			\arrow["\cZ"', from=1-2, to=2-1]
			\arrow["\supset"{description}, draw=none, from=1-3, to=1-4]
			\arrow["{\cZ|_{\cN}}"', from=1-3, to=2-3]
			\arrow["{\cZ|_{\Gstab{N}(X))}}"', from=1-4, to=2-4]
			\arrow["\supset"{description}, draw=none, from=2-1, to=2-2]
			\arrow["\supset"{description}, draw=none, from=2-2, to=2-3]
			\arrow["\supset"{description}, draw=none, from=2-3, to=2-4]
		\end{tikzcd}
	\end{equation*}
	Since $\cZ$ is a local homeomorphism and restriction is injective, $\cZ|_\cN$ and $\cZ|_{\Gstab{N}(X)}$ are also local homeomorphisms. Moreover, by the same argument as in Step~\ref{pTgsdZ-1} of the proof of \cref{thm: geo stab determined by Z},
	\begin{align*}
		\{ Z : Z(\cO_x)=-1\} &\overset{\cong}\lra \left( \NS_\bR(X)\right)^2 \times \bR^2\\
		Z = Z_{H,B,\alpha,\beta} &\longmapsto (H,B,\alpha,\beta).
	\end{align*}
	Define $\Pi$ to be the composition of $\cZ|_{\Gstab{N}(X)}$ with this isomorphism. Now \cref{thm: geo stab determined by Z} implies that $\Pi$ is an injective local homeomorphism and
	\begin{equation*}\pushQED{\qed}
		\im \Pi \subseteq  \Amp_\R(X)\times\NS_\R(X)\times\R^2.
\qedhere \popQED
	\end{equation*}
\renewcommand{\qed}{}    
\end{proof}

Before we can prove that $\Pi(\Gstab{N}(X))\subseteq\UU$, we need the following result.

\begin{lemma}\label{lem: open neighbourhood of geometric stability condition}
	Suppose $\sigma=\sigma_{H,B,\alpha,\beta}\in\Gstab{N}(X)$ is geometric. There there is an open neighbourhood $W\subset\R^2$ of $(\alpha,\beta)$ such that for every $(\alpha',\beta')\in W$, we have $\sigma_{H,B,\alpha',\beta'}\in\Gstab{N}(X)$.
\end{lemma}

\begin{proof}
	By \cite[Proposition 9.4]{bridgelandStabilityConditionsK32008}, there is an open neighbourhood $V$ of $\sigma$ in $\Stab(X)$ where all skyscraper sheaves are stable. Let $\ZZ\colon \Stab(X)\rightarrow\Hom_\bZ(\Knum(X),\C)$ denote the local homeomorphism from Theorem~\ref{Bridgelandmanifold}. By \cref{prop: inj local homeo}, $\Pi(V)$ is open in $\Amp_\bR(X)\times\NS_\bR(X)\times \bR^2$. Therefore, $W\coloneqq \Pi(V)|_{\bR^2}$ is open in $\bR^2$.
\end{proof}

\begin{proposition}[\textit{cf.} {\cite[Proposition 3.6]{fuStabilityManifoldsVarieties2022}}]\label{prop: geometric implies twisted LP condition}
	Let $\sigma=\sigma_{H,B,\alpha,\beta}\in\Gstab{N}(X)$. Then $\alpha>\Phi_{X,H,B}(\beta)$, \textit{i.e.}~$\Pi(\Gstab{N}(X))\subseteq\UU$, where $\Pi$ is the injective local homeomorphism from Proposition~{\rm\ref{prop: inj local homeo}}.
\end{proposition}

\begin{proof}
	Suppose towards a contradiction that $\alpha \leq \Phi_{X,H,B}(\beta)$.	Let $W\subset\R^2$ be the open neighbourhood of $(\alpha,\beta)$ from Lemma~\ref{lem: open neighbourhood of geometric stability condition}. Recall that
	\begin{equation*}
		\Phi_{X,H,B}(\beta)\coloneqq \limsup_{\mu\rightarrow \beta}\left\{ \nu_{H,B}(F) : \text{$F\in\Coh(X)$ is $H$-semistable with $\mu_H(F)=\mu$} \right\}.
	\end{equation*}
	Therefore, there exist $H$-semistable sheaves with slopes arbitrarily close to $\beta$ and $\nu_{H,B}$ arbitrarily close to $\Phi_{X,H,B}(\beta)$. In particular, there exist a pair $(\alpha_0,\beta_0)\in W$ and a torsion-free $H$-semistable sheaf $F$ with
	\begin{equation}\label{eq: beta 0 is slope of F and alpha 0 is less than nu of F}
		\beta_0=\mu_H(F)=\frac{H\ldot\ch_1(F)}{H^2\ch_0(F)}
		\quad
		\text{and}
		\quad
		\alpha_0 \leq \nu_{H,B}(F)=\frac{\ch_2(F)-B\ldot\ch_1(F)}{H^2\ch_0(F)}.
	\end{equation}
	Since $(\alpha_0,\beta_0)\in W$, we have $\sigma_{H,B,\alpha_0,\beta_0}\in\Gstab{N}(X)$. Moreover, $\ch_0(F)>0$, so 
	\begin{align*}
		\Im(Z_{H,B,\alpha_0,\beta_0}([F]))&= H\ldot\ch_1([F])-\beta_0 H^2\ch_0([F])=0.
	\end{align*}
	By the definition of the torsion pair $(\TT_{H,\beta_0},\FF_{H,\beta_0})$ in \cref{thm: geo stab determined by Z}, it follows that $F\in\FF_{H,\beta_0}$. This implies that $Z_{H,B,\alpha_0,\beta_0}([F])\in\R_{>0}$. However, by \eqref{eq: beta 0 is slope of F and alpha 0 is less than nu of F},
	\begin{equation*}
		\Re(Z_{H,B,\alpha_0,\beta_0}([F]))=\alpha_0 H^2\ch_0([F])+B\ldot\ch_1([F])-\ch_2([F])\leq 0. 
	\end{equation*}
	Hence $Z_{H,B,\alpha_0,\beta_0}([F])\in\R_{\leq 0}$, so we have a contradiction.
\end{proof}

\subsubsection{STEP 2: Construction of the pointwise inverse $\boldsymbol{\UU\rightarrow \Gstab{N}(X)}$}

We first recall the construction of stability conditions in \cite[Theorem 6.10]{macriLecturesBridgelandStability2017}.

\begin{definition}
	Let $X$ be a surface, and fix classes $(H,B)\in \Amp_\R(X)\times\NS_\R(X)$. Define the pair $\sigma_{H,B}\coloneqq (\Coh^{H,B}(X),Z_{H, B})$, where
	\begin{align*}
		Z_{H,B}([E]) &= \left(-\ch_2^B([E])+\frac{H^2}{2}\ldot\ch_0^B([E])\right) + iH\ldot\ch_1^B([E])\\
		&= \left[\frac{1}{2}\left(1-\frac{B^2}{H^2}\right)-i\frac{H\ldot B}{H^2}\right]H^2\ch_0([E])+(B+iH)\ldot\ch_1([E])-\ch_2([E]),\\
		\TT_{H,B} &= \left\{ E\in \Coh(X) : \parbox{21em}{any $H$-semistable Harder--Narasimhan factor $F$ of the torsion-free part of $E$ satisfies $\Im Z_{H,B}([F])> 0$}  \right\}, \\
		\FF_{H,B} &= \left\{ E\in \Coh(X) : \parbox{21em}{$E$ is torsion-free, and any $H$-semistable Harder--Narasimhan factor $F$ of $E$ satisfies $\Im Z_{H,B}([F]) \leq 0$} \right\},
	\end{align*}
	and $\Coh^{H,B}(X)$ is the tilt of $\Coh(X)$ at the torsion pair $(\TT_{H,B}, \FF_{H,B})$.
\end{definition}

\begin{lemma}[\textit{cf.} {\cite[Exercise 6.11]{macriLecturesBridgelandStability2017}}]\label{lem: constant for torsion sheaves}
	Let $X$ be a surface. Then there exists a continuous function $C_{(-)}\colon\Amp_\R(X)\rightarrow \R_{\geq 0}$ such that, for every $D\in\Eff_\R(X)$,
	\begin{equation*}
		C_H(H\ldot D)^2+D^2\geq 0.
	\end{equation*}
\end{lemma}

\begin{proof}
  The inequality $C_H(H\ldot D)^2+D^2\geq 0$ is invariant under rescaling. If we consider $\Eff_\R(X)\subset \NS_\R(X)$ as normed vector spaces, it is therefore enough to look at the subspace of unit vectors $U\subset\Eff_\R(X)$.
	
	Since $D\in U$ is effective and $D\neq 0$, we have $H\ldot D>0$. Hence there exists a $C\in\R_{\geq 0}$ such that $C(H\ldot D)^2+D^2\geq 0$. Define 
	\begin{equation*}
		C_{H,D}\coloneqq \inf \left\{ C\in\R_{\geq 0} :  C_H(H\ldot D)^2+D^2\geq 0\right\}.
	\end{equation*}
	
	Since $\Amp_\R(X)$ is open, $H'\ldot D>0$ for a small deformation $H'$ of $H$. It follows that $\overline{U}$ is strictly contained in the subspace $\{E\in\NS_\R(X) : E\ldot H>0\}$. Moreover, $C_{H,D}$ is a continuous function on $\overline{U}$, and $\overline{U}$ is compact as it is a closed subset of the unit sphere in $\NS_\R(X)$. Therefore, $C_{H,D}$ has a maximum, which we call $C_H$. By construction, this is a continuous function on $\Amp_\R(X)$.
\end{proof}

\begin{definition}\label{defn: BG and MS inequalities}
	Let $X$ be a surface. Let $(H,B)\in\Amp_\R(X)\times\NS_\R(X)$. We define the following quadratic forms on $\Knum(X)\otimes\R$:
	\begin{align*}
		&Q_{BG}\coloneqq \ch_1^2-2\ch_2\ch_0,\\
		&\MSQuadratic{H}{B}\coloneqq Q_{BG}+ C_H(H\ldot \ch_1^B)^2,
	\end{align*}
	where $C_H\in\R_{\geq 0}$ is the constant from Lemma~\ref{lem: constant for torsion sheaves}.
\end{definition}

\begin{theorem}[\textit{cf.} {\cite[Theorem 6.10]{macriLecturesBridgelandStability2017}}]\label{thm: MS17 stability conditions exist}
  Let $X$ be a surface. Let $(H,B)\in\Amp_\R(X)\times\NS_\R(X)$. Then $\sigma_{H,B}\in\Gstab{N}(X)$. In particular, $\sigma_{H,B}$ satisfies the support property with respect to $\MSQuadratic{H'}{B'}$, where the pair $(H',B')\in\Amp_\Q(X)\times\NS_\Q(X)$ consists of nearby rational classes.
\end{theorem}

\begin{remark}
	Theorem~\ref{thm: MS17 stability conditions exist} was first proved for K3 surfaces in \cite{bridgelandStabilityConditionsK32008}, along with the fact that this gives rise to a continuous family. In \cite[Theorem 6.10]{macriLecturesBridgelandStability2017}, the authors first prove the result holds for rational classes $(H,B)$ and sketch how to extend this to arbitrary classes. In particular, $\sigma_{H,B}$ can be obtained as a deformation of $\sigma_{H',B'}$ for nearby rational classes $(H',B')$, and $\sigma_{H,B}$ satisfies the same support property, $\MSQuadratic{H'}{B'}$. This uses the fact that $\MSQuadratic{H}{B}$ varies continuously with  $(H,B)$, together with similar arguments to Proposition~\ref{prop: deformation property}.
\end{remark}

\begin{proposition}\label{prop: special values of alpha and beta give a geometric stability condition}
	Let $X$ be a surface. Let $(H,B)\in \Amp_\R(X)\times\NS_\R(X)$, and fix $\alpha_0,\beta_0\in\R$ such that $\alpha_0>\Phi_{X,H,B}(\beta_0)$. Suppose $\alpha>\frac{1}{2}\big[\big(\beta_0 - \frac{H\ldot B}{H^2}\big)^2-\frac{B^2}{H^2}\big]$. Define $b\coloneqq \beta_0 - \frac{H\ldot B}{H^2}\in\R$ and $a\coloneqq \sqrt{2\alpha -b^2 + \frac{B^2}{H^2}}\in\R_{>0}$. Then $\sigma_{H,B,\alpha,\beta_0}$ and $\sigma_{aH,B+bH}$ are the same up to the action of\, $\GLcov$. Moreover, this is a continuous family in $\Gstab{N}(X)$ for $\alpha>\frac{1}{2}\big[\big(\beta_0 - \frac{H\ldot B}{H^2}\big)^2-\frac{B^2}{H^2}\big]$. 
\end{proposition}

\begin{proof}
	We abuse notation and consider the central charges as homomorphisms $\Knum(X)\otimes \R\rightarrow\C$. We first claim that $\Ker Z_{H,B,\alpha,\beta_0}=\Ker Z_{aH,B+bH}$ as sub-vector spaces of $\Knum(X)\otimes\R$. Fix $u\in\Knum(X)\otimes\R$. Since $a>0$, we have $\Im Z_{aH,B+bH}(u)=0$ if and only if
	\begin{align*}
		0 &= a H\ldot B \ch_0(u) + abH^2\ch_0(u) -aH\ldot \ch_1(u) \\
		 &= a \left( H\ldot B \ch_0(u) + \left(\beta_0 - \frac{H\ldot B}{H^2}\right) H^2\ch_0(u) -H\ldot \ch_1(u)\right)\\
		 &= a \left(\beta_0H^2\ch_0(u) -H\ldot \ch_1(u)\right) \\
		 &= -a \Im Z_{H,B,\alpha,\beta_0}(u).
	\end{align*}
	Therefore, $\Im Z_{aH,B+bH}(u)=0$ if and only if $\Im Z_{H,B,\alpha,\beta_0}(u)=0$. Now assume $\Im Z_{aH,B+bH}(u)=0$, so $H\ldot\ch_1(u)=\beta_0 H^2\ch_0(u)$. Then $\Re Z_{aH,B+bH}(u)=0$ if and only if
	\begin{align*}
		0	&= \frac{1}{2}\left((aH)^2-(B+bH)^2\right)\ch_0 + B\ldot\ch_1+ bH\ldot\ch_1(u) - \ch_2(u)\\
			&= \frac{1}{2}\left(a^2 - \frac{(B+bH)^2}{H^2}  +2b\beta_0\right)H^2\ch_0(u) + B\ldot\ch_1(u)-\ch_2(u).
	\end{align*}
	Moreover,
	\begin{align*}
		\frac{1}{2}\left(a^2 - \frac{(B+bH)^2}{H^2}  +2b\beta_0\right)
			&= \frac{1}{2}\left(a^2 -\frac{B^2}{H^2} +2b\left(\beta_0 - \frac{B\ldot H}{H^2}\right) - b^2\right)\\
			&= \frac{1}{2}\left(2\alpha -b^2 + \frac{B^2}{H^2} -\frac{B^2}{H^2} +b^2\right)\\
			&= \alpha.
	\end{align*}
	Therefore, $Z_{aH,B+bH}$ and $Z_{H,B,\alpha,\beta_0}$ are the same up to the action of $\GL_2^+(\bR)$. Moreover, by \cref{thm: MS17 stability conditions exist}, $\sigma_{aH,B+bH}\in\Gstab{N}(X)$. Together with \cref{thm: geo stab determined by Z}, this implies that $\sigma_{H,B,\alpha,\beta_0}=g\cdot \sigma_{aH,B+bH} \in\Stab(X)$ for some $g\in\GLcov$. Then, by definition, $\sigma_{H,B,\alpha,\beta_0}\in\Gstab{N}(X)$. It remains to show this gives rise to a continuous family. By Propositions~\ref{prop: inj local homeo} and~\ref{prop: geometric implies twisted LP condition}, 
	\begin{equation*}
		\Pi \colon\Gstab{N}(X) \longrightarrow \UU,\quad \sigma_{H,B,\alpha,\beta}\longmapsto (H,B,\alpha,\beta) 
	 \end{equation*}
	 is an injective local homeomorphism. Let $V\coloneqq   \big\{(H,B,\alpha,\beta) : \alpha>\frac{1}{2}\big[\big(\beta - \frac{H\ldot B}{H^2}\big)^2-\frac{B^2}{H^2}\big]\big\}$. The restriction $\Pi|_{\Pi^{-1}(V)}$ is still an injective local homeomorphism. Moreover, by the arguments above, $\Pi|_{\Pi^{-1}(V)}$ is surjective; hence it is continuous.
\end{proof}

\begin{remark}
	Let $\Sh_2^+(\R)\subset \GL_2^+(\R)$ denote the subgroup of shearings, \textit{i.e.}~transformations that preserve the real line. It is simply connected; hence it embeds as a subgroup into $\GLcov$ and acts on $\Stab(X)$. In the above proof, $\sigma_{H,B,\alpha,\beta_0}$ and $\sigma_{aH,B+bH}$ have the same hearts, so they are the same up to the action of $\Sh_2^+(\R)$.
\end{remark}

The next result follows from the proof of Theorem~\ref{thm: MS17 stability conditions exist}. We explain this part of the argument explicitly as it will be essential for extending the support property in Lemma~\ref{lem: special values of alpha and beta satisfy new support property}.

\begin{lemma}\label{lem: MS support property for large a}
	Let $X$ be a surface. Let $(H,B)\in\Amp_\R(X)\times\NS_\R(X)$. There exist rational classes $(H',B')$ in $\Amp_\Q(X)\times\NS_\Q(X)$ such that, for $a\geq 1$, the quadratic form $\Delta^{C_{H'}}_{H',B'}$ is negative definite on $\Ker Z_{aH,B}\otimes\R$. In particular, $\Delta_{H',B'}^{C_H'}$ gives the support property for $\sigma_{aH,B}$.
\end{lemma}

\begin{proof}
	By Theorem~\ref{thm: MS17 stability conditions exist}, we have $\sigma_{aH,B}\in\Gstab{N}(X)$ for $a\geq 1$, and near $(H,B)$, there exist rational classes $(H',B')\in\Amp_\Q(X)\times\NS_\Q(X)$ such that $\Delta^{C_{H'}}_{H',B'}$ gives the support property for $\sigma_{H,B}\in\Gstab{N}(X)$. In particular, $\Delta^{C_{H'}}_{H',B'}$ is negative definite on $\K_1\coloneqq \Ker Z_{H,B}\otimes\R$. By Proposition~\ref{prop: deformation property}, it is enough to show $\MSQuadratic{H'}{B'}$ is negative definite on $K_a\coloneqq \Ker Z_{aH,B}\otimes\R$ for $a\geq 1$.
	
	Recall that $u=(\ch_0^B(u),\ch_1^B(u),\ch_2^B(u))\in\K_a$ if and only if
	\begin{equation*}
		a^2\frac{H^2}{2}\ch_0^B(u)=\ch_2^B(u), \quad H\ldot \ch_1^B(u)=0.
	\end{equation*}

	Let $\Psi_a\colon \K_1\rightarrow \K_a$ be the isomorphism of sub-vector spaces of $\Knum(X)\otimes\R$ given by
	\begin{equation*}
		\Psi_a\colon v=\left(\ch_0^B(v),\ch_1^B(v),\ch_2^B(v)\right)\longmapsto \left(\ch_0^B(v),\ch_1^B(v),\ch_2^B(v)+(a^2-1)\frac{H^2}{2}\ch_0^B(v)\right).
	\end{equation*}

	Let $u\in\K_a$. Then $u=\Psi_a(v)$ for some $v\in\K_1$. Clearly $\Delta^{C_{H'}}_{H',B'}(0)=0$, so we may assume $u\neq 0$. Hence $v\neq 0$, and it is enough to show that $\Delta^{C_{H'}}_{H',B'}(\Psi_a(v))<0$. Recall that $\ch_1^{B'}=\ch_1 - B'\ldot\ch_0$; hence $\ch_1^{B'}(\Psi_a(v))=\ch_1^{B'}(v)$. Therefore,
	\begin{align*}
		\Delta^{C_{H'}}_{H',B'}(\Psi_a(v)) &= \left(\ch_1^B(v)\right)^2-2\ch_0^B(v)\ch_2^B(v)-2(a^2-1)\frac{H^2}{2}\left(\ch_0^B(v)\right)^2 + C_{H'}\left(H'\ldot\ch_1^{B'}(v)\right)^2\\
		&= \Delta^{C_{H'}}_{H',B'}(v) - 2(a^2-1)\frac{H^2}{2}\left(\ch_0^B(v)\right)^2 \\
		&\leq \Delta^{C_{H'}}_{H',B'}(v).
	\end{align*}
	Since $\Delta^{C_{H'}}_{H',B'}$ is negative definite on $\K_1$, it follows that $\Delta^{C_{H'}}_{H',B'}(\Psi_a(v))<0$.
\end{proof}

\begin{definition}
	Let $X$ be a surface. Let $(H,B)\in \Amp_\R(X)\times\NS_\R(X)$. Let $\alpha>\Phi_{X,H,B}(\beta)$, and let $\delta>0$. We define the following quadratic form on $\Knum(X)\otimes\R$:
	\begin{equation*}
		Q_{H,B,\alpha,\beta,\delta}\coloneqq \delta^{-1}(H\ldot\ch_1-\beta_0H^2\ch_0)^2 -  (H^2\ch_0)\left( \ch_2-B\ldot\ch_1 - (\alpha_0-\delta)H^2\ch_0\right).
	\end{equation*}
\end{definition}

\begin{lemma}\label{lem: LP quadratic form for torsion-free sheaves}
	Let $X$ be a surface. Let $(H,B)\in \Amp_\R(X)\times\NS_\R(X)$. Fix $\alpha_0,\beta_0\in\R$ such that $\alpha_0>\Phi_{X,H,B}(\beta_0)$. Then there exists a $\delta>0$ such that, for every $H$-semistable torsion-free sheaf\, $F$, we have $Q_{H,B,\alpha_0,\beta_0,\delta}([F])\geq 0$.
\end{lemma}

\begin{proof}
	Since $\Phi_{X,H,B}$ is upper semi-continuous and bounded above by a quadratic polynomial in $x$, the same argument as in \cite[Remark 3.5]{fuStabilityManifoldsVarieties2022} applies. In particular, there exists a sufficiently small $\delta>0$ such that
	\begin{equation*}
		\frac{(x-\beta_0)^2}{\delta}+\alpha_0-\delta \geq \Phi_{X,H,B}(x).
	\end{equation*}
	Suppose $F$ is an $H$-semistable torsion-free sheaf. Let $x=\mu_H(F)=\frac{H\ldot\ch_1(F)}{H^2\ch_0(F)}$; then
	\begin{equation*}
		\delta^{-1}\left(H\ldot\ch_1(F)-\beta_0H^2\ch_0(F)\right)^2+(\alpha_0-\delta)\left(H^2\ch_0(F)\right)^2\geq \left(H^2\ch_0(F)\right)^2\Phi_{X,H,B}\left(\frac{H\ldot\ch_1(F)}{H^2\ch_0(F)}\right).
	\end{equation*}
	From Lemma~\ref{lem: LP function well defined and bounded} it follows that
	\begin{equation*}
		\delta^{-1}\left(H\ldot\ch_1(F)-\beta_0H^2\ch_0(F)\right)^2+(\alpha_0-\delta)\left(H^2\ch_0(F)\right)^2 \geq \left(H^2\ch_0(F)\right)^2\frac{\ch_2(F)-B\ldot\ch_1(F)}{H^2\ch_0(F)}.
	\end{equation*}
	In particular,
	\begin{equation*}\pushQED{\qed}
		\delta^{-1}\left(H\ldot\ch_1(F)-\beta_0H^2\ch_0(F)\right)^2 -  \left(H^2\ch_0(F)\right)\left( \ch_2(F)-B\ldot\ch_1(F) - (\alpha_0-\delta)H^2\ch_0(F)\right)\geq 0.\qedhere \popQED
	\end{equation*}
\renewcommand{\qed}{}     
\end{proof}

\begin{remark}
	Let $u\in\Knum(X)\otimes \R$. We now consider $Z_{H,B,\alpha_0,\beta_0}$ again as a homomorphism $\Knum(X)\otimes\R\rightarrow \C$. Recall that $u\in \K_{\alpha_0}\coloneqq \Ker Z_{H,B,\alpha_0,\beta_0}\subseteq \Knum(X)\otimes\R$ if and only if
	\begin{equation*}
		\alpha_0H^2\ch_0(u)+B\ldot\ch_1(u)-\ch_2(u)=0 \quad \text{ and } \quad H\ldot\ch_1(u)-\beta_0H^2\ch_0(u)=0.
	\end{equation*}
	Then
	\begin{equation*}
		Q_{H,B,\alpha_0,\beta_0,\delta}(u)=-\delta\left(H^2\ch_0(u)\right)^2\leq 0 
	\end{equation*}
	for all $u\in \K_{\alpha_0}$. In particular, $Q_{H,B,\alpha_0,\beta_0,\delta}$ is negative semi-definite on $K_{\alpha_0}$. Hence $Q_{H,B,\alpha_0,\beta_0,\delta}$ does not fulfil the support property.
\end{remark}

To construct a quadratic form which is negative definite on $K_{\alpha_0}=\Ker Z_{H,B,\alpha_0,\beta_0}$, we will combine $Q_{H,B,\alpha_0,\beta_0,\delta}$ with $Q_{BG}$, the quadratic form coming from the Bogomolov--Gieseker inequality introduced in Definition~\ref{defn: BG and MS inequalities}.

\begin{lemma}[\textit{cf.} {\cite[Section~10]{bogomolovHolomorphicTensorsVector1979}, \cite[Theorem 3.4.1]{huybrechtsGeometryModuliSpaces2010}}]\label{lem: BG quadratic form for torsion-free sheaves}
	Let $X$ be a surface. Let $H\in\Amp_\R(X)$. Then $Q_{BG}([F])\geq 0$ for every $H$-semistable torsion-free sheaf\, $F$.
\end{lemma}

\begin{proposition}\label{prop: final quadratic form}
	Let $X$ be a surface. Let $(H,B)\in \Amp_\R(X)\times\NS_\R(X)$, and fix $\alpha_0,\beta_0\in\R$ such that $\alpha_0>\Phi_{X,H,B}(\beta_0)$. Choose $\delta>0$ as in Lemma~{\rm\ref{lem: LP quadratic form for torsion-free sheaves}}. Let $Q^{\delta,\varepsilon}_{H,B,\alpha_0,\beta_0}\coloneqq Q_{H,B,\alpha_0,\beta_0,\delta}+\varepsilon Q_{BG}$. Then there exists an $\varepsilon>0$ such that
	\begin{enumerate}
		\item\label{p:fqf-1} $Q^{\delta,\varepsilon}_{H,B,\alpha_0,\beta_0}([F])\geq 0$ for every $H$-semistable torsion-free sheaf $F$, 
		\item\label{p:fqf-2} $Q^{\delta,\varepsilon}_{H,B,\alpha_0,\beta_0}([T])\geq 0$ for every torsion sheaf T, and
		\item\label{p:fqf-3} $Q^{\delta,\varepsilon}_{H,B,\alpha_0,\beta_0}$ is negative definite on $K_{\alpha_0}\coloneqq \Ker Z_{H,B,\alpha_0,\beta_0}\subseteq\Knum(X)\otimes \R$.
	\end{enumerate}
\end{proposition}

\begin{proof}
\eqref{p:fqf-1} follows immediately for any $\varepsilon>0$ from Lemmas~\ref{lem: LP quadratic form for torsion-free sheaves} and~\ref{lem: BG quadratic form for torsion-free sheaves}. For \eqref{p:fqf-2}, let $C_H$ be the constant from Lemma~\ref{lem: constant for torsion sheaves}. Choose $\varepsilon_1 >0$ such that $\varepsilon_1 < \frac{\delta^{-1}}{C_H}$. Let $T$ be a torsion sheaf; then
	\begin{align*}
		Q^{\delta,\varepsilon_1}_{H,B,\alpha_0,\beta_0}([T]) &= \delta^{-1}\left(H\ldot\ch_1([T])\right)^2+\varepsilon_1\ch_1([T])^2 \\
			&= \varepsilon_1\left(\frac{\delta^{-1}}{\varepsilon_1}\left(H\ldot \ch_1([T])\right)^2+\ch_1([T])^2\right)\\
			&> \varepsilon_1 \left(  C_H\left(H\ldot \ch_1([T])\right)^2+\ch_1([T])^2\right)\\
			&\geq 0. 
	\end{align*}

	For \eqref{p:fqf-3}, fix a norm on $\Knum(X)$, and let $U$ denote the set of unit vectors in $K_{\alpha_0}$ with respect to this norm. It will be enough to show there exists an $\varepsilon_2>0$ such that $Q^{\delta,\varepsilon_2}_{H,B,\alpha_0,\beta_0}|_U<0$.
	
	Let $A\coloneqq \{u\in U \st Q_{H,B,\alpha_0,\beta_0,\delta}=0\}$. For any $a\in A$, we have $\ch_0(a)=0$. The condition that $Z_{H,B,\alpha_0,\beta_0}(a)=0$ becomes
	\begin{equation*}
		B\ldot\ch_1(a)=\ch_2(a) \quad \text{ and } \quad H\ldot\ch_1(a)=0.
	\end{equation*}
        The divisor
        $H$ is ample, so $\ch_1(a)^2\leq 0$ by the Hodge index theorem. If $\ch_1^2(a)=0$, then $\ch_1(a)=0$, and hence $0=B\ldot\ch_1(a)=\ch_2(a)$. So $a=0$, which contradicts the fact that $a\in U$. Therefore,
	\begin{equation*}
		\restr{Q_{BG}}{A}([E])=\ch_1([E])^2<0.
	\end{equation*}
	
	We now claim that there exists a sufficiently small $\varepsilon_2>0$ such that $Q^{\delta,\varepsilon_2}_{H,B,\alpha_0,\beta_0}<0$ on $U$. Note that $\restr{Q^{\delta,\varepsilon_2}_{H,B,\alpha_0,\beta_0}}{A}=\restr{\varepsilon_2 Q_{BG}}{A}<0$, so we only need to check the claim on $U\setminus A$. Now suppose the converse, so for every $\varepsilon>0$, there exists a $ u\in U\setminus A$ such that
	\begin{equation*}
		Q_{BG}(u)\geq -\frac{1}{\varepsilon}Q_{H,B,\alpha_0,\beta_0,\delta}(u).
	\end{equation*}
We have	$Q_{H,B,\alpha_0,\beta_0,\delta}(u)<0$ since $Q_{H,B,\alpha_0,\beta_0,\delta}$ is negative semi-definite on $U$ and $u\in U\setminus A$. Therefore,
	\begin{equation*}
		P(u)\coloneqq \frac{Q_{BG}(u)}{-Q_{H,B,\alpha_0,\beta_0,\delta}(u)}\geq \frac{1}{\varepsilon}.
	\end{equation*}
	Thus $P$ is not bounded above on $U\setminus A$. Moreover, $A$ is closed and $\restr{Q_{BG}}{A}<0$ on $A$. Hence $Q_{BG}$ is negative definite on some open neighbourhood $V$ of $A$, so $\restr{P}{V}<0$. Finally, $U\setminus V$ is compact, so $P$ must be bounded above on $U\setminus V$. In particular, $P$ is bounded above on $U\setminus A$, so we have a contradiction. It follows that there exists an $\varepsilon_2>0$ such that $Q^{\delta,\varepsilon_2}_{H,B,\alpha_0,\beta_0}$ is negative definite on $K_{\alpha_0}$. Finally, let $\varepsilon =\min\{\varepsilon_1,\varepsilon_2\}$.
\end{proof}

\begin{lemma}\label{lem: Q negative definite on kernel as alpha varies}
	Let $X$ be a surface. Let $(H,B)\in \Amp_\R(X)\times\NS_\R(X)$. Fix $\alpha_0,\beta_0\in\R$ such that $\alpha_0>\Phi_{X,H,B}(\beta_0)$. Choose $\delta,\varepsilon>0$ as in Proposition~{\rm\ref{prop: final quadratic form}}. Then $Q^{\delta,\varepsilon}_{H,B,\alpha_0,\beta_0}$ is negative definite on $K_\alpha \coloneqq  \Ker Z_{H,B,\alpha,\beta} \otimes \R$ for all $\alpha\geq \alpha_0$.
\end{lemma}

\begin{proof}
	Recall that $u=(\ch_0(u),\ch_1(u),\ch_2(u))\in \K_\alpha=\Ker Z_{H,B,\alpha,\beta_0}\otimes\R$ if and only if
	\begin{align*}
		\alpha H^2\ch_0(u)+B\ldot\ch_1(u)-\ch_2(u)=0, \quad H\ldot\ch_1(u)-\beta_0H^2\ch_0(u)=0.
	\end{align*}

	Let $\Psi_\alpha\colon \K_{\alpha_0}\rightarrow \K_{\alpha}$ be the isomorphism of sub-vector spaces of $\Knum(X)\otimes\R$ given by
	\begin{equation*}
		\Psi_\alpha \colon v=\left(\ch_0(v),\ch_1(v),\ch_2(v)\right)\longmapsto \left(\ch_0(v),\ch_1(v),\ch_2(v)+(\alpha-\alpha_0)H^2\ch_0(v)\right).
	\end{equation*}

	Let $u\in \K_\alpha$; then $u=\Psi_\alpha(v)$ for some $v\in \K_{\alpha_0}$. Clearly $Q^{\delta,\varepsilon}_{H,B,\alpha_0,\beta_0}(0)=0$, so we may assume $u\neq0$. Hence $v\neq 0$, and it is enough to show that $Q^{\delta,\varepsilon}_{H,B,\alpha_0,\beta_0}(\Psi_\alpha(v))<0$. Moreover,
	\begin{align*}
		Q^{\delta,\varepsilon}_{H,B,\alpha_0,\beta_0}(\Psi_\alpha(v)) &= Q_{H,B,\alpha_0,\beta_0,\delta}(\Psi_\alpha(v))+\varepsilon Q_{BG}(\Psi_\alpha(v)) \\
		&= Q_{H,B,\alpha_0,\beta_0,\delta}(v) - (\alpha-\alpha_0)\left(H^2\ch_0(v)\right)^2 + \varepsilon Q_{BG}(v) -2\varepsilon(\alpha-\alpha_0)H^2\ch_0(v)^2\\
		&= Q^{\delta,\varepsilon}_{H,B,\alpha_0,\beta_0}(v)-(\alpha-\alpha_0)H^2\ch_0(v)^2\left(H^2+2\varepsilon\right)\\
		&\leq Q^{\delta,\varepsilon}_{H,B,\alpha_0,\beta_0}(v).
	\end{align*}
	Finally, by Proposition~\ref{prop: final quadratic form}\eqref{p:fqf-3}, $Q^{\delta,\varepsilon}_{H,B,\alpha_0,\beta_0}(v)<0$.
\end{proof}

\begin{lemma}[\textit{cf.} {\cite[Lemma 6.18]{macriLecturesBridgelandStability2017}}]\label{lem: MS large volume limit}
	Let $(H,B)\in \Amp_\R(X)\times\NS_\R(X)$. If\, $E\in\Coh^{H,B}(X)$ is $\sigma_{aH,B}$-semistable for all $a\gg 0$, then it satisfies one of the following conditions:
	\begin{enumerate}
		\item\label{l:MSlvl-1} $\HH^{-1}(E)=0$ and $\HH^0(E)$ is an $H$-semistable torsion-free sheaf.
		\item\label{l:MSlvl-2} $\HH^{-1}(E)=0$ and $\HH^0(E)$ is a torsion sheaf.
		\item\label{l:MSlvl-3} $\HH^{-1}(E)$ is a $H$-semistable torsion-free sheaf, and $\HH^0(E)$ is either $0$ or a torsion sheaf supported in dimension~$0$.
	\end{enumerate}
\end{lemma}

\begin{proposition}\label{prop: Q non-neg at large volume limit}
	Let $(H,B)\in \Amp_\R(X)\times\NS_\R(X)$. Fix $\alpha_0,\beta_0\in\R$ such that $\alpha_0>\Phi_{X,H,B}(\beta_0)$. Choose $\delta,\varepsilon>0$ as in Proposition~{\rm\ref{prop: final quadratic form}}. If\, $E\in\Coh^{H,B}(X)$ is $\sigma_{aH,B}$-semistable for all $a\gg 0$, then $Q^{\delta,\varepsilon}_{H,B,\alpha_0,\beta_0}([E])\geq 0$.
\end{proposition}

\begin{proof}
	Let $Q\coloneqq Q^{\delta,\varepsilon}_{H,B,\alpha_0,\beta_0}$. By our hypotheses, $E$ satisfies one of the three conditions in Lemma~\ref{lem: MS large volume limit}. If $E$ satisfies~\eqref{l:MSlvl-1}, then $Q([E])=Q([\HH^0(E)])$, where $\HH^0(E)$ is a $H$-semistable torsion-free sheaf, and the result follows from Proposition~\ref{prop: final quadratic form}\eqref{p:fqf-1}. Similarly, if $E$ satisfies~\eqref{l:MSlvl-2}, then by Proposition~\ref{prop: final quadratic form}\eqref{p:fqf-2}, $Q([E])=Q([\HH^{0}(E)])\geq 0$. Now assume $E$ satisfies~\eqref{l:MSlvl-3}. Then 
	\begin{equation*}
		\ch([E])=-\ch\left(\HH^{-1}(E)\right)+\length\left(\HH^0(E)\right). 
	\end{equation*}
	Hence
	\begin{equation*}
		Q_{BG}([E]) = Q_{BG}\left([\HH^{-1}(E)]\right) - 2\left(-\ch_0\left(\HH^{-1}(E)\right)\right)\length\left(\cH^0(E)\right) \geq Q_{BG}\left(\HH^{-1}(E)\right).
	\end{equation*}
	The same argument applies to $Q_{H,B,\alpha_0,\beta_0,\delta}$. Hence $Q([E])\geq Q([\HH^{-1}(E)])$. The result follows by Proposition~\ref{prop: final quadratic form}\eqref{p:fqf-1}.
\end{proof}

\begin{lemma}\label{lem: Q smaller on JH factors}
  Let $\sigma=(Z,\PP)\in\Stab(X)$ with support property given by a quadratic form $Q$ on $\Knum(X)\otimes\R$. Suppose $E\in\DbX$ is strictly $\sigma$-semistable and satisfies $Q(E)\neq 0$. Let $A_1,\ldots,A_m$ be the Jordan--H\"older 
  factors of\,~$E$. Then $Q(A_i)<Q(E)$ for all $1\leq i\leq m$.
\end{lemma}

\begin{proof}
	It is enough to prove that $Q(A_1)<Q(E)$. Since $E$ is $\sigma$-semistable, $E\in\PP(\phi)$ for some $\phi\in\R$. By definition, $A_1\in\PP(\phi)$, and hence also $E/A_1\in\PP(\phi)$. Therefore, by the support property, $Q(A_1)\geq 0$ and $Q(E/A_1)\geq 0$. Moreover, since $A_1$ and $E/A_1$ have the same phase, there exists a $\lambda\in\bR_{>0}$ such that $Z(A_1)-\lambda Z(E/A_1)=0$. Hence $[A_1]-\lambda[E/A_1]\in\Ker Z\otimes \R$.
	
	Let $Q$ also denote the associated symmetric bilinear form. Now assume $[A_1]-\lambda[E/A_1]\neq 0$ in $\Knum(X)\otimes \bR$. By the support property, $Q$ is negative definite on $\Knum(X)\otimes \bR$; hence
	\begin{equation*}
		0> Q([A_1]-\lambda[E/A_1]) = Q(A_1)-2\lambda Q(A_1,E/A_1)+\lambda^2 Q(E/A_1).
	\end{equation*}
	Moreover, $\lambda>0$ and $Q(A_1), Q(E/A_1)\geq 0$. It follows that $Q(A_1,E/A_1)>0$. Therefore,
	\begin{equation*}
		Q(E) = Q(A_1) + Q(E/A_1)+2Q(A_1,E/A_1)>Q(A_1).
	\end{equation*}
	Otherwise, if $[A_1]=\lambda[E/A_1]$, then $\mu\coloneqq 1/\lambda >0$ and 
	\begin{equation*}
		Q(E) = Q(A_1) + \mu(\mu + 2) Q(A_1).
	\end{equation*}
	If $Q(A_1)=0$, then $Q(E)=0$, so we have a contradiction. Hence $Q(A_1)>0$, so $Q(E)> Q(A_1)$.
\end{proof}

\begin{lemma}[\textit{cf.} {\cite[Lemma 6.1]{bayerShortProofDeformation2019}}]\label{lem: Q neg def and E violates Q implies JH factor violates Q}
	Let $\sigma=(Z,\PP)\in\Stab(X)$, and let $Q$ be a quadratic form which is negative definite on $\Ker Z\otimes \R$. Suppose $E\in\DbX$ is strictly $\sigma$-semistable, and let $A_1,\ldots,A_m$ be the Jordan--H\"older factors of\, $E$. If $Q(E)<0$, then for some $1\leq j\leq m$, we have $Q(A_j)<0$.
\end{lemma}

\begin{proof}
	Assume towards a contradiction that $Q(A_1),Q(E/A_1)\geq 0$. Let $Q$ also denote the associated symmetric bilinear form. By the same argument as in the proof of Lemma~\ref{lem: Q smaller on JH factors}, it follows that $Q(A,E/A_1)>0$. Therefore,
	\begin{equation*}
		Q(E) = Q(A_1) + Q(E/A_1)+2Q(A_1,E/A_1) >0,
	\end{equation*}
	so we have a contradiction. Hence either $Q(A_1)<0$ and we are done, or $Q(E/A_1)<0$. If $Q(E/A_1)<0$, we can repeat the argument with $E/A_1$ and $A_2$ instead of $E$ and $A_1$. There are finitely many Jordan--H\"older factors, so this process terminates. Therefore, $Q(A_j)<0$ for some $1\leq j\leq n$. 
\end{proof}

\begin{lemma}\label{lem: special values of alpha and beta satisfy new support property}
	Let $X$ be a surface. Fix classes $(H,B)\in \Amp_\R(X)\times\NS_\R(X)$, and take $\alpha_0,\beta_0\in\R$ such that ${\alpha_0>\Phi_{X,H,B}(\beta_0)}$. Choose  $\delta,\varepsilon>0$ as in Proposition~{\rm\ref{prop: final quadratic form}}. Fix an $\alpha_1\in\R$ with $\alpha_1>\max\big\{\alpha_0, \frac{1}{2}\big[\big(\beta_0 - \frac{H\ldot B}{H^2}\big)^2-\frac{B^2}{H^2}\big]\big\}$. Assume $E\in\DbX$ is $\sigma_{H,B,\alpha_1,\beta_0}$-semistable. Then it follows that $Q^{\delta,\varepsilon}_{H,B,\alpha_0,\beta_0}([E])\geq 0$. In particular, $\sigma_{H,B,\alpha_1,\beta_0}$ satisfies the support property with respect to $Q^{\delta,\varepsilon}_{H,B,\alpha_0,\beta_0}$.
\end{lemma}

\begin{proof}
	To ease notation, let $Q\coloneqq Q^{\delta,\varepsilon}_{H,B,\alpha_0,\beta_0}$. From \cref{prop: special values of alpha and beta give a geometric stability condition}, we know that for every $\alpha\geq \alpha_1$, the stability conditions $\sigma_{H,B,\alpha,\beta_0}$ and $\sigma_{a_\alpha H,B+bH}$ have the same heart when $b=\beta_0-\frac{H\ldot B}{H^2}$ and $a_\alpha=\sqrt{2\alpha -b^2+\frac{B^2}{H^2}}$.
	
	Moreover, by Lemma~\ref{lem: MS support property for large a}, there exist  $(H',B')\in \Amp_\Q(X)\times\NS_\Q(X)$ such that $\Delta^{C_{H'}}_{H',B'}$ gives the support property for $\sigma_{aH,B+bH}$ if $a\geq a_{\alpha_1}$. We may assume $\Delta^{C_{H'}}_{H',B'}\in\Z$ since it is true after rescaling by some integer. Furthermore, since $E$ is $\sigma_{a_{\alpha_1}H,B+bH}$-semistable, $\Delta^{C_{H'}}_{H',B'}([E])\in\Z_{\geq 0}$.
	
	If $E$ is $\sigma_{H,B,\alpha,\beta_0}$-stable for $\alpha\gg 0$, then by the definition of $a_\alpha$, the vector bundle $E$ is $\sigma_{aH,B}$-stable for $a\gg 0$. It then follows by \cref{prop: Q non-neg at large volume limit} that $Q([E])\geq 0$. Otherwise, there exists some $\alpha_2\geq\alpha_1$ such that $E$ is strictly $\sigma_{H,B,\alpha_2,\beta_0}$-semistable. Let $A_1,\ldots, A_m$ denote the Jordan--H\"older factors of $E$. Then by Lemma~\ref{lem: Q smaller on JH factors}, we have $\Delta^{C_{H'}}_{H',B'}([A_i])<\Delta^{C_{H'}}_{H',B'}([E])$ for all $1\leq i \leq m$. Each $A_i$ is $\sigma_{H,B,\alpha_2,\beta_0}$-stable, so $\Delta^{C_{H'}}_{H',B'}([A_i])\geq 0$ for all $1\leq i\leq m$.
	
	Assume towards a contradiction that $Q([E])<0$. From Lemma~\ref{lem: Q neg def and E violates Q implies JH factor violates Q}, we have $Q([A_j])<0$ for some $1\leq j \leq m$. Let $E_2\coloneqq A_j$. We can now repeat this process for $E_2$ in place of $E_1\coloneqq E$, and so on. This gives a sequence $E_1, E_2, E_3, \ldots, E_k, \ldots$ and $\alpha_1\leq\alpha_2<\alpha_3<\cdots <\alpha_k \cdots$ such that $E_k\in\DbX$ is $\sigma_{H,B,\alpha_k,\beta_0}$-semistable, $Q(E_k)<0$, and $0\leq\Delta^{C_{H'}}_{H',B'}([E_{k+1}])<\Delta^{C_{H'}}_{H',B'}([E_k])$ for all $k\geq 1$. But $\Delta^{C_{H'}}_{H',B'}([E_k])\in\Z_{\geq 0}$ for all $k$, so no such sequence can exist. Hence we have a contradiction.
	
	Finally, by Lemma~\ref{lem: Q negative definite on kernel as alpha varies}, the quadratic form $Q^{\delta,\varepsilon}_{H,B,\alpha_0,\beta_0}$ is negative definite on $\Ker Z_{H,B,\alpha_1,\beta_0}\otimes\R$.
\end{proof}

We are finally ready to apply Corollary~\ref{cor: defo property for ImZ constant}.

\begin{proposition}\label{prop: LP stability conditions satisfy support property}
	Let $X$ be a surface. Let $(H,B)\in\Amp_\R(X)\times\NS_\R(X)$, and let $\alpha_0,\beta_0\in\R$ be such that $\alpha_0>\Phi_{X,H,B}(\beta_0)$. Then $\sigma_{H,B,\alpha,\beta_0}\in\Gstab{N}(X)$ for all $\alpha\geq\alpha_0$.
\end{proposition}

\begin{proof}
	Fix $\alpha_1\in\R$ such that $\alpha_1 >\max\big\{\alpha_0, \frac{1}{2}\big[\big(\beta_0 - \frac{H\ldot B}{H^2}\big)^2-\frac{B^2}{H^2}\big]\big\}$. By \cref{prop: special values of alpha and beta give a geometric stability condition}, it follows that $\sigma_{H,B,\alpha_1,\beta_0}\in\Gstab{N}(X)$. Choose  $\delta,\varepsilon>0$ as in Proposition~\ref{prop: final quadratic form}; then by Lemma~\ref{lem: special values of alpha and beta satisfy new support property}, the stability condition $\sigma_{H,B,\alpha_1,\beta_0}$ satisfies the support property with respect to $Q^{\delta,\varepsilon}_{H,B,\alpha_0,\beta_0}$.
	
	By Lemma~\ref{lem: Q negative definite on kernel as alpha varies}, the quadratic form $Q^{\delta,\varepsilon}_{H,B,\alpha_0,\beta_0}$ is negative definite on $\Ker Z_{H,B,\alpha,\beta_0}$ for all $\alpha\geq\alpha_0$. Moreover, $\Im Z_{H,B,\alpha,\beta_0}$ remains constant as $\alpha$ varies. Therefore, the result follows by Corollary~\ref{cor: defo property for ImZ constant}.
\end{proof}

\begin{proof}[Proof of Theorem~\ref{thm: LP gives precise control over set of geometric stability conditions}]
	By Theorem~\ref{thm: geo stab determined by Z}, for every $\sigma\in\Gstab{}(X)$, there exists a unique $g\in\C$ such that $g^\ast \sigma\in\Gstab{N}(X)$. Hence it is enough to show that $\Gstab{N}(X)\cong\UU$, where
	\begin{equation*}
		\UU=\left\{(H,B,\alpha,\beta)\in\Amp_\R(X)\times\NS_\R(X)\times\R^2 : \alpha>\Phi_{X,H,B}(\beta)\right\}. 
	\end{equation*}
	This follows from  Propositions~\ref{prop: inj local homeo},~\ref{prop: geometric implies twisted LP condition}, and~\ref{prop: LP stability conditions satisfy support property}.
\end{proof}

\subsection{Applications of Theorem~\ref{thm: LP gives precise control over set of geometric stability conditions}}

\begin{theorem}\label{thm: geometric open set connected for surfaces}
	Let $X$ be a surface. Then $\Gstab{}(X)$ is connected.
\end{theorem}

\begin{remark}\label{rem: Le Potier controls boundary of the geometric chamber}
	There are precisely two types of walls of the geometric chamber for K3 surfaces and rational surfaces. They  correspond either to walls of the nef cone (see \cite[Lemma 7.2]{tramelBridgelandStabilityConditions2022} for a construction) or to discontinuities of the Le Potier function. For K3 surfaces, the second case comes from the existence of spherical bundles, which is explained in \cite[Proposition 2.7]{yoshiokaStabilityFourierMukai2009}. For rational surfaces, the discontinuities correspond to exceptional bundles. This is explained for $\Tot(\OO_{\P^2}(-3))$ in \cite[Section~5]{bayerSpaceStabilityConditions2011}, and the arguments generalise to any rational surface.
	
	It seems reasonable to expect this to hold for all surfaces. The description of the geometric chamber given by Theorem~\ref{thm: LP gives precise control over set of geometric stability conditions} also supports this. Indeed, a wall where $\OO_x$ is destabilised corresponds locally to the boundary of $\UU$ being linear. This boundary is exactly where one of the following holds: 
	\begin{enumerate}
	\item $H$ becomes nef and not ample. We expect that this only gives rise to walls in the following cases:
		\begin{itemize}
			\item $H$ is big and nef. Then $H$ induces a contraction of rational curves. This can be used to construct non-geometric stability conditions; see \cite[Lemma 7.2]{tramelBridgelandStabilityConditions2022}.
			\item $H$ is nef and induces a contraction to a curve whose fibres are rational curves. In this case, we expect a wall. For example, let $f\colon S \rightarrow C$ be a $\P^1$-bundle over a curve. We expect the existence of stability conditions on $S$ such that all skyscraper sheaves are strictly semistable and the sheaves are destabilised by
			\begin{equation*}
				\OO_{f^{-1}(x)}\longrightarrow \OO_x \longrightarrow \OO_{f^{-1}(x)}(-1)[1] \longrightarrow \OO_{f^{-1}(x)}[1].
			\end{equation*}
		\end{itemize}
	\item If $\Phi_{X,H,B}$ is discontinuous at $\beta$, then $\Gstab{N}(X)$ locally has a linear boundary. We expect this to give rise to non-geometric stability conditions.
	\item If $\alpha = \Phi_{X,H,B}(\beta)$, then we expect no boundary.
	\end{enumerate}
\end{remark}

\begin{corollary}\label{cor: LP not dcts means geometric walls come from a nef divisor}
	Let $X$ be a surface. If\, $\Phi_{X,H,B}$ has no discontinuities and no linear pieces for any classes $(H,B)\in\Amp_\R(X)\times\NS_\R(X)$, then any wall of\, $\Gstab{}(X)$ where $\OO_x$ is destabilised corresponds to a class $H'\in\NS_\R(X)$ which is nef and not ample.
\end{corollary}

\section{Further questions}\label{section: further questions}

Let $X$ be a variety. There are no examples in the literature where $\Stab(X)$ is known to be disconnected. It would be interesting to investigate the following examples.

\begin{question}
    Let $S$ be a Beauville-type or bielliptic surface. Is $\Stab(S)$ connected?
\end{question}

The surface $S$ has non-finite Albanese morphism, and $\Gstab{}(S)\subset\Stab(S)$ is a connected component by \cref{finite albanese surface quotient has connected component of geos}. If $\Stab(S)$ is connected, the following question would have a negative answer.

\begin{repquestion}{Question FLZ}[\textit{cf.} {\cite[Question 4.11]{fuStabilityManifoldsVarieties2022}}]
	Let $X$ be a variety whose Albanese morphism is not finite. Are there always non-geometric stability conditions on $\DbX$?
\end{repquestion}

\begin{question}
    Suppose $\DbX$ has a strong exceptional collection of vector bundles and a corresponding heart $\AA$ that can be used to construct stability conditions as in \cite[Section~4.2]{macriExamplesSpacesStability2007}. If $\OO_x\in\AA$, then does~$\OO_x$ correspond to a stable quiver representation?
\end{question}


\providecommand{\bysame}{\leavevmode\hbox to3em{\hrulefill}\thinspace}


\begin{thebibliography}{BBDG82+++}

\bibitem[AB13]{arcaraBridgelandStableModuliSpaces2013}
D.~Arcara and A.~Bertram, \emph{Bridgeland-stable moduli spaces for
  $K$-trivial surfaces} (with an appendix by M.~Lieblich), J.~Eur.\ Math.\ Soc.\ (JEMS) \textbf{15} (2013), no.~1, 1--38, \doi{10.4171/JEMS/354}.

\bibitem[BDF07]{bagneraSopraSuperficieAlgebriche1907}
G.~Bagnera and M.~De~Franchis, \emph{Sopra le superficie algebriche che hanno
  le coordinate del punto generico esprimibili con funzioni meromorfe
  quadruplemente periodiche di 2 parametri}, Rend.\ Acc.\ dei Lincei \textbf{16} (1907).
  
\bibitem[Bal11]{balmerSeparabilityTriangulatedCategories2011a}
P.~Balmer, \emph{Separability and triangulated categories}, Adv.\ Math.\ \textbf{226} (2011), no.~5, 4352--4372, \doi{10.1016/j.aim.2010.12.003}.

\bibitem[BCG08]{bauerClassificationSurfacesIsogenous2008}
I.~Bauer, F.~Catanese, and F.~Grunewald, \emph{The Classification of
  Surfaces with $p_g=q=0$ Isogenous to a Product of Curves}, Pure Appl.\ Math.\ Q.~\textbf{4} (2008), no.~2, 547--586,  \doi{10.4310/PAMQ.2008.v4.n2.a10}.
  
\bibitem[Bay18]{bayerWallcrossingImpliesBrillNoether2018}
A.~Bayer, \emph{Wall-crossing implies Brill-Noether: Applications of
  stability conditions on surfaces}, in: \emph{Algebraic Geometry: Salt Lake City 2015}, pp.~3--27, Proc.\ Sympos.\ Pure Math., vol.~97.1, Amer.\ Math.\ Soc., Providence, RI, 2018,  \doi{10.1090/pspum/097.1/01668}. 

\bibitem[Bay19]{bayerShortProofDeformation2019}
\bysame, \emph{A short proof of the deformation property of Bridgeland
  stability conditions}, Math.\ Ann.\ \textbf{375} (2019), no.~3-4, 1597--1613,
  \doi{10.1007/s00208-019-01900-w}.

\bibitem[BM11]{bayerSpaceStabilityConditions2011}
A.~Bayer and E.~Macr\`i, \emph{The space of stability conditions on the local
  projective plane}, Duke Math.~J.\ \textbf{160} (2011), no.~2, 263--322, 
  \doi{10.1215/00127094-1444249}.

\bibitem[BM14]{bayerMMPModuliSheaves2014}
\bysame, \emph{MMP for moduli of sheaves on K3s via wall-crossing: nef
  and movable cones, Lagrangian fibrations}, Invent.\ Math.\ \textbf{198}
  (2014), no.~3, 505--590, \doi{10.1007/s00222-014-0501-8}.

\bibitem[BMS16]{bayerSpaceStabilityConditions2016}
A.~Bayer, E.~Macr\`i, and P.~Stellari, \emph{The space of stability conditions
  on abelian threefolds, and on some Calabi-Yau threefolds}, Invent.\ Math.\
\textbf{206} (2016), no.~3, 869--933, \doi{10.1007/s00222-016-0665-5}.

\bibitem[BMT14]{bayerBridgelandStabilityConditions2013}
A.~Bayer, E.~Macr\`i, and Y.~Toda, \emph{Bridgeland stability conditions on
  threefolds I: Bogomolov-Gieseker type inequalities}, J.~Algebraic
  Geom.\ \textbf{23} (2014), no.~1, 117--163,
  \doi{10.1090/S1056-3911-2013-00617-7}.
  
\bibitem[Bea96]{beauvilleComplexAlgebraicSurfaces1996}
A.~Beauville, \emph{Complex Algebraic Surfaces}, 2nd ed., London
Math.\ Soc.\ Stud.\ Texts, vol.~34, Cambridge Univ.\ Press, Cambridge, 1996, \doi{10.1017/CBO9780511623936}.

\bibitem[BO23]{beckmannEquivariantDerivedCategories2023}
T.~Beckmann and G.~Oberdieck, \emph{On equivariant derived categories},
  Eur.~J.\ Math.\ \textbf{9} (2023), no.~2, Paper No.~36,
  \doi{10.1007/s40879-023-00635-y}.
  
\bibitem[BBDG82]{beilinsonFaisceauxPervers1982}
A.~Beilinson, J.~Bernstein, P.~Deligne, and O.~Gabber, \emph{Faisceaux
pervers}, Ast\'erisque \textbf{100} (1982).

\bibitem[BMSZ17]{bernardaraBridgelandStabilityConditions2017}
M.~Bernardara, E.~Macr\`i, B.~Schmidt, and X.~Zhao, \emph{Bridgeland Stability
  Conditions on Fano Threefolds}, \'Epijournal G\'eom.\ Alg\'ebrique
  \textbf{1} (2017), Art.~2, \doi{10.46298/epiga.2017.volume1.2008}.

\bibitem[Bog79]{bogomolovHolomorphicTensorsVector1979}
F.\,A.~Bogomolov, \emph{Holomorphic tensors and vector bundles on projective
  varieties}, Math.\ USSR-Izv.\ \textbf{13} (1979), no.~3, 499--555,
\doi{10.1070/IM1979v013n03ABEH002076}.

\bibitem[Bri07]{bridgelandStabilityConditionsTriangulated2007}
T.~Bridgeland, \emph{Stability conditions on triangulated categories},
  Ann.\ of Math.~(2) \textbf{166} (2007), no.~2, 317--345,
  \doi{10.4007/annals.2007.166.317}.

\bibitem[Bri08]{bridgelandStabilityConditionsK32008}
\bysame, \emph{Stability conditions on $K3$ surfaces}, Duke Math.~J.\
\textbf{141} (2008), no.~2, 241--291, \doi{10.1215/S0012-7094-08-14122-5}.

\bibitem[BM02]{bridgelandFourierMukaiTransformsK32002}
T.~Bridgeland and A.~Maciocia, \emph{Fourier-Mukai transforms for K3
  and elliptic fibrations}, J.~Algebraic Geom.\ \textbf{11} (2002), no.~4,
  629--657, \doi{10.1090/S1056-3911-02-00317-X}.

\bibitem[CH21]{coskunExistenceSemistableSheaves2021}
I.~Coskun and J.~Huizenga, \emph{Existence of semistable sheaves on
  Hirzebruch surfaces}, Adv.\ Math.\ \textbf{381} (2021),
 Paper No.~107636, \doi{10.1016/j.aim.2021.107636}.

\bibitem[Del97]{deligneActionGroupeTresses1997}
P.~Deligne, \emph{Action du groupe des tresses sur une cat\'egorie}, Invent.\
  Math.\ \textbf{128} (1997), no.~1, 159--175, \doi{10.1007/s002220050138}.

\bibitem[DHL24]{dellFusionequivariantStabilityConditions2024}
H.~Dell, E.~Heng, and A.\,M.~Licata, \emph{Fusion-equivariant stability
  conditions and Morita duality}, preprint \arXiv{2311.06857v2} (2024).

\bibitem[DLP85]{drezetFibresStablesFibres1985}
J.-M.~Drezet and J.~Le~Potier, \emph{Fibr\'es stables et fibr\'es exceptionnels
  sur $\mathbf{P}^2$}, Ann.\ Sci.\ \'Ecole Norm.\ Sup.~(4)
  \textbf{18} (1985), no.~2, 193--243, \doi{10.24033/asens.1489}.

\bibitem[Ela15]{elaginEquivariantTriangulatedCategories2015}
A.~Elagin, \emph{On equivariant triangulated categories}, preprint
  \arXiv{1403.7027v2} (2015).

\bibitem[Fey20]{feyzbakhshMukaiProgramReconstructing2020a}
S.~Feyzbakhsh, \emph{Mukai’s program (reconstructing a K3 surface from a
  curve) via wall-crossing}, J.~reine angew.\ Math.\ \textbf{765} (2020), 101--137,  \doi{10.1515/crelle-2019-0025}.

\bibitem[FLZ22]{fuStabilityManifoldsVarieties2022}
L.~Fu, C.~Li, and X.~Zhao, \emph{Stability manifolds of varieties with finite
  Albanese morphisms}, Trans.\ Amer.\ Math.\ Soc.\ \textbf{375} (2022), no.~8,
  5669--5690, \doi{10.1090/tran/8651}.

\bibitem[GS13]{galkinExceptionalCollectionsLine2013}
S.~Galkin and E.~Shinder, \emph{Exceptional collections of line bundles on the
  Beauville surface}, Adv.\ Math.\ \textbf{244} (2013),
  1033--1050, \doi{10.1016/j.aim.2013.06.007}.

\bibitem[HRS96]{happelTiltingAbelianCategories1996}
D.~Happel, I.~Reiten, and S.\,O.~Smal\o, \emph{Tilting in abelian categories
  and quasitilted algebras}, Mem.\ Amer.\ Math.\ Soc.\ \textbf{120} (1996),
no.~575, \doi{10.1090/memo/0575}.

\bibitem[Huy16]{huybrechtsLecturesK3Surfaces2016}
D.~Huybrechts, \emph{Lectures on K3 Surfaces}, Cambridge Stud.\ Adv.\ Math., vol.~158, Cambridge Univ.\ Press, Cambridge, 2016, \doi{10.1017/CBO9781316594193}.

\bibitem[HL10]{huybrechtsGeometryModuliSpaces2010}
D.~Huybrechts and M.~Lehn, \emph{The Geometry of Moduli Spaces of
  Sheaves}, 2nd ed., Cambridge Math.\ Lib., Cambridge Univ.\
  Press, Cambridge, 2010, \doi{10.1017/CBO9780511711985}.

\bibitem[Kos18]{kosekiStabilityConditionsProduct2018}
N.~Koseki, \emph{Stability conditions on product threefolds of projective
  spaces and Abelian varieties}, Bull.\ Lond.\ Math.\ Soc.~\textbf{50} (2018), no.~2, 229--244, \doi{10.1112/blms.12132}.

\bibitem[Kos20]{kosekiStabilityConditionsThreefolds2020}
\bysame, \emph{Stability conditions on threefolds with nef tangent bundles},
  Adv.\ Math.~\textbf{372} (2020), Paper No.~107316,
  \doi{10.1016/j.aim.2020.107316}.

\bibitem[Kos22]{kosekiStabilityConditionsCalabiYau2022}
\bysame, \emph{Stability conditions on Calabi-Yau double/triple solids},
  Forum Math.\ Sigma \textbf{10} (2022), Paper No.~e63,
  \doi{10.1017/fms.2022.58}.

\bibitem[LR23]{lahozChernDegreeFunctions2022}
M.~Lahoz and A.~Rojas, \emph{Chern degree functions}, Commun.\
  Contemp.\ Math.~\textbf{25} (2023), no.~5, Paper No.~2250007, \doi{10.1142/S0219199722500079}.
  
\bibitem[Lan04]{langerSemistableSheavesPositive2004}
A.~Langer, \emph{Semistable sheaves in positive characteristic}, Ann.\
of Math.~(2) \textbf{159} (2004), no.~1, 251--276, \doi{10.4007/annals.2004.159.251}.

\bibitem[LZ23]{levineBrillNoetherExistenceSemistable2019}
D.~Levine and S.~Zhang, \emph{Brill--Noether and existence of semistable
  sheaves for del Pezzo surfaces}, Ann.\ Inst.\ Fourier \textbf{74} (2024), no.~3, 1189--1227, \doi{10.5802/aif.3619}.

\bibitem[Li17]{liSpaceStabilityConditions2017}
C.~Li, \emph{The space of stability conditions on the projective plane}, Selecta Math.\ (N.S.) \textbf{23} (2017), no.~4, 2927--2945,
  \doi{10.1007/s00029-017-0352-4}.

\bibitem[Li19]{liStabilityConditionsQuintic2019}
\bysame, \emph{On stability conditions for the quintic threefold}, Invent.\
  Math.\ \textbf{218} (2019), no.~1, 301--340, \doi{10.1007/s00222-019-00888-z}.

\bibitem[LZ19]{liBirationalModelsModuli2019}
C.~Li and X.~Zhao, \emph{Birational models of moduli spaces of coherent sheaves
  on the projective plane}, Geom.\ Topol.\ \textbf{23} (2019), no.~1, 347--426,
\doi{10.2140/gt.2019.23.347}.

\bibitem[LR22]{limCharacteristicClassesStability2022}
B.~Lim and F.~Rota, \emph{Characteristic classes and stability conditions for
  projective Kleinian orbisurfaces}, Math.~Z.\ \textbf{300} (2022), no.~1,
  827--849, \doi{10.1007/s00209-021-02805-8}.
  
\bibitem[Liu22]{liuStabilityConditionCalabi2022}
S.~Liu, \emph{Stability condition on Calabi–Yau threefold of complete
  intersection of quadratic and quartic hypersurfaces}, Forum Math.\
  Sigma \textbf{10} (2022), Paper No.~e106, \doi{10.1017/fms.2022.96}.

\bibitem[Mac07a]{macriExamplesSpacesStability2007}
E.~Macr\`i, \emph{Some examples of spaces of stability conditions on derived
  categories}, preprint \arXiv{math/0411613v3} (2007).

\bibitem[Mac07b]{macriStabilityConditionsCurves2007}
\bysame, \emph{Stability conditions on curves}, Math.\ Res.\ Lett.\ \textbf{14}
  (2007), no.~4, 657--672, \doi{10.4310/MRL.2007.v14.n4.a10}.

\bibitem[MMS09]{macriInducingStabilityConditions2009}
E.~Macr\`i, S.~Mehrotra, and P.~Stellari, \emph{Inducing stability conditions},
  J.~Algebraic Geom.\ \textbf{18} (2009), no.~4, 605--649,
  \doi{10.1090/S1056-3911-09-00524-4}.

\bibitem[MS17]{macriLecturesBridgelandStability2017}
  E.~Macr\`i and B.~Schmidt, \emph{Lectures on Bridgeland Stability}, in: \emph{Moduli  of Curves}, pp.~139--211, Lect.\ Notes Unione Mat.\ Ital., vol.~21, Springer, Cham, 2017, \doi{10.1007/978-3-319-59486-6\_5}.

\bibitem[MYY12]{minamideFourierMukaiTransformsWallcrossing2012}
H.~Minamide, S.~Yanagida, and K.~Yoshioka, \emph{Fourier-Mukai transforms
  and the wall-crossing behavior for Bridgeland's stability conditions},
  preprint \arXiv{1106.5217v2} (2012).

\bibitem[Mu21]{muNewModuliSpaces2021}
D.~Mu, \emph{New moduli spaces of one-dimensional sheaves on $\mathbf{P}^3$},
  preprint \arXiv{2002.00442v2} (2021).

\bibitem[Muk78]{mukaiSemihomogeneousVectorBundles1978}
S.~Mukai, \emph{Semi-homogeneous vector bundles on an abelian variety}, J.~Math.\ Kyoto Univ.\ \textbf{18} (1978), no.~2, 239--272, \doi{10.1215/kjm/1250522574}.

\bibitem[Muk84]{mukaiSymplecticStructureModuli1984}
\bysame, \emph{Symplectic structure of the moduli space of sheaves on an
  abelian or K3 surface}, Invent.\ Math.\ \textbf{77} (1984), no.~1,
  101--116, \doi{10.1007/BF01389137}.

\bibitem[OPT22]{oberdieckDonaldsonThomasInvariants2022}
G.~Oberdieck, D.~Piyaratne, and Y.~Toda, \emph{Donaldson--Thomas
invariants of abelian threefolds and Bridgeland stability conditions},
J.~Algebraic Geom.\ \textbf{31} (2022), no.~1, 13--73, \doi{10.1090/jag/788}.

\bibitem[OS01]{oguisoCalabiYauThreefolds2001}
K.~Oguiso and J.~Sakurai, \emph{Calabi--Yau threefolds of quotient type},
  Asian J.~Math.\ \textbf{5} (2001), no.~1, 43--77,
  \doi{10.4310/AJM.2001.v5.n1.a5}.

\bibitem[Oka06]{okadaStabilityManifold2006}
S.~Okada, \emph{Stability manifold of $\mathbf{P}^1$}, J.~Algebraic Geom.\
  \textbf{15} (2006), no.~3, 487--505, \doi{10.1090/S1056-3911-06-00432-2}.  

\bibitem[PPZ23]{perryModuliSpacesStable2023}
A.~Perry, L.~Pertusi, and X.~Zhao, \emph{Moduli spaces of stable objects in
  Enriques categories}, preprint \arXiv{2305.10702} (2023).
  
\bibitem[Pet22]{petkovicPositivityDeterminantLine2022}
M.~Petkovi\'c, \emph{Positivity of Determinant Line Bundles on the Moduli
  Spaces of Sheaves and Stability on the Veronese Double Cone}, PhD thesis, Univ.\ of Utah, 2022, available at \url{https://www.proquest.com/docview/2742640161}.

\bibitem[Piy17]{piyaratneStabilityConditionsBogomolovGieseker2017}
D.~Piyaratne, \emph{Stability conditions, Bogomolov-Gieseker type
  inequalities and Fano 3-folds}, preprint \arXiv{1705.04011} (2017).

\bibitem[Pol07]{polishchukConstantFamiliesTStructures2007}
A.~Polishchuk, \emph{Constant Families of t-Structures on Derived
  Categories of Coherent Sheaves}, Mosc.\ Math.~J.\ \textbf{7} (2007), no.~1,
  109--134, 167, \doi{10.17323/1609-4514-2007-7-1-109-134}.

\bibitem[Rek23]{rekuskiContractibilityGeometricStability2023}
N.~Rekuski, \emph{Contractibility of the geometric stability manifold of a
surface}, preprint \arXiv{2310.10499} (2023).

\bibitem[TX22]{tramelBridgelandStabilityConditions2022}
R.~Tramel and B.~Xia, \emph{Bridgeland stability conditions on surfaces with
  curves of negative self-intersection}, Adv.\ Geom.\ \textbf{22}
  (2022), no.~3, 383--408, \doi{10.1515/advgeom-2022-0009}.

\bibitem[Yos01]{yoshiokaModuliSpacesStable2001}
K.~Yoshioka, \emph{Moduli spaces of stable sheaves on abelian surfaces}, Math.\
  Ann.\ \textbf{321} (2001), no.~4, 817--884, \doi{10.1007/s002080100255}.

\bibitem[Yos09]{yoshiokaStabilityFourierMukai2009}
\bysame, \emph{Stability and the Fourier–Mukai transform II},
  Compos.\ Math.\ \textbf{145} (2009), no.~1, 112--142,
  \doi{10.1112/S0010437X08003758}.

\end{thebibliography}
\end{document}